\documentclass[3p]{elsarticle}
\usepackage[T1]{fontenc}

\usepackage{tikz}
\usetikzlibrary{shapes ,arrows ,positioning}
\usetikzlibrary{automata,calc}
\usepackage{dsfont}
\usepackage{amsfonts}
\usepackage{amsmath}
\usepackage{ragged2e}
\usepackage{wrapfig}
\usepackage{bm}
\usepackage{array}
\usepackage{multicol}
\usepackage{multirow}
\usepackage{pifont}
\usepackage{dsfont}
\usepackage{enumitem}
\usepackage{versions}
\usepackage{amssymb}
\usepackage{amsthm}

\usepackage{babel}
\usepackage[font=small,labelfont=bf]{caption}

\newtheorem{theorem}{Theorem} 
\newtheorem{lemma}[theorem]{Lemma}
\newtheorem{proposition}[theorem]{Proposition}
\newdefinition{remark}{Remark}
\newtheorem{corollary}[theorem]{Corollary}
\excludeversion{brainstorming}

\newtheorem{claim}{Claim}

\newcommand{\EE}{\mathbb E }
\newcommand{\PP}{\mathbb P }
\newcommand{\RR}{\mathbb R }

\newcommand{\limt}{\lim_{t \rightarrow \infty}}

\newcommand{\q}{{\bf q}}

\usepackage{enumitem}

\newlist{thmlist}{enumerate}{1}
\setlist[thmlist]{label=(\roman{thmlisti}), ref=\thethm.(\roman{thmlisti}),noitemsep}

\usepackage{thmtools}

\makeatletter
\def\thickhline{%
  \noalign{\ifnum0=`}\fi\hrule \@height \thickarrayrulewidth \futurelet
   \reserved@a\@xthickhline}
\def\@xthickhline{\ifx\reserved@a\thickhline
               \vskip\doublerulesep
               \vskip-\thickarrayrulewidth
             \fi
      \ifnum0=`{\fi}}
\makeatother

\newlength{\thickarrayrulewidth}
\newcolumntype{?}{!{\vrule width 1pt}}

\usepackage{lipsum} 
\usepackage{fancyhdr}
\begin{document}

\title{Exponential Tail Bounds on Queues: A Confluence of Non-Asymptotic Heavy Traffic and Large Deviations}


\begin{abstract}
In general, obtaining the exact steady-state distribution of queue lengths is not feasible. Therefore, our focus is on establish bounds for the tail probabilities of queue lengths. Specifically, we examine queueing systems under Heavy-Traffic (HT) conditions and provide exponentially decaying bounds for the probability $\PP(\epsilon q > x)$, where $\epsilon$ is the HT parameter denoting how far the load is from the maximum allowed load. Our bounds are not limited to asymptotic cases and are applicable even for finite values of $\epsilon$, and they get sharper as $\epsilon \to 0$. Consequently, we derive non-asymptotic convergence rates for the tail probabilities. Unlike other approaches such as moment bounds based on drift arguments and bounds on Wasserstein distance using Stein's method, our method yields sharper tail bounds. Furthermore, our results offer bounds on the exponential rate of decay of the tail, given by $-\frac{1}{x} \log \PP(\epsilon q > x)$ for any finite value of $x$. These can be interpreted as non-asymptotic versions of Large Deviation (LD) results. To obtain our results, we use an exponential Lyapunov function to bound the moment generating function of queue lengths and apply Markov's inequality.

We demonstrate our approach by presenting tail bounds for: (i) a continuous time Join-the-shortest queue (JSQ) load balancing system, (ii) a discrete time single-server queue and (iii) an $M/M/n$ queue. We not only bridge the gap between classical-HT and LD regimes but also explore the large system HT regimes for JSQ and $M/M/n$ systems. In these regimes, both the system size and the system load increase simultaneously. Our results also close a gap in the existing literature on the limiting distribution of JSQ in the super-NDS (a.k.a. super slowdown) regime. This contribution is of an independent interest. Here, a key ingredient is a more refined characterization of state space collapse for JSQ system, achieved by using an exponential Lyapunov function designed to approximate the $\ell_{\infty}$ norm.
\end{abstract}

% \author[1]{}
% \ead{}
% \affiliation[1]{}

\author[1]{Prakirt Raj Jhunjhunwala}
\ead{prakirt@gatech.edu}
\affiliation[1]{organization={Georgia Institute of Technology},
                addressline={North Avenue},
                city={Atlanta},
                postcode={30332 GA},
                country={USA}
                }

\author[2]{Daniela Hurtado-Lange}
\ead{dahurtadolange@wm.edu}
\affiliation[2]{organization={William and Mary},
                % addressline={Sadler Center, 200 Stadium Dr},
                city={Williamsburg},
                postcode={VA 23185},
                country={USA}
                }

\author[3]{Siva Theja Maguluri}
\ead{siva.theja@gatech.edu}
\affiliation[3]{organization={Georgia Institute of Technology},
                addressline={North Avenue},
                postcode={30332 GA},
                city={Atlanta},
                country={USA}
                }
\maketitle

\section{Introduction}
Queueing models are used to study performance of many systems such as cloud computing, data centers, ride hailing, call centers etc. In general, obtaining the complete distribution of queue lengths in these systems is intractable.
Therefore, a common approach is to study asymptotic regimes. There are several popular regimes such as Heavy Traffic (HT), large scale regime, or Large Deviations (LD). In the HT regime, the system is loaded close to its maximum capacity while keeping the number of servers fixed. In the large-systems regime, the system's load is fixed, but the number of servers is increased to infinity. And in the LD limit one studies the probability of rare events, that is, the tail probability for large thresholds. Recently, the Many-Server Heavy-Traffic (Many-Server-HT) regime has gained more popularity, where the system is loaded to maximum capacity while simultaneously increasing the number of servers. The system's behavior varies greatly depending on how quickly the load increases relative to the number of servers. As such one employs very different analysis techniques to study queueing systems in different regimes.

In the study of HT asymptotics, one typically scales the queue lengths using a parameter that represents the system's load. By denoting the load as $1-\epsilon$, the HT limit is achieved when $\epsilon$ approaches zero. Most of the literature focuses on systems that
satisfy the so-called Complete Resource Pooling (CRP) condition and behave like a single-server queue in the limit. For such systems, it is well-known that the scaled queue length follows an exponential distribution in the HT limit, which gives the tail probabilities of the limiting system. However, the rate of convergence of the tail probabilities (of the pre-limit system) to the corresponding HT value remains unknown.

Most real world systems involve Service Level Agreements (SLA), where customers are promised a specific level of service, including the maximum delay they can expect.
Motivated by this, in this paper, we focus on establishing sharp bounds on the tail probabilities of scaled queue length of the pre-limit system, i.e., for $\epsilon>0$. In particular, we get non-asymptotic bounds of the form
$$\PP(\epsilon q>x)\leq \kappa(\epsilon,x) e^{-\theta(\epsilon) x},$$
where $q$ represents the total queue length in steady state. Here, $\theta(\epsilon)$ gives the decay rate of the tail probability of the pre-limit system, and $\theta(\epsilon)$ converges to the correct HT value as $\epsilon \rightarrow 0$. Recent results show the rate of convergence to HT in terms of the mean, moments, or Wasserstein's distance (for references on each of these, see Section \ref{sec: intro-related-work}). These methods focus on the entire distribution of the queue lengths and drown the tail. For example, consider the second moment, and suppose $\epsilon q$ converges in distribution to the random variable $\Upsilon$. Then, from existing results, one obtains that $|\EE[\epsilon^2 q^2] - \EE[\Upsilon^2] |$ is $O(\epsilon)$, which gives a valid bound. From these results, one can obtain bounds in terms of tail probability of the form $|\PP(\epsilon q>x) - \PP(\Upsilon >x)| \leq O(\epsilon)$. However, these are not very informative as the tail probability itself can be much smaller than $O(\epsilon)$. Therefore, the rate of convergence of tail probabilities cannot be obtained using the existing methodologies. In this work, we correctly characterize $\theta(\epsilon)$ to obtain the rate of convergence of tail probability to the corresponding HT value.  Our results are non-asymptotic in the sense that they are valid whenever $\epsilon$ is small, and not just when $\epsilon\rightarrow 0$. Also, our results are precise when $\epsilon$ gets closer to 0, recovering the HT results.

Our work bridges the gap between the LD,  HT, and Many-Server-HT regimes.
When one studies the LD regime, the goal is to find the exponential rate at which the tail probability decays, which is precisely given by $\theta(\epsilon)$. As such, our tail bounds can be used to recover the non-asymptotic LD results. 
Thus, our tail bounds are at a confluence of non-asymptotic HT  and non-asymptotic LD. This extends the understanding of tail behavior beyond the classical HT regime. To the best of our knowledge, such comprehensive LD results have not been previously reported in the existing literature. 

\subsection{Main contribution}\label{sec:intro-main-contribution}

We illustrate our methodology by providing results for three 
well-studied systems, viz, a load-balancing system under Join the Shortest Queue (hereafter referred to as the JSQ system), a discrete-time Single-Server Queue (SSQ), and a multi-server system with a single queue ($M/M/n$ queue). Our contributions for each of these systems are mentioned below.

\subsubsection{JSQ system}
We consider a continuous-time system with $n$ servers, each of them with its own queue.  Jobs arrive according to a Poisson process with rate $\lambda_n$, and are routed to the server with the shortest queue (breaking ties at random). Further, the service times are exponentially distributed with rate $\mu$. In the context of the JSQ system, we consider the Many-Server-HT regime, where the system size $n$ grows to infinity while the HT parameter $\epsilon_n$ approaches zero. Specifically, we consider $\epsilon_n = n^{-\alpha}$ with $\alpha>1$ constant, and take the limit as $n\rightarrow \infty$. For the JSQ system, we have the following two contributions.

\begin{itemize}
    \item \textbf{Tail probability: } In Theorem \ref{thm: jsq_cont_tail_bound}, we show that the tail probability of the steady-state scaled total queue length satisfies
\begin{align*}
   \frac{1}{1-\epsilon_n} e^{-\theta_n x} \leq \PP\Big(\epsilon_n \sum_{i=1}^n q_i > x\Big) \leq 2ex\Big(1+ \kappa_2 n \epsilon_n \log \frac{1}{\epsilon_n}\Big) e^{-\theta_n x},
\end{align*}
where $\theta_n := \frac{1}{\epsilon_n }\log \frac{1}{1-\epsilon_n}$.
The lower bound is valid for all values of $n$ and $\epsilon_n$, while the upper bound holds when the term $n \epsilon_n \log \frac{1}{\epsilon_n}$ is sufficiently small. The upper bound leverages the State Space Collapse (SSC) property of the JSQ system, where the $n$-dimensional state vector collapses to a one-dimensional subspace, so the JSQ system behaves like an SSQ. The SSC property holds only when the load is sufficiently large, and so we need $n \epsilon_n \log \frac{1}{\epsilon_n}$ to be sufficiently small. This is satisfied for sufficiently large $n$, when $\epsilon_n=n^{-\alpha}$ with $\alpha>1$.

\begin{figure}
    \centering
    \includegraphics[scale=0.5]{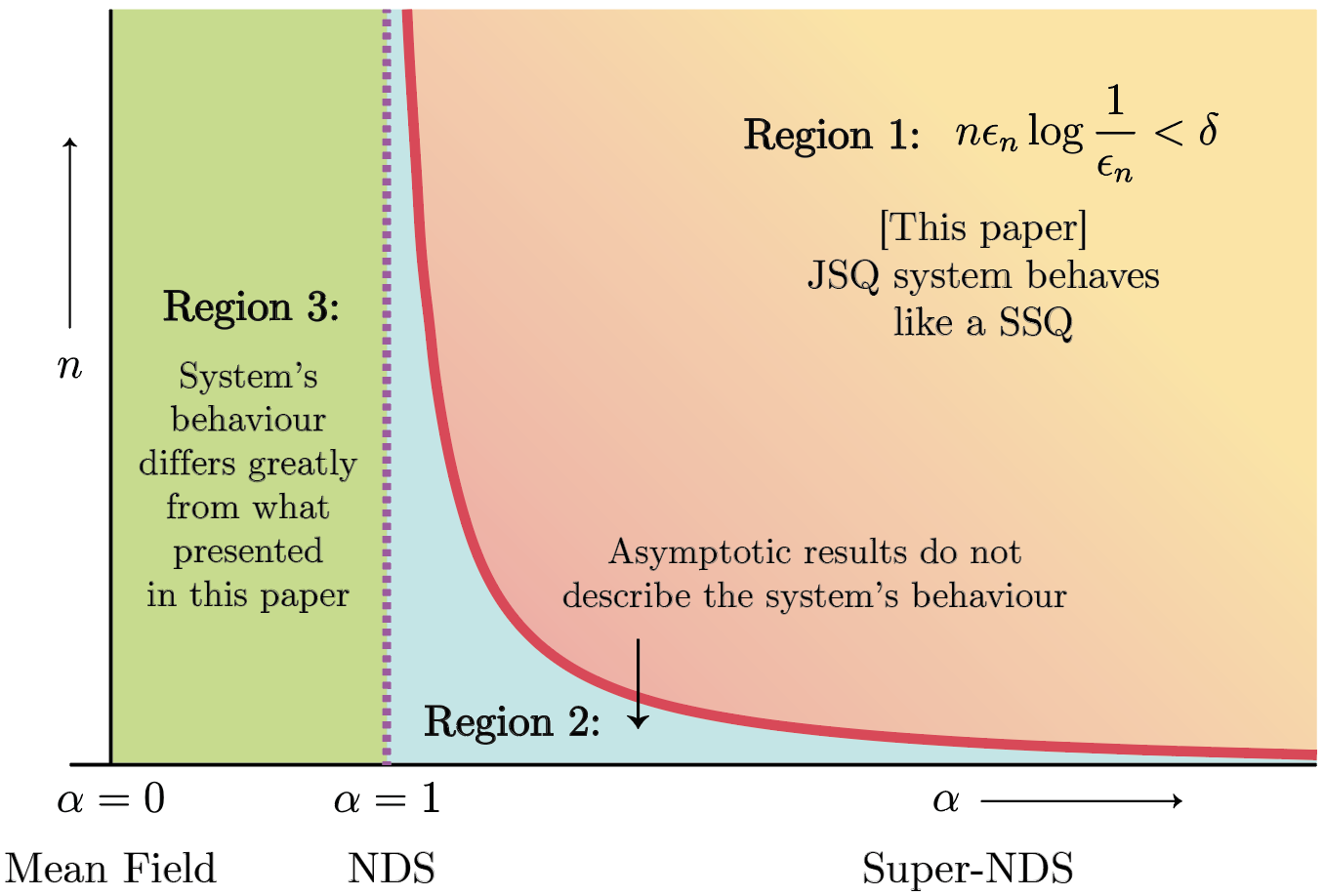}
    \caption{Visual representation of the condition $n \epsilon_n \log \frac{1}{\epsilon_n}< \delta$, where $\delta$ is a small constant. }
    \label{fig:JSQ-contributions}
\end{figure}
    
Figure \ref{fig:JSQ-contributions} visually illustrates the condition discussed above. Region 1 represents values of $n$ and $\alpha$ where $n \epsilon_n \log \frac{1}{\epsilon_n}$ is small, and our tail bounds for the JSQ system hold within this region. As $\alpha$ approaches 1, satisfying the condition $n \epsilon_n \log \frac{1}{\epsilon_n}<\delta$ becomes more challenging, leading to a regime change consistent with existing literature \cite{GupWal_NDS_JSQ}. In Region 2, we believe that the HT approximation fails to accurately describe the system dynamics. Finally, Region 3 corresponds to the range $\alpha\in[0,1]$, where the results and analysis techniques significantly differ from those presented in this paper.

We also obtain an LD result for the JSQ system, i.e., we have
$$\lim_{x\rightarrow\infty} \frac{1}{x} \log \PP\Big(\epsilon_n \sum_{i=1}^n q_i > x\Big) = -\theta_n.$$
Our LD result concerning the scaled total queue length is asymptotically precise as $\epsilon_n$ approaches zero in Many-Server-HT settings (provided that $\alpha>1$).

Importantly, our work provides an exact characterization of the decay rate of the tail probabilities for the pre-limit system. This is possible because we establish both upper and lower bounds on the tail probability. As a result, our research represents a significant advancement compared to existing literature.
Thus, for JSQ, our result not only bridges the gap between HT and LD, but also connects these to (some) many server regimes. 

\item \textbf{Limiting distribution: } In Theorem \ref{thm: jsq_cont_limiting_dis}, we show that as $n \rightarrow \infty$, all the (scaled) queue lengths become identical and are exponentially distributed with mean $1$, i.e., in the Many-Server-HT (i.e., $\epsilon_n = n^{-\alpha}$ with $\alpha>1$), we prove that
\[n\epsilon_n \q \stackrel{d}{\rightarrow} \Upsilon \mathbf{1},\]
where $\q$ is the queue length vector in steady state, $\Upsilon$ is an exponentially distributed random variable, and $\mathbf{1}$ is a vector of all ones.
Prior work \cite{HurMag_2022_alpha_continuous} establishes such a result only for $\alpha>2$ and leaves the regime $1<\alpha\leq 2$ open. 
This result thus fills the gap in the literature and in conjunction with all the prior work \cite{HurMag_2022_alpha_continuous}, and
completes our understanding of JSQ in all of the Many-Server-HT regimes.
A key step to this end is in establishing SSC for the JSQ system especially in the regimes $\alpha \in (1, 2]$. Our approach in establishing SSC result for the JSQ system uses a novel Lyapunov function, and differs from the approach based on the drift arguments of \cite{hajek_drift} as used in prior work, \cite{atilla, MagSri_SSY16_Switch, hurtado2020transform}. Intuitively, instead of working with the $\ell_2$ distance of the queue length vector from the one dimensional subspace, our key idea is that one should work with the $\ell_{\infty}$ distance which gives sharper bounds. However, since the $\ell_{\infty}$ distance is not easily amenable to drift arguments, we work with a Lyapunov function that can be interpreted as a smooth approximation of the $\ell_{\infty}$ distance.
\end{itemize} 

\subsubsection{Discrete time Single-Server Queue}
% 	Single server queue

We consider a discrete-time SSQ with general arrival and service distributions, characterized by mean arrival rate $\lambda_\epsilon$ and mean service rate $\mu_\epsilon$, respectively. We introduce the parameter $\epsilon$ to represent the HT condition, where $\lambda_\epsilon = \mu_\epsilon(1-\epsilon)$. Additionally, we define $\sigma_\epsilon^2$ as the sum of the variances of the arrival and service processes.

In Theorem \ref{thm: ssq_tail_bound}, we obtain the following tail bound and LD result for the pre-limit system:
\begin{align*}
\mathbb P(\epsilon q > x) \leq e\theta_\epsilon x  e^{-\theta_\epsilon (1-\kappa_1 \epsilon)x}, && \lim_{x\rightarrow \infty} \frac{1}{x} \log \PP(\epsilon q > x) \leq -\theta_\epsilon (1- \kappa_1 \epsilon).
\end{align*}
Here, $\theta_\epsilon = \frac{2\mu_\epsilon}{\sigma_\epsilon^2}$ and $\kappa_1$ are defined in terms of the system parameters and moments of the arrival and service processes. Also, the constant $\kappa_1$ depend on the thrid moments of arrival and service. This result exhibit the influence of a general arrival and service distribution on the decay rate of the tail probability under non-asymptotic HT conditions and also provide a non-asymptotic LD result for the steady-state queue length.

% 	MMn
\subsubsection{$M/M/n$ system}

An $M/M/n$ queue is a multi-server system with a single queue, where the inter-arrival and service times (of each server) are exponentially distributed, with rates $\lambda_n$ and $\mu$, respectively. To parameterize the Many-Server-HT regime, we use $\lambda_n = n\mu(1-\epsilon_n)$. In this case, we consider two separate quantities: (i) the number of waiting customers, denoted by $w_n$, and (ii) the number of idle servers, denoted by $r_n$; and provide tail probabilities for both quantities under proper scaling. 

Our approach provides results in three different Many-Server-HT regimes, viz. Sub-Halfin-Whitt (Sub-HW) regime when $\alpha\in\left(0,\frac{1}{2}\right)$, Halfin-Whitt (HW) regime when $\alpha=\frac{1}{2}$, and Super-Halfin-Whitt (Super-HW) regime when $\alpha\in\left(\frac{1}{2},\infty\right)$. A brief overview of our results for the $M/M/n$ system is presented in Table \ref{tbl: mmn} below. More details are provided in Section \ref{sec: mmn}.

\renewcommand{\arraystretch}{1.6}
\begin{table}[ht]
\begin{center}
\begin{tabular}{c?c?c?c}
% \thickhline
&  Sub-HW & Halfin-Whitt \rule{0pt}{12pt}& Super-HW \\ \thickhline
$\epsilon_n = n^{-\alpha}$ with  & $\alpha\in \big(0, \frac{1}{2}\big)$ & $\alpha=\frac{1}{2}$&  $\alpha> \frac{1}{2}$    \\ \hline
$\PP(\epsilon_n w_n>x| w_n>0)$  & \multicolumn{3}{c}{$O\big( e^{-\theta_n x}\big)$ \rule{0pt}{12pt}}  \\ \hline
$\PP(\eta_n  \tilde r_n>x| r_n>0)$ &\multicolumn{3}{c}{$O\big(e^{-\frac{1}{2}x^2}\big)$ }  \\ \hline
$\PP(w_n>0)$  & $O\big(n^{\alpha-\frac{1}{2}} e^{-n\epsilon_n}\big)$                                            & $\rightarrow p<1$& $\rightarrow 1$   \\ \hline
$\PP(r_n>0)$  &$\rightarrow 1$ & $\rightarrow 1-p<1$  & $O(n^{\frac{1}{2}-\alpha}) $
\\
% \thickhline
\end{tabular}
\caption{A brief overview of our results for the $M/M/n$ queue. Here, $\tilde r_n := r_n -n\epsilon_n$. The scaling parameter for $w_n$ is the HT parameter $\epsilon_n$, while the scaling parameter for $\tilde r_n$ is $\eta_n:= \frac{1}{\sqrt{n(1-\epsilon_n)}}$. Also, $p$ is a constant given in Theorem \ref{thm: mmn_idleservers_HW}.}
\label{tbl: mmn}
\vspace{-2em}
\end{center}
\end{table}

The limiting distribution of the scaled number of waiting customers and the scaled number of idle servers were established in \cite{braverman2017stein} so we skip mentioning that here. However, one can easily verify that our results recover the limiting distribution as $n\rightarrow \infty$. 
Note that the probability of there being any idle server, i.e., $\PP(r_n>0)$ goes to zero in the Super-HW regime. As such, the tail bound on $r_n$ is useful only in Sub-HW and HW regime. Similarly, the tail bound on $w_n$ is valid only in HW and Super-HW regimes.

\subsection{Key aspects of our approach}\label{sec:intro-methodology}

 The key idea of this paper is to leverage the existence of the Moment Generating Function (MGF) of the scaled queue length process in the vicinity of zero to obtain the decay rate of the tail probability. Let's consider a non-negative random variable $X$ and a constant $a > 0$ such that $\mathbb{E}[e^{aX}] \leq \kappa$. By applying Markov's inequality, we can establish that $\mathbb{P}(X > x) \leq \kappa e^{-ax}$. This implies that the tail of $X$ decays exponentially with rate $a$. Thus, determining the appropriate decay rate of the tail probability reduces to
 finding the largest $a$ such that $\mathbb{E}[e^{aX}]$ is bounded by a constant.

We adopt this approach to derive tail bounds for the scaled steady-state queue-length process. We establish an upper bound on the MGF of the scaled queue lengths, denoted as $\mathbb{E}[e^{\theta \epsilon q}]$, for a large range of $\theta$. To determine the correct range of $\theta$, we approximate the behavior of the actual system by considering its HT counterpart. By doing so, we also obtain the rate of convergence of the tail probabilities to their corresponding values in the HT regime.

The use of exponential Lyapunov functions to study queue-length behavior was introduced as a transform method to obtain HT results in \cite{hurtado2020transform}. To get tight pre-limit results, we use the same exponential test function, but we bound the error terms in a more refined manner. 
As such, our technique is divided into three steps: (i) derive the MGF (or an upper bound) of the scaled queue length, $\mathbb{E}[e^{\theta \epsilon q}]$, for a large range of values of $\theta$ , (ii) approximate the MGF with its HT counterpart, and (iii) use Markov's inequality and optimize over the values of $\theta$. In the process, we encounter three  multiplicative terms that are required to characterize the tail probability.

The first term is called the SSC violation. As mentioned before, JSQ system satisfies SSC in HT. However, in non-asymptotic HT conditions, SSC is not fully satisfied. As such, the term SSC violation  accounts for the level to which SSC is violated in a pre-limit system. The second one is the pre-limit tail. When the HT parameter $\epsilon$ is greater than zero, the decay rate of the tail probability deviates slightly from its HT value, and we refers to this correct decay rate as the pre-limit tail. The third term is referred to as the pre-exponent error. Our approach uses Markov's inequality to compute bounds tail probabilities from the MGF. However, Markov's inequality incurs a cost while obtaining the tail bound, which is captured by the pre-exponent error.

As a consequence, we easily obtain a LD result after characterizing a bound on the tail probability. Mathematically, as $\EE[e^{aX}] \leq \kappa$ implies $\PP(X>x) \leq \kappa e^{-ax}$, we also get $\lim_{x\rightarrow \infty} \frac{1}{x} \log \PP(X>x) \leq -a$. While characterizing the LD result, the pre-limit error term plays a key role. The other two error terms, namely the pre-exponent error and the SSC error, vanish as $x \rightarrow \infty$. Therefore, the pre-limit error term is essential for accurately characterizing the LD behavior of the system.

\subsection{Related work}\label{sec: intro-related-work}
Most of the literature studying the HT behavior of various queues uses a methodology frequently called `diffusion limits.' A non-exhaustive sample of articles using this method is \cite{kingman1962_brownian,harrison1998heavy,Williams_CRP,harrison2004dynamic,gamarnik2006validity}. Under this approach, the scaled system is shown to converge to a Reflected Brownian Motion (RBM) in HT, and the steady-state behavior of this RBM is studied. The final step is the so-called interchange-of-limit proof which is usually remarkably challenging.

More recently, there has been an increased number of papers using alternative methods that do not require the interchange-of-limits step. These are Stein's method \cite{braverman_steins_2017,braverman2017stein2,Lei_steinHT_SIGM17,Walton_SteinHT}, the BAR approach \cite{braverman_BAR,braverman_BAR_23,DaiGlynnXu_BAR_23}, and the drift method \cite{atilla,hurtado2020transform,JhunZubMag_2022_JSQ-A}. All these methods have something in common with the work we present in this paper and, at the same time, are substantially different.

In the first approach, i.e., the Stein's method, 
one derives bounds for the Wasserstein distance \cite{vallender1974calculation} between the pre-limit system and the limiting distribution, enabling one to obtain the convergence rates. Similar to this, we also establish bounds that depict the similarity between the pre-limit system and the limiting distribution. However, our focus lies in directly characterizing the tail probability. In the second method, i.e., the BAR approach, the goal is to establish a certain equation, called Basic Adjoint Relationship (BAR), in terms of the MGF of the steady state distribution of the limiting distribution. In contrast, we direct work with the MGF of the pre-limit system.

In the third method, i.e., the drift method, the main idea is to use steady-state conditions and carefully chosen test functions to compute bounds that are tight in heavy traffic. Within this method, the use of exponential test functions contributed to the development of the transform method \cite{hurtado2020transform,JhunZubMag_2022_JSQ-A}. Our work in this paper is inspired by the transform method in the sense that using exponential test functions and careful manipulation of the queue-length dynamics are essential to obtain the results. In contrast to the transform method, which only focuses on the limiting distribution, we carefully compute the error terms also to obtain the rate of convergence of the tail probabilities. To the best of our knowledge, the rate of convergence to heavy traffic in terms of tail probabilities has been not been known before.

\renewcommand{\arraystretch}{1.4}
\begin{table}
\centering
\begin{tabular}{l?l?l}
 $\alpha$     & Regime     & Reference \\ \thickhline
$ \alpha\downarrow 0$ & Mean Field &  \cite{mukherjee2018universality,mitzenmacher_po2,mitzenmacher_po2_2,stolyar_JIQ,Lei_steinHT_SIGM17}    \\\hline
$\alpha\in \big(0,\frac{1}{2}\big)$& Sub-Halfin-Whitt &  \cite{liu2018simple,varma2022powerofd,bhamidi2022near}    \\\hline
 $\alpha=\frac{1}{2}$& Halfin-Whitt &  \cite{BanMukh_JSQ_tail_asymptotics,BanMukh_JSQ_sensitivity,Braverman2020_jsq,Gamarnik_JSQ,HalfinWhitt_Regime}    \\\hline
$\alpha\in \big(\frac{1}{2},1\big)$ & Super-Halfin-Whitt &  \cite{liu2021universal,zhao2023manyserver} \\\hline
$\alpha=1$  &   Non-Degenerate Slowdown     &  \cite{GupWal_NDS_JSQ}    \\\hline
$\alpha \in (1,2]$ &   \multirow{2}{*}{Super-Slowdown}& This Work\\\cline{1-1} \cline{3-3}
     $\alpha>2 $ &           &   \cite{HurMag_2022_alpha_continuous}   \\\hline 
$\alpha=\infty$& Classic-HT &    \cite{atilla,foschini1978basic,Williams_state_space,Williams_CRP,hurtado2020transform,harrison1999heavy,harrison1998heavy}  \\
\end{tabular}
\caption{Literature review of asymptotic regimes for Load balancing system.}
\label{tab:my_label}
\end{table}

% Many-Server-HT work
The literature on Many-Server-HT asymptotics is sub-divided in multiple categories depending on how fast the load increases with respect to the number of servers. Using the parameterization $\epsilon_n=n^{-\alpha}$ with $\alpha>0$ introduced above, different regimes are obtained depending on the value of $\alpha$. The literature on Many-Server-HT with $\alpha\in(0,1]$ is vast and uses different analysis techniques. We direct the readers to \cite{zhao2023manyserver} and references therein for more details on Many-Server-HT with $\alpha\in(0,1]$.
Closest to our work, in \cite{HurMag_2022_alpha_continuous}, it is proved that for $\alpha>2$ the scaled total queue length is exponentially distribution in limit as $n\rightarrow \infty$. In our work, we close the gap between the result in \cite{HurMag_2022_alpha_continuous} and the Non-Degenerate Slowdown (NDS) \cite{gupta2019load} regime $(\alpha=1)$, and we obtain the HT behavior for all $\alpha>1$.

% LD
Large Deviations (LD) is a popular regime when one needs to study the performance of a control policy for routing or scheduling jobs. A comprehensive report on the application of large deviation theory to queueing problems can be found in \cite{ganesh2004big}. In \cite{glynn1994logarithmic}, it was proved that the if the sequence of arrivals and services follow a certain Large Deviation Principle (LDP), then, one can obtain the decay rate of the tail probability of the associated queue length. This argument was used in \cite{puhalskii1995large,duffield1995large} to prove a LDP for the queue length process of a single server process, with further generalization in \cite{dupuis1995large}. Recently, there has been quite some work on establishing a LDP for the JSQ system, see \cite{ridder2005large,foley2001join,puhalskii2007large} and the references therein. In contrast to the existing literature, our work provides a LD result as a simple closed form expression in a non-asymptotic LD regime. Notably, we also provide an upper bound on the pre-exponent term, which is typically absent in previous works. Moreover, our work considers the JSQ system in a Many-Server regime, and incorporates the crucial phenomenon of State Space Collapse. Such a phenomenon was considered in \cite{foley2001join}, but only for a JSQ system of size $2$.

In addition to proving a LDP, an intriguing research area involves the use of Lyapunov functions to develop policies that minimize the probability of queue overflow \cite{LinVen_09_LD,venkataramanan2013queue,BanKanQian_18_LD-MW}. The focus in \cite{LinVen_09_LD,venkataramanan2013queue} is on minimizing the decay rate of tail probability, and non-asymptotic large deviation results are not provided. Another related research field examines queues with Gaussian input, which are particularly challenging to analyze in most scenarios \cite{ZubMan_21_LD_GaussianInput,Man_07_LD-book}. These are related to ours in the sense that all compute tight bounds that characterize the probability of rare events. However, the methodologies are different, and they focus on the LD regime only. In this paper, we obtain LD results that are closely connected to the tail probabilities of the queue lengths in HT and Many-Server-HT.

\section{Join-the-shortest queue system} \label{sec: jsq_cont}

% In this section, we provide the tail bounds on the continuous time Join-the-Shortest Queue (JSQ) system. We start with describing the model.

\subsection{Model}
\label{sec: jsqcont_model}

% In this section, we present the model for the Join-the-Shortest Queue (JSQ) system.

We consider a continuous-time queueing system consisting of $n$ SSQs in parallel, each serving jobs according to first-come-first-serve. At any time $t$, let $\q(t)$ denote the queue length vector, where $q_i(t)$ is the queue length of $i^{th}$ queue. For the ease of notation, we use $\overline q(t)$ to denote the total queue length at time $t$, i.e., $\overline q(t) = \sum_{i=1}^n q_i(t)$.
Jobs arrive to the system according to a Poisson process with rate  $\lambda_n$, and service times are exponentially distributed with rate $\mu$.

 When a job arrives, it is dispatched according to JSQ, that is, the job is sent to the queue with index
\begin{equation*}
    i^*(t) \in \arg\min_{i} q_i(t),
\end{equation*}
breaking ties uniformly at random. Under the JSQ policy, the queue-length process $\{\q(t)\}_{t\geq 0}$ is a Continuous-Time Markov Chain (CTMC). Further, it is well-known that the queue-length process is stable (positive recurrent) if the arrival rate is strictly smaller than the total service rate $(\lambda_n < n\mu)$. In this work, we assume that the system satisfies the stability condition. Under this assumption, the steady-state distribution of the queue-length process $\{\q(t)\}_{t=0}^\infty$ exists, and we denote by $\pi_n$. We use $\q$ to denote the steady-state queue length vector, that is, $\q$ follows the distribution $\pi_n$, and $\overline q := \sum_{i=1}^n q_i$.

The system load is $\rho_n := \frac{\lambda_n}{n\mu}$, and we define $\epsilon_n = 1 - \rho_n$. Then, the system approaches HT as $\epsilon_n \rightarrow 0$. We consider the JSQ system in HT and Many-Server-HT. In (classical) HT, the system size $n$ is a constant, and the HT parameter $\epsilon_n$ does not depend on $n$. Then, we drop the subscript $n$ and take the limit $\epsilon\rightarrow 0$. In the Many-Server-HT, the load and the number of servers increase together. Then, we consider $\epsilon_n = n^{-\alpha}$ and take the limit as $n\to\infty$. In this work, in case of Many-Server-HT we only consider the case when $\alpha>1$, which is also known as super-slowdown regime.
In both regimes, we aim to provide a tail bound on the total queue length, i.e., a bound on $\PP\big( \epsilon_n \overline{q} >x \big) $.

% Let $\rho_n$ denote the load of the system, i.e., $\rho_n = \frac{\lambda_n}{n\mu}$, and suppose $\epsilon_n = 1- \rho_n$. In this work, we consider the JSQ system operating in two well-known regime. First one in the classic heavy traffic regime. In this regime, system size $n$ is a constant and the heavy traffic parameter does not depend on $n$, we drop the subscript $n$ and denote the heavy traffic parameter using $\epsilon$. In this regime, we look at the quantity $\epsilon \overline{q}$ as $\epsilon\rightarrow 0$.

% Second regime is the many server heavy traffic regime. In this regime, the size of the system goes to infinity, and we assume that the heavy traffic parameter $\epsilon_n$ goes to $0$ as $n$ grows large according to the relation $\epsilon_n = n^{-\alpha}$. Our results for the JSQ system in both the regimes are given in the following section.

\subsection{Results for JSQ system}

It is well known that the HT distribution of the scaled steady-state total queue length $\epsilon \overline{q}$ converges to an exponential random variable as $\epsilon \rightarrow 0$ \cite{hurtado2020transform}. Further, as shown in \cite{HurMag_2022_alpha_continuous}, the result extends to Many-Server-HT, where $\epsilon_n \overline{q}$ converges to an exponential random variable in distribution if $\alpha>2$. 

In Theorem \ref{thm: jsq_cont_limiting_dis}, we complete the result by demonstrating that $\epsilon_n \overline{q}$ converges in distribution to an exponential random variable $\alpha>1$. Our result enhances the understanding of the behavior of the scaled queue length in Many-Server-HT under a broader range of traffic conditions, encompassing values of $\alpha \in (1,2]$.
% \subsection{Many server heavy traffic regime: Results for JSQ system}
% Similar to that for SSQ, we establish an upper bound on the tail of the distribution of the total queue length for the JSQ system. For the JSQ system, we assume that the arrival rate and the service rate satisfy the relation $\lambda_n = n\mu(1-n^{-\alpha})$, where $\alpha >1$. Under this relation,  as the system size $n$ grows, the system approaches the \textit{many-server heavy traffic regime}. For $\alpha>2$, it was shown in  that the distribution of the steady state scaled total queue length, i.e., $n^{-\alpha} \sum_{i=1}^n q_i$, converges to the distribution of an exponential random variable with mean $1$. In this work, we show that the result is valid even for $\alpha \in (1,2]$.

\begin{theorem} \label{thm: jsq_cont_limiting_dis}
    Consider the JSQ system as presented in Section \ref{sec: jsqcont_model}. Suppose the system satisfies the condition $\lambda_n = n\mu(1-n^{-\alpha})$, i.e., $\epsilon_n = n^{-\alpha}$, where $\alpha >1$. Then, for any $\theta <1$, we have
    \begin{align*}
        \lim_{n\rightarrow \infty} \EE[e^{\theta n^{-\alpha} \sum_{i=1}^n q_i} ] = \frac{1}{1-\theta}, && n^{1-\alpha} \q \stackrel{d}{\rightarrow} \Upsilon \mathbf 1, \ \text{ as } \ n\rightarrow \infty,
    \end{align*}
    where $\Upsilon$ is an exponential random variable with mean $1$.
\end{theorem}

% The limiting distribution for the JSQ system in \textit{many-server heavy traffic regime}

Theorem \ref{thm: jsq_cont_limiting_dis} provides the limiting distribution of the steady-state scaled queue length vector as $n\rightarrow \infty$ for all $\alpha>1$. The proof of Theorem \ref{thm: jsq_cont_limiting_dis} relies on the fact that the JSQ system satisfies SSC, where the $n$-dimensional state vector of the system collapses to a one-dimensional subspace. More precisely, as $n \rightarrow \infty$, we have that $n^{1-\alpha} q_i \approx n^{-\alpha} \sum_{j=1}^n q_j$ for all $i \in \{1,2,\dots,n\}$. The correct characterization of SSC in the Many-Server-HT regime is crucial in proving Theorem \ref{thm: jsq_cont_limiting_dis} and in completing the result of \cite{HurMag_2022_alpha_continuous} for all $\alpha>1$. More details on SSC for the JSQ system in Many-Server-HT are provided in Proposition \ref{prop: jsq_cont_ssc} in Section \ref{sec: jsq_ssc}. A proof sketch for Theorem \ref{thm: jsq_cont_limiting_dis} is provided in Section \ref{sec: jsq_cont_technical_details} and the details are provided in Appendix \ref{app: jsq_cont}.

% In the many server heavy traffic regime with $\alpha>1$, the JSQ system exhibits an important property known as State Space Collapse (SSC), where the $n$-dimensional state vector of the system collapses to a one-dimensional subspace. 

Next, we provide the tail bound of the scaled steady-state total queue length for the JSQ system. In Theorem \ref{thm: jsq_cont_tail_bound}, we provide a bound on the tail probability $\mathbb P (\epsilon_n \overline{q} > x)$ when the system size is finite.

\begin{theorem} \label{thm: jsq_cont_tail_bound} Consider the JSQ system as presented in Section \ref{sec: jsqcont_model}. Suppose the system satisfies the condition $\lambda_n = n\mu(1-\epsilon_n)$, where $\epsilon_n$ is small enough such that $n \epsilon_n\log \big(\frac{1}{\epsilon_n}\big)  < \frac{ \theta_\perp}{4}$, and let $ \kappa_2:=\frac{4e\kappa_\perp^2}{\theta_\perp}$, where $\kappa_\perp$ and $\theta_\perp$ are constants given in Proposition \ref{prop: jsq_cont_ssc}. Suppose $\theta_n := \frac{1}{\epsilon_n} \log \frac{1}{1-\epsilon_n} $ Then, for all $x>1-\epsilon_n$ we have
\begin{align} \label{eq: jsq_cont_tail_bound}
    \PP\Big(\epsilon_n \sum_{i=1}^n q_i > x\Big) \leq  \Big[2ex\big(1- \kappa_2 n \epsilon_n \log \epsilon_n\big)\Big]e^{-\theta_n x}.
\end{align}
Further, for any $n\geq 1$ and $\epsilon_n \in(0,1)$, we have the lower bound 
\begin{align}\label{eq: jsq_cont_tail_lowerbound}
    \PP\Big(\epsilon_n \sum_{i=1}^n q_i > x\Big) \geq \frac{1}{1-\epsilon_n} e^{-\theta_n x}.
\end{align}
As a consequence, we have the following large deviation result.
\begin{align}
\label{eq: jsq_cont_large_dev}
   \lim_{x\rightarrow \infty} -\frac{1}{x} \log \PP\Big(\epsilon_n \sum_{i=1}^n q_i > x\Big) = \theta_n = \frac{1}{\epsilon_n} \log \frac{1}{1-\epsilon_n} \in \Big(1, \frac{1}{1-\epsilon_n} \Big).
\end{align}
\end{theorem}

Theorem \ref{thm: jsq_cont_tail_bound} establishes the exponential decay of the tail of the total queue length for a JSQ system. The bounds presented in Eq. \eqref{eq: jsq_cont_tail_bound} and \eqref{eq: jsq_cont_tail_lowerbound} are valid for both the HT and Many-Server-HT regimes. 
% For Classic-HT regimes, the condition $n \epsilon_n\log \big(\frac{1}{\epsilon_n}\big)  < \frac{ \theta_\perp}{4}$ translates to $\epsilon\log \big(\frac{1}{\epsilon}\big)  < \frac{ \theta_\perp}{4n}$. 
Further, the result in Theorem \ref{thm: jsq_cont_tail_bound} is consistent with the fact that the distribution of the scaled steady-state total queue length, i.e., $\epsilon_n \overline{q}$, converges to an exponential random variable in distribution, as $n$ grows to $\infty$. 

In Theorem \ref{thm: jsq_cont_tail_bound}, we are able to characterize the exact tail decay rate of the continuous time JSQ system. Our result implies that, in Many-Server-HT with $\alpha>1$ and when the term $n \epsilon_n\log \big(\frac{1}{\epsilon_n}\big)$ is small enough, the decay rate of the JSQ system exactly matches the tail decay rate of an SSQ. This is a significant advancement compared to existing literature. Previous work primarily focused on comparing the behavior of the JSQ system with an SSQ in the limiting traffic condition, specifically as $\epsilon_n \rightarrow 0$. In contrast, our work examines the behavior of a pre-limit JSQ system and directly compares it to the corresponding SSQ.

% \begin{remark}
% Theorem \ref{thm: jsq_cont_tail_bound} represents 
% \end{remark}

\begin{remark}
Tail probabilities are often better characterized by multiplicative rather than additive errors. For instance, a bound of the form $\Big| \mathbb P (\epsilon q > x) - \exp \Big(-\frac{2\mu x}{\sigma^2}\Big) \Big| \leq O(\epsilon)$ can become very loose for large values of $x$. 
An interesting example is presented in \cite[Section 2.1]{vershynin2018high}, where the Central Limit Theorem provides an approximation of $N$ i.i.d. random variables by a normal distribution, but the error in the approximation can be as large as $O(1/\sqrt{N})$, which is much larger than the tail itself. Thus, even if the tail probability converges to an exponentially decaying tail in the limit, the error in the approximation may still decay slowly.
Therefore, concentration-inequality-type bounds (as in Theorem \ref{thm: ssq_tail_bound}) better characterize the tail probability. 
\end{remark}

Additionally, Eq. \eqref{eq: jsq_cont_large_dev} represents a significant result in the form of an LD principle. As $\epsilon_n \rightarrow 0$ in HT or Many-Server-HT, we have $\theta_n \rightarrow 1$, and so, our LD result for $\epsilon_n \overline{q}$ is asymptotically precise in both, the HT and Many-Server-HT settings (provided that $\alpha>1$). %A proof sketch for Theorem \ref{thm: jsq_cont_tail_bound} is provided in Section \ref{sec: jsq_cont_technical_details} and complete proof is provided in Appendix \ref{app: jsq_cont}.

% \begin{remark}
% The tail probability $\PP\big(\epsilon_n \overline q > x\big)$ is bounded both from above and below, as established in Theorem \ref{thm: jsq_cont_tail_bound}. Consequently, we derive a significant result in the form of a large deviation principle. Specifically, as $x$ tends to infinity, we have 
% \begin{align*}
%     \lim_{x\rightarrow \infty} -\frac{1}{x} \log \PP\Big(\epsilon_n \sum_{i=1}^n q_i > x\Big) \in \Big(1,\frac{1}{1-\epsilon_n} \Big).
% \end{align*}
% Therefore, our large deviation result pertaining to the scaled total queue length, namely $\epsilon_n \overline{q}$, is asymptotically precise in both the Classic-HT and Many-Server-HT settings, provided that $\alpha>1$.
% \end{remark}

% \begin{theorem}
%     Consider the JSQ system as presented in Section \ref{sec: jsqcont_model}. Suppose the system satisfies the condition $\lambda_n = n\mu(1-\epsilon_n)$, i.e., the load of the system in $\rho_n:= 1- \epsilon_n$. Then, for all $n \geq 1$ and all $x>0$ such that $\epsilon_n(1+x)>1$, we have
%     \begin{align*}
%         \PP\Big( \frac{1}{n} \sum_{i=1}^n q_i >x \Big) \leq e \epsilon_n(1+ x) e^{-\epsilon_n (1+x) }.
%     \end{align*}
% \end{theorem}

% The result presented in Theorem \ref{thm: jsq_cont_tail_bound} for JSQ  has similar implication as that of Theorem \ref{thm: ssq_tail_bound} for SSQ. In fact, one can easily obtain Theorem \ref{thm: ssq_tail_bound} from Theorem \ref{thm: jsq_dis_tail_bound} by substituting $n=1$ and $\kappa_2 =0$. 

\subsubsection{Discussion on terms in Theorem \ref{thm: jsq_cont_tail_bound}:} 

Our bounds on tail probability on JSQ system, presented in Theorem \ref{thm: jsq_cont_tail_bound}, can be decomposed into terms as discussed below.
\begin{itemize}
    \item \textbf{SSC violation:} For the JSQ system, the SSC violation term is given by $\big(1- \kappa_2 n\epsilon_n \log \epsilon_n\big) $. In non-asymptotic HT conditions (i.e., when $\epsilon >0$ in HT, or $n< \infty$ in Many-Server-HT), the SSC property is not fully satisfied. This introduces an additional multiplicative term in the tail probability bound, which is captured by $1- \kappa_2 n\epsilon_n \log \epsilon_n$, and reflects the extent to which SSC is violated. 

To get the SSC violation term below a certain threshold $\delta$ in HT, we need $\epsilon \log \big(\frac{1}{ \epsilon}\big)$ to be on the order of $O\big(\frac{\delta}{n} \big)$. Similarly, in the Many-Server-HT scenario, it is required that $n^{1-\alpha}\log n $ is on the order of $O(\delta)$, or alternatively, $n$ must be at least $\Omega \big(\exp\big(\frac{1}{\alpha-1} \log \frac{1}{\delta}\big)\big)$, or $n \sim \Omega \big(\delta^{-\frac{1}{\alpha-1}}\big)$
 in magnitude. Satisfying such conditions becomes increasingly challenging as $\alpha$ approaches $1$. This is also shown in Fig. \ref{fig:JSQ-contributions} in Section \ref{sec:intro-main-contribution}. Furthermore, it is important to acknowledge that such conditions cannot be met for $\alpha \leq 1$. These observations provide an intuitive argument for the failure of the SSC property (as presented in this paper) when $\alpha \leq 1$. Essentially, for $\alpha \leq 1$, the notion of SSC deviates significantly from the one considered in this paper.

 It is worth noting that although a large system size, represented by $n$, is required for the SSC error to be small, it remains bounded. Specifically, in the case of Many-Server-HT, it can be shown that the term $n \epsilon_n \log \frac{1}{\epsilon_n}$ is bounded by $\frac{\alpha}{e(\alpha-1)}$. Consequently, the SSC error is of order at most $O\left(\frac{1}{\alpha-1}\right)$. Therefore, %in Many-Server-HT with $\alpha>1$, 
 it is reasonable to expect that even for moderate system sizes, the tail probability, $\PP(\epsilon\overline{q}>x)$, exhibits exponential decay. 
 \vspace{0.5em}

\item \textbf{Pre-limit tail:} 
    The pre-limit tail denotes the actual decay rate of the tail probability of $\epsilon_n \overline{q}$ under non-asymptotic HT condition, i.e., $n<\infty$. For the continuous-time JSQ system, we exactly characterize the pre-limit tail, which is given by $\theta_n = \frac{1}{\epsilon_n} \log \frac{1}{1-\epsilon_n}$. In super-slowdown regime,
    as $\epsilon_n$ goes to zero, the tail of $\epsilon_n \overline{q}$ matches that of an exponential distribution with mean $1$, as $\lim_{n\rightarrow \infty} \theta_n =1$. Further, note that, the deviation of the pre-limit tail from the corresponding HT value is given by $|\theta_n-1|$, which is of order $O(\epsilon_n)$.

    In the case of continuous-time JSQ, the arrivals and service times follow an exponential distribution, which is characterized by a single parameter, i.e., the arrival or service rate. The specific properties of the exponential distribution allow for a more accurate analysis of the tail probability.  Intuitively, this distinction is the reason why we are able to precisely characterize the pre-limit tail for the JSQ system. In case of general arrival and service distributions, as presented in the case of the discrete-time SSQ, 
    % the pre-limit error occurs because arrivals and service times follow a general distribution and 
    we do a second-order approximation to approximate the tail probability of the scaled queue length with its HT counterpart. Consequently, we obtain an upper bound on the pre-limit tail. More arguments are provided in Section \ref{sec: ssq}.
    % Consequently, there is no need for a second-order approximation to determine the decay rate of the tail probability. 
    
\item \textbf{Pre-exponent error:} In the context of the JSQ system, the pre-exponent error is represented by the expression $2ex$. This error term arises from  using Markov's Inequality to obtain tail-probability bounds using MGF. %Specifically, this error term accounts for any potential discrepancy between the actual tail probability and the upper bound derived using Markov's Inequality.
        % As illustrated in Section \ref{sec:intro-methodology}, even for an exponential random variable, there is a multiplicative error of the form $e\lambda x$ in the bound. 
        To clarify this error term, consider a random variable $X$ that follows an exponential distribution with rate $\lambda$. 
        In this case, the MGF of $X$ is given by $\EE[\exp (\theta X)] = \frac{1}{1- \theta/\lambda}$ for all $\theta< \lambda$. As shown in Lemma \ref{lem: mgf_to_prob_bound} of Appendix \ref{app: essential_lemmas}, by 
        Applying Markov's Inequality to the MGF and optimizing over the value of $\theta$, we obtain $$\PP(X>x) \leq e\lambda x e^{-\lambda x}.$$
        % \begin{equation*}
        % \end{equation*}
        The upper bound differs from the actual tail of $X$ by a multiplicative factor of $e\lambda x$, which arises from using Markov's Inequality. 
        We acknowledge that it may be possible to eliminate the Markov-Inequality error by employing more complex techniques. However, we have chosen to rely solely on Markov's Inequality for our analysis to maintain simplicity. 
\end{itemize}
 % However, for moderate system sizes, the difference between lower bound and upper bound in Theorem \ref{thm: jsq_cont_tail_bound} might be significant. 

% i.e., for $- \kappa_2 n\epsilon_n \log \epsilon_n \leq \delta$

% The explanation for the first two error terms, that is Pre-limit error and Pre-exponent error is  same as that in the case of SSQ. The intuitive reasoning for the third error term is as follows.

% We consider the sum exponential as the Lyapunov function, given by
% \begin{align*}
%     V(\mathbf x;\theta) = \sum_{i=1}^n e^{\theta x_i} + \sum_{i=1}^n e^{- \theta x_i}. 
% \end{align*}

\subsection{State Space Collapse for JSQ system}
\label{sec: jsq_ssc}

Next, we mathematically specify SSC for the JSQ system with $\lambda_n = n\mu(1-n^{-\alpha})$ and $\alpha>1$.
 %Next, we start with stating that continuous time JSQ system achieves SSC for all $\alpha>1$, where $\lambda_n = n\mu(1-n^{-\alpha})$. 

\begin{proposition}
\label{prop: jsq_cont_ssc}
Consider the JSQ system as presented in Section \ref{sec: jsqcont_model}. Suppose the system satisfies the condition $\lambda_n = n\mu(1-\epsilon_n)$, and let $\kappa_\perp = 128$ and $\theta_\perp = 1/96$. Then, for $\epsilon_n \leq 1/2$, and $\theta\in (0,\theta_\perp)$, we have that for all $i \in \{1,2,\dots,n\}$,
\begin{align}
    \label{eq: jsq_cont_perpmgfbound}
    \EE_{\pi_n}\Big[ \sum_{i=1}^n e^{-\theta q_{\perp i}}+   \sum_{i=1}^n e^{\theta q_{\perp i}} \Big] \leq \kappa_\perp n, && \EE_{\pi_n}\big[  e^{\theta |q_{\perp i}|}\big] \leq \kappa_\perp,
\end{align}
where $q_{\perp i} = q_i - \frac{1}{n} \sum_{j=1}^n q_j$. 
\end{proposition}

As mentioned before, an essential component in establishing the limiting distribution (Theorem \ref{thm: jsq_cont_limiting_dis}) and the tail probability (Theorem \ref{thm: jsq_cont_tail_bound}) is to characterize SSC accurately. For the JSQ system, when SSC is achieved, the queue-length vector $\q$ collapses to a subspace where all its coordinates become identical, implying that the asymptotic behavior of the JSQ system closely resembles that of an SSQ.

In Proposition \ref{prop: jsq_cont_ssc}, we establish a specific notion of SSC by demonstrating that the MGF of the deviations of the individual queue lengths and their average, referred to as the "perpendicular component" and denoted by $\q_{\perp} = \q - \frac{1}{n}\overline{q} \mathbf{1}_n$, remains uniformly bounded within an interval around zero for all values of $\epsilon_n\leq 1/2$. In particular, this result implies that all the moments of the perpendicular component remain uniformly bounded, even in the limit as the system approaches HT conditions ($\epsilon_n \rightarrow 0$). In comparison, the elements of the queue length vector $\q$ are on the order of $\frac{1}{\epsilon_n}$. Consequently, as $\epsilon_n$ becomes small, we observe that $\epsilon_n \q\approx\frac{\epsilon_n}{n} \overline{q}\, \mathbf{1}_n$ (or equivalently, $\epsilon_n \q_\perp \approx \mathbf{0}_n$).

\paragraph{Intuition behind the choice of the Lyapunov function} In order to establish Proposition \ref{prop: jsq_cont_ssc}, we employ the Lyapunov function  $\sum_{i=1}^n e^{-\theta q_{\perp i}}+ \sum_{i=1}^n e^{\theta q_{\perp i}}$, demonstrating its negative drift. It is worth noting that this Lyapunov function differs from the one considered in \cite{HurMag_2022_alpha_continuous}, namely $\exp \big(\theta \| \q_\perp\|_2\big)$, which yields the SSC result in the Many-Server-HT regime for $\alpha>2$. However, $\exp \big(\theta \| \q_\perp\|_2\big)$ proves inadequate for establishing SSC in the case of $\alpha \in (1,2]$. This limitation arises because, to show SSC in Many-Server-HT for $\alpha \in (1,2]$, it is necessary to obtain a bound for the MGF of $\|\cdot\|_\infty$-norm of the perpendicular component, denoted as $\| \q_\perp\|_\infty$. Intuitively, for an $n$-dim vector $\mathbf{x}$, if $\|\mathbf x\|_{\infty}$ is $O(1)$, then it implies that $\|\mathbf x\|_{2}$ is $O(\sqrt{n})$. However, vice-versa need not hold. As such, a bound on $\ell_{\infty}$ norm is sharper than a bound on $\ell_2$ norm, and it plays a significant role while proving SSC for $\alpha\in (1,2]$. Furthermore, using $\exp \big(\theta \| \q_\perp\|_\infty\big)$ as the Lyapunov function proves challenging due to the non-smooth nature of the $\| \cdot\|_\infty$-norm. Consequently, we use the Lyapunov function $\sum_{i=1}^n e^{-\theta q_{\perp i}}+ \sum_{i=1}^n e^{\theta q_{\perp i}}$ as a suitable alternative. This choice represents a smooth approximation of $\exp \big(\theta \| \q_\perp\|_\infty\big)$ and simplifies the technical aspects of the analysis compared to dealing directly with $\exp \big(\theta \| \q_\perp\|_\infty\big)$. Similar exponential function was also used in the context of graphical allocation of balls in bins in \cite{peres2015graphical, bansal2022power}.

\subsection{Proof of Theorem \ref{thm: jsq_cont_limiting_dis} and Theorem \ref{thm: jsq_cont_tail_bound}}
\label{sec: jsq_cont_technical_details}

In this section, provide the proof  sketches for Theorem \ref{thm: jsq_cont_limiting_dis} and Theorem \ref{thm: jsq_cont_tail_bound}. More details and precise mathematical arguments are provided in Appendix \ref{app: jsq_cont}.

\begin{lemma}\label{lem: jsq_cont_steady_state}
Consider the JSQ system as presented in Section \ref{sec: jsqcont_model}. Suppose the system satisfies the condition $\lambda_n = n\mu(1-\epsilon_n)$, where $\alpha >1$. Let  
\begin{align*}
    \gamma_n(\theta) :=  n \mu -\lambda_n  e^{\epsilon_n\theta} &&\beta_n (\q;\theta):= \mu\sum_{i=1}^n  \mathds 1_{\{q_i = 0\}} e^{\theta \epsilon_n \sum_{i=1}^n q_i}.
\end{align*}
Then, for any $\theta< \theta_n:= - \frac{1}{\epsilon_n} \log\big( 1- \epsilon_n\big)$, such that $\theta n\epsilon_n < \theta_\perp$, we have 
\begin{align}\label{eq: jsq_cont_mgf}
    \EE_{\pi_n}\Big[e^{\theta \epsilon_n \sum_{i=1}^n q_i }\Big] = \frac{1}{\gamma_n(\theta)} \EE_{\pi_n}[\beta_n(\q;\theta)].
    \end{align}
\end{lemma}
Lemma \ref{lem: jsq_cont_steady_state} provides the MGF of the scaled total queue length $\epsilon_n \overline{q}$ and is an intermediate step to prove the results in Theorems \ref{thm: jsq_cont_limiting_dis} and \ref{thm: jsq_cont_tail_bound}. The proof of Lemma \ref{lem: jsq_cont_steady_state} is provided in Appendix \ref{app: jsq_cont}.

\paragraph{Technical note} It is important to note that the condition $\theta n\epsilon_n < \theta_\perp$ in Lemma \ref{lem: jsq_cont_steady_state} holds true for all $\theta<0$, and when $\epsilon_n$ is small for $\theta \in (0,\theta_n)$. Consequently, Eq. \eqref{eq: jsq_cont_mgf} is valid for all $n$ and $\epsilon_n$ if $\theta<0$, but if $\theta \in (0,\theta_n)$, it can only be used for $\epsilon_n$ small enough. Essentially, we need to use the SSC result given in Proposition \ref{prop: jsq_cont_ssc} to obtain the bounds in Theorems \ref{thm: jsq_cont_limiting_dis} and \ref{thm: jsq_cont_tail_bound}, and these are valid only under the condition $\theta n\epsilon_n < \theta_\perp$.

\begin{proof}[Proof of Theorem \ref{thm: jsq_cont_limiting_dis}]
Recall that we use the notation $ \lambda_n = n\mu(1-\epsilon_n)$ and according to the theorem statement, $\epsilon_n = n^{-\alpha}$, where $\alpha >1$. Also,  $\overline{q} = \sum_{i=1}^n q_i$.  In order to prove Theorem \ref{thm: jsq_cont_limiting_dis}, we use Lemma \ref{lem: jsq_cont_steady_state}. We first prove the following claim.

    \begin{claim}\label{claim: jsq_cont_beta}
    For $n\geq 2$, we have 
\begin{align}\label{eq: jsq_cont_unused}
      \EE_{\pi_n}\big[  \mathds 1_{\{q_i = 0\}}\big] = \epsilon_n, \quad \forall i \in \{1,2,\dots,n\}.
  \end{align}
    And, for any $\theta<\theta_n$ and $\epsilon_n $ small enough such that $2|\theta |  n\epsilon_n\log \big(\frac{1}{\epsilon_n} \big)  < \theta_\perp$, we have 
    \begin{align*}
        \sum_{i=1}^n \EE_{\pi_n} \Big[ \mathds 1_{\{q_i = 0\}} \big| e^{\theta \epsilon_n \overline q} -1\big|\Big] \leq \frac{2e\kappa_\perp^2}{\theta_\perp} |\theta| n^2\epsilon_n^2 \log \Big(\frac{1}{\epsilon_n} \Big),
    \end{align*}
    where $\theta_\perp$ and $\kappa_\perp$ are as defined in Proposition \ref{prop: jsq_cont_ssc}.
    \end{claim}
    Proof of Claim \ref{claim: jsq_cont_beta} is provide in Appendix \ref{app: jsq_cont}. Now, according to the assumption, $\epsilon_n = n^{-\alpha}$, where $\alpha >1$. Under the $\alpha>1$ condition, the condition $2|\theta |  n\epsilon_n\log \big(\frac{1}{\epsilon_n} \big)  < \theta_\perp$ can be satisfied for $n$ large enough, i.e., we need $n$ large enough such that $2\alpha |\theta |  n^{1-\alpha}\log n  < \theta_\perp$, or $\frac{\log n}{ n^{\alpha-1}} < \frac{\theta_\perp}{2\alpha |\theta |}$.
    From Claim \ref{claim: jsq_cont_beta}, we have that for any $\theta <1$, (as $\theta_n \geq 1$), we have,
    \begin{align*}
        \lim_{n\rightarrow \infty} &  \frac{1}{n\epsilon_n}\sum_{i=1}^n\EE_{\pi_n}\big[  \mathds 1_{\{q_i = 0\}}\big] =1,
    \end{align*}
    and
    \begin{align*}
       \lim_{n\rightarrow \infty} &  \frac{1}{n\mu\epsilon_n}\Big|\EE_{\pi_n}[\beta_n (\q;\theta)] - \mu\sum_{i=1}^n\EE_{\pi_n}\big[  \mathds 1_{\{q_i = 0\}}\big] \Big|\leq  \lim_{n\rightarrow \infty}  \frac{1}{n\mu\epsilon_n} \sum_{i=1}^n \EE_{\pi_n} \Big[ \mathds 1_{\{q_i = 0\}} \big| e^{\theta \epsilon_n \overline q} -1\big|\Big] =  0.
    \end{align*}
    Thus, we have $\lim_{n\rightarrow \infty} \frac{1}{n\mu\epsilon_n}\EE_{\pi_n}[\beta_n (\q;\theta)] =1$. Further, for any $\theta$, we have 
    \begin{align*}
        \lim_{n\rightarrow \infty} \frac{1}{n\mu\epsilon_n} \gamma_n(\theta) &= \frac{1}{n\mu\epsilon_n} \big( n \mu -\lambda_n  e^{\epsilon_n\theta} \big) = \frac{1}{\epsilon_n} \big( 1- e^{\epsilon_n\theta}+ \epsilon_n  e^{\epsilon_n\theta} \big)=1-\theta.
    \end{align*}
    Thus, by using Lemma \ref{lem: jsq_cont_steady_state}, for any $\theta<1$,
    \begin{align*}
        \lim_{n\rightarrow \infty} \EE[e^{\theta \epsilon_n \overline q}] = \lim_{n\rightarrow \infty} \frac{1}{\gamma_n(\theta)} \EE_{\pi_n}[\beta_n (\q;\theta)] = \frac{1}{1-\theta}.
    \end{align*}
    This shows that the Moment Generating Function (MGF) of  $\epsilon_n \overline{q}$ converges to that of an exponential random variable with mean $1$. Now, by using \cite[Theorem 25.10]{Bill86}, we have that $\epsilon_n \overline{q} \stackrel{d}{\rightarrow} \Upsilon$, where $\Upsilon$ is an exponential random variable with mean $1$. 

    Next, by using Proposition \ref{prop: jsq_cont_ssc}, for $\theta\in(0,\theta_\perp)$ we have that
    \begin{align*}
    1\leq \lim_{n\rightarrow\infty} \EE_{\pi_\epsilon} \big[ e^{n \epsilon_n \theta |q_{\perp_i}|}\big] \leq  \lim_{n\rightarrow\infty} \Big(\EE_{\pi_\epsilon} \big[ e^{ \theta |q_{\perp_i}|}\big]\Big)^{n\epsilon_n} \leq \lim_{n\rightarrow\infty} \kappa_\perp^{n\epsilon_n} =1,
    \end{align*}
    where the second inequality follows by using Jensen's Inequality.
    This implies that, for all $i \in \{1,2,\dots,n\}$, we have $n \epsilon_n q_{\perp_i} \stackrel{d}{\rightarrow} 0 $. Now, by definition of $\q_\perp$, we have,  $n\epsilon_n \q = \epsilon_n \overline q \mathbf 1 +  n\epsilon_n \q_\perp$, where $\epsilon_n \overline{q} \stackrel{d}{\rightarrow} \Upsilon$ and $n \epsilon_n q_{\perp i} \stackrel{d}{\rightarrow} 0 $ as $n\rightarrow \infty$. Thus, we have $n\epsilon_n \q \stackrel{d}{\rightarrow} \Upsilon \mathbf 1$. This completes the proof.
\end{proof}

% \begin{proof}[Proof sketch for Theorem \ref{thm: jsq_cont_limiting_dis}] 
% To prove Theorem \ref{thm: jsq_cont_limiting_dis}, we use directly use Lemma \ref{lem: jsq_cont_steady_state}. First note that $\theta_n >1$ for all $n$, and so, Eq. \eqref{eq: jsq_cont_mgf} is valid for all $n$ such that $n\epsilon_n< \theta_\perp$. As $\epsilon_n = n^{-\alpha}$ with $\alpha>1$, the condition $n\epsilon_n< \theta_\perp$ can be satisfied for $n$ large enough. Next, we show that in Many-Server-HT with $\alpha>1$, we have that for all $\theta<1$,
% \begin{align*}
%     \lim_{n\rightarrow \infty} \frac{1}{\mu n \epsilon_n} \EE_{\pi_n}[\beta_n(\q;\theta)]=1, && \lim_{n\rightarrow \infty} \frac{1}{\mu n \epsilon_n} \gamma_n(\theta) = 1-\theta.
% \end{align*}
% Afterward, the result follows by using \cite[Theorem 25.10]{Bill86} and the SSC result in Proposition \ref{prop: jsq_cont_ssc}.
% \end{proof}

\begin{proof}[Proof of Theorem \ref{thm: jsq_cont_tail_bound}]
Note that for $\epsilon_n \leq 1/2$, we have $\theta_n \leq 2$. Thus, the condition $2|\theta|  n \epsilon_n\log \big(\frac{1}{\epsilon_n}\big)  < \theta_\perp$ is satisfied for any $\theta \in(0, \theta_n)$ whenever $\epsilon_n$ is small enough such that $4 n \epsilon_n\log \big(\frac{1}{\epsilon_n}\big)  < \theta_\perp$. Then, from Claim \ref{claim: jsq_cont_beta}, for any $\theta \in(0, \theta_n)$, we have 
  \begin{align*}
     \EE_{\pi_n} [\beta_n(\q;\theta)] & \leq \mu n\epsilon_n\big( 1-  \kappa_2 n\epsilon_n \log \epsilon_n \big),
  \end{align*}
  where, as $\theta\leq \theta_n \leq2$, we can take $\kappa_2:=\frac{4e\kappa_\perp^2}{\theta_\perp}$.
    Now, by using Markov's inequality, we have that for all $\theta\in(0,\theta_n)$,
    \begin{align}
    \label{eq: jsq_cont_thm_markovineq}
        \PP\big(\epsilon_n \overline{q} > x \big) &\leq e^{-\theta x} \EE\big[ e^{\theta \epsilon_n \sum_{i=1}^n q_i  }\big] \allowdisplaybreaks \nonumber\\
        &\leq \mu n\epsilon_n\big( 1-  \kappa_2 n\epsilon_n \log \epsilon_n\big) \times \frac{1}{\gamma_n(\theta)} e^{-\theta x} \allowdisplaybreaks \nonumber\\
        & = \mu n\epsilon_n\big( 1-  \kappa_2 n\epsilon_n \log \epsilon_n\big) \times \frac{1}{n \mu -\lambda_n  e^{\epsilon_n\theta} } e^{-\theta x}.
    \end{align}
    Next, we optimize the upper bound over the values of $\theta \in(0,\theta_n)$. By differentiation, 
    \begin{align*}
        \frac{d}{d\theta} \Big( \frac{1}{\gamma_n(\theta)} e^{-\theta x} \Big) = \frac{1}{\gamma_n^2(\theta)} \Big( -x\gamma_n(\theta) + \lambda_n \epsilon_n e^{\epsilon_n \theta} \Big)e^{-\theta x} 
    \end{align*}
    Then, the derivative is equal to zero at
    \begin{align*}
        \theta = \theta_{x,n} := \frac{1}{\epsilon_n}\log \Big( \frac{n\mu x}{\lambda_n (x + \epsilon_n)}\Big)= \frac{1}{\epsilon_n} \log \Big( \frac{x}{(1-\epsilon_n)(x + \epsilon_n)}\Big)< \theta_n,
    \end{align*}
    where the last inequality is valid for any $x >0$.
    Also, to ensure that $\theta_{x,n}>0$, we need that
    \begin{align*}
        \frac{x}{(1-\epsilon_n)(x + \epsilon_n)}>1 \implies x > 1- \epsilon_n.
    \end{align*}
    Then, by substituting $\theta = \theta_{x,n}$, we have
    \begin{align*}
        \frac{1}{n \mu -\lambda_n  e^{\epsilon_n\theta_{x,n}} } e^{-\theta_{x,n} x} &= \frac{x+\epsilon_n}{\mu n\epsilon_n} \exp\Big( -\frac{x}{\epsilon_n}\log \Big( \frac{x}{(1-\epsilon_n)(x + \epsilon_n)}\Big) \Big)\\
        & = \frac{x+\epsilon_n}{\mu n\epsilon_n} \exp\Big( -\frac{x}{\epsilon_n} \log \frac{1}{1-\epsilon_n} + \frac{x}{\epsilon_n} \log \frac{x+ \epsilon_n}{x} \Big)\\
        &\leq \frac{e(x+\epsilon_n)}{\mu n\epsilon_n} \exp\Big( -\frac{x}{\epsilon_n} \log \frac{1}{1-\epsilon_n} \Big).
    \end{align*}
    Using the above result in Eq. \eqref{eq: jsq_cont_thm_markovineq} and $x\geq \epsilon_n$, we have 
    \begin{align*}
        \PP\big(\epsilon_n \sum_{i=1}^n q_i > x \big) &\leq  2ex \big( 1-  \kappa_2 n\epsilon_n \log \epsilon_n\big)  \exp\Big( -\frac{x}{\epsilon_n} \log \frac{1}{1-\epsilon_n} \Big).
        % \\
        % &\leq 2ex \big( 1-  \kappa_2 n\epsilon_n \log \epsilon_n\big) e^{-x},
    \end{align*}
    % where the last inequality follows simply because $\frac{1}{\epsilon_n} \log \frac{1}{1-\epsilon_n}\geq 1$.
    This completes the first part of the proof. For the second, i.e., the lower bound in Eq. \eqref{eq: jsq_cont_tail_lowerbound}, we couple the JSQ system with a Single Server Queue (SSQ) as follows, i.e., we consider a SSQ process $\{q_{ssq}(t)\}_{t\geq 0}$, where all the servers of the original JSQ system are pooled together to create a single server with service rate $n\mu$. Thus, $\{q_{ssq}(t)\}_{t\geq 0}$ is an $M/M/1$ queue with arrival rate $\lambda_n$ and service rate $n\mu$. Then, the stationary distribution of $\{q_{ssq}(t)\}_{t\geq 0}$ is given by,
    \begin{align*}
        \PP(q_{ssq} = i) = \Big( 1- \frac{\lambda_n}{n\mu} \Big) \Big( \frac{\lambda_n}{n\mu} \Big)^{i} = \epsilon_n (1- \epsilon_n)^{i}.
    \end{align*}
    This gives us
    \begin{align*}
       \PP\Big(\epsilon_n \sum_{i=1}^n q_i > x\Big) &\geq   \PP(\epsilon_n q_{ssq}  > x) = (1- \epsilon_n)^{\big\lceil\frac{x}{\epsilon_n}\big\rceil} \geq (1- \epsilon_n)^{\frac{x}{\epsilon_n}-1} = \frac{1}{1-\epsilon_n} \exp\Big( -\frac{x}{\epsilon_n} \log \frac{1}{1-\epsilon_n} \Big).
       % \exp \Big( \Big\lceil\frac{x}{\epsilon_n}\Big\rceil \log(1-\epsilon_n) \Big).
    \end{align*}
    This completes the proof.
    % such that $q_{ssq}(t) \leq \sum_{i=1}^n q_i(t)$. Such a process exists smiply because one
% Suppose  is a queue length process for a SSQ. At any time $t$, if there is an arrival in the $\{\q(t)\}_{t\geq 0}$ process, then there is an arrival in the $\{q_{ssq}(t)\}_{t\geq 0}$ process also. Similarly, at any time $t$, if there is a service in the $\{\q(t)\}_{t\geq 0}$ process, then there is a in the $\{q_{ssq}(t)\}_{t\geq 0}$ process also, if $q_{ssq}(t) \neq 0$. Then, for the SSQ, by memoryless property of the exponential random variables, the arrivals follow a Poisson process with rate $\lambda_n$ and service is exponentially distributed with rate $n\mu$. 
\end{proof}

\section{Discrete time Single server queue}
\label{sec: ssq}
In this section, we present the bound on the tail probability for a SSQ. The aim is to show the affect of a general arrival and service distribution on the decay rate of the tail probability $\PP(\epsilon q>x)$. The model is as follows.

\subsection{Model for Single Server Queue}
\label{sec: ssq_model}
We consider a sequence of discrete-time SSQ indexed by the HT parameter $\epsilon$, where $q(t)$ denotes the number of customers in the system (queue length) at the beginning of time slot $t$. Customers arrive to the system as an i.i.d. process $\{a(t)\}_{t\geq 0}$, with $\EE [a(t)] = \lambda_\epsilon$, Var$(a(t)) = \sigma_{\epsilon,a}^2$, and $a(t)\leq A$ almost surely. Once the customers join the queue, the server decides to serve up to $s(t)$ jobs waiting in the queues, with $\EE[s(t)] = \mu_\epsilon$, Var$(s(t)) = \sigma_{\epsilon,s}^2$, and $s(t)\leq A$ almost surely. Even though the server can process $s(t)$ customers, it is possible that there are not enough customers in the system, and some of the service gets wasted. Hence, we call $s(t)$ as the `potential' service. Similar to the arrival process, the potential services are also independent and identically distributed across time slots. Moreover, we assume that these potential services are independent of the queue length vector. 
The queue evolution process is given by
\begin{equation}
\label{eq: ssq_lindley}
    q(t+1) = [q(t) +a(t) - s(t)]^+ = q(t) +a(t) - s(t) + u(t) ,
\end{equation}
where $[x]^+:=\max\{x,0\}$ is used because the queue length cannot be negative. The term ${u}(t)$ is the unused service and represents the difference between the potential and actual service. By definition, the unused  service $u(t) $ is positive only if $q(t+1) =0$, which implies $q(t+1)u(t) =0$ for all $t$. Also, the  unused service cannot be higher than the potential, and so, $0\leq u(t)\leq s(t) \leq A$ almost surely.  

It is well known that $\lambda_\epsilon < \mu_\epsilon$ implies stability, that is, the Markov chain $\{q(t)\}_{t\geq 0}$ is positive recurrent. As a result, the system reaches steady state. We use $\pi_\epsilon$ to denote the steady state distribution of the queue length process $\{q(t)\}_{t\geq 0}$. We drop the symbol $t$ to denote that variables in steady state, i.e., $q$ follows the steady state distribution $\pi_\epsilon$, and $q^+$ denote the state that comes after $q$, i.e., $q^+ = [q+a-s]^+= q + a-s+u$, where $a$ and ${s}$ follow the same distribution as $a(t)$ and ${ s}(t)$, respectively.

\subsection{Results for SSQ}

In this section, we establish an upper bound on the tail of the queue length distribution for a system in non-asymptotic  HT. We consider $\lambda_\epsilon = \mu_\epsilon(1-\epsilon)$, so that as $\epsilon \rightarrow 0$, the system approaches HT conditions. In Theorem \ref{thm: ssq_tail_bound}, we provide a bound on the probability $\mathbb P (\epsilon q > x)$ when the HT parameter is bounded away from zero.

\begin{theorem}
\label{thm: ssq_tail_bound}
     Consider the SSQ as defined in Section \ref{sec: ssq_model}, with $\lambda_\epsilon = \mu_\epsilon(1-\epsilon)$. Let $\sigma_\epsilon^2 := \sigma_{\epsilon,a}^2+\sigma_{\epsilon,s}^2$,
     \begin{align*}
       \theta_\epsilon := \frac{2\mu_\epsilon}{\sigma_\epsilon^2}, &&  \kappa_1:=\frac{ 2\mu_\epsilon E_3}{3\sigma_\epsilon^4}  + \epsilon \frac{9 \mu_\epsilon^2 A^4}{\sigma_\epsilon^6},
     \end{align*}
     where $E_3:= \max\{0, \EE[(a-s)^3]\}$. Let $\EE[e^{\epsilon \theta q(0)}]<\infty$ for all $\theta \leq \theta_\epsilon$. Then, for all $x> \frac{\theta_\epsilon}{ (1+ \kappa_1 \epsilon)}$, we have
     \begin{align}
         \label{eq: ssq_tail_bound}
          \mathbb P (\epsilon q > x) \leq e\theta_\epsilon x  e^{-\theta_\epsilon (1-\kappa_1 \epsilon)x}, 
          % \exp\Big(-\frac{2\mu_\epsilon(1 - \kappa_2\epsilon ) x }{\sigma_\epsilon^2 }\Big),
         \ \ \ \ \forall \epsilon \in (0,1).
     \end{align}
     % Then, in HT, i.e., for $\epsilon \rightarrow 0$, we have
     % \begin{align*}
     %     \lim_{\epsilon \rightarrow 0} \mathbb P (\epsilon q > x) \leq  \frac{2e\mu x}{\sigma^2 } \exp\Big(-\frac{2\mu
     %      x }{\sigma^2 }\Big), 
     % \end{align*}
     % where $\mu := \lim_{\epsilon\rightarrow 0} \mu_\epsilon$ and $\sigma^2 = \lim_{\epsilon\rightarrow 0} \sigma_\epsilon^2$. 
     Further, as $x\rightarrow \infty$, we have 
     \begin{align}
     \label{eq: ssq_large_deviation}
         \lim_{x\rightarrow \infty} \frac{1}{x} \log \mathbb P (\epsilon q > x) \leq \ - \theta_\epsilon (1 - \kappa_1\epsilon 
         ).
     \end{align}
     
\end{theorem}

The exponential decay of the tail of the queue-length distribution of an SSQ has been established by the tail bound presented in Eq. \eqref{eq: ssq_tail_bound}. This result is consistent with the well-known fact that the distribution of the scaled steady-state queue length $\epsilon q$ converges to an exponential distribution as $\epsilon$ tends to zero. However, our result provides a more detailed insight by characterizing the rate of convergence of the tail probability to its corresponding HT value.

Note that, SSQ is a one-dimensional system, and so, there is no concept of SSC in SSQ. Further, in case of SSQ, the pre-exponent error is given by $e\theta_\epsilon x$. The reasoning for this is same as that in the case of JSQ system. Therefore, we only discuss the pre-limit tail for the discrete time SSQ.
\vspace{0.5em}

% The tail bound presented in Eq. \eqref{eq: ssq_tail_bound} shows that the queue length distribution for a SSQ has an exponentially decaying tail. It is well known that the distribution of the scaled steady state queue length, i.e. $\epsilon q$ converges to that of an exponential random variable. Our result provides a characterization of the rate of convergence of the tail probability to the corresponding HT value. 

 % \subsubsection{Discussion on terms in Theorem \ref{thm: ssq_tail_bound}:} \label{sec: ssq_dis_errorterms}
 
%  There are two multiplicative terms in the bound presented in Theorem \ref{thm: ssq_tail_bound}. 
% \vspace{0.5em}

\textbf{Pre-limit tail:} 
As mentioned before, pre-limit tail refers to the actual decay rate of the tail probability for the pre-limit system. Under general arrival and service distribution, as in the case of discrete time SSQ, we provide an upper bound on the pre-limit tail, which given by $\theta_\epsilon (1-\kappa_1 \epsilon)$. The steady-state distribution of the pre-limit system need not follow an exponential distribution. As such, to obtain an exponential decay rate on the tail probability, we employ a second order approximation to approximate the distribution with an exponential. It is important to note that as $\epsilon$ approaches zero, the upper bound on the pre-limit tail converges to the correct HT value, i.e. $\lim_{\epsilon \rightarrow 0} \theta_\epsilon (1-\kappa_1 \epsilon) = \frac{2\mu}{\sigma^2}$, where $\mu$ and $\sigma$ are the limiting values of $\mu_\epsilon$ and $\sigma_\epsilon$ as $\epsilon \rightarrow 0$.

% and arises because when $\epsilon > 0$, the scaled steady-state queue length $\epsilon q$ does not necessarily follow an exponential distribution. For a given $\epsilon>0$, the Pre-limit Error captures the deviation of the actual tail behavior of $\epsilon q$ from the tail behavior of the corresponding HT distribution.  the Pre-limit Error converges to 1, indicating that the deviation from exponential behavior diminishes in the limit as $\epsilon$ vanishes.

Note that the constant $\kappa_1$ in the pre-limit tail depends on the third moment of $(a-s)$, while the actual decay rate of the tail probability in HT depends on the corresponding first and the second moment. Using exponential Lyapunov functions, the HT distribution of a queueing system is generally obtained using a second-order approximation of the MGF of $(a-s)$, which involves the first two moments. As such, the `strength' of this second-order approximation depends on the third moment of $(a-s)$, which is essentially captured by the dependency of the constant $\kappa_1$ on $E_3$ in Theorem \ref{thm: ssq_tail_bound}.

\subsection{Proof of Theorem \ref{thm: ssq_tail_bound}}

The first step to prove Theorem \ref{thm: ssq_tail_bound} is characterizing the steady-state distribution of $\epsilon q$ in terms of its MGF. For this goal, we consider the Lyapunov function $V(x;\theta) := e^{\theta \epsilon x}$ and obtain the following result. 
% For the Markov chain $\{q(t)\}_{t\geq 0}$, the drift of the Lyapunov function $V(q;\theta)$, denoted by $\Delta V(q(t);\theta)$ is defined as 
% \begin{align*}
%     \Delta V(q(t);\theta) := V(q(t+1);\theta)-V(q(t);\theta).
% \end{align*}

%We have the following result for the Lyapunov function $V(\cdot)$ for the Markov chain $\{q(t)\}_{t\geq 0}$.

\begin{lemma}
\label{lem: ssq_mgf_steady_state}
    Consider an SSQ as defined in Section \ref{sec: ssq_model}. Suppose the initial state $q(0)$ satisfies $\EE[e^{\theta \epsilon q(0)}] <\infty$ for all $\theta < \theta_0$ and for all $\epsilon \in (0,1)$. Then, for any $\theta <\theta_0$, we have 
\begin{align}
\label{eq: ssq_mgf_recursiveeq}
    \EE[V(q(t+1);\theta)]
    & = (1-\gamma_\epsilon(\theta))^{t+1}\EE[V(q(0);\theta)]  + \sum_{i=0}^t (1-\gamma_\epsilon(\theta))^i \EE[\beta_\epsilon (t-i;\theta)],
\end{align}
where $\gamma_\epsilon(\theta) := 1-\EE \big[e^{\epsilon\theta(a-s)}\big]$ and $\beta_\epsilon (t;\theta) := 1- \EE\big[e^{-\theta \epsilon u(t)} \big| q(t) \big].$
Let $ \Theta_\epsilon :=\{\theta \in \RR: \theta< \theta_0, \ \EE[e^{\theta \epsilon(a-s)}] < 1\}$. Then, for all $\theta\in \Theta_\epsilon \cup \{\theta: \theta\leq 0\}$, we have
    \begin{align}
    \label{eq: ssq_mfg_closed_form}
       \EE_{\pi_\epsilon}[V(q;\theta)]:= \EE_{\pi_\epsilon}\big[e^{\theta \epsilon q}\big] = \frac{1- \EE_{\pi_\epsilon}[e^{-\theta \epsilon u}]}{1-\EE \big[e^{\epsilon\theta(a-s)}\big]}.
    \end{align}
\end{lemma}

The expression in Eq. \eqref{eq: ssq_mfg_closed_form} was first introduced in \cite{hurtado2020transform} for $\theta$ in an interval around zero. In this work, we have expanded the range of validity for Eq. \eqref{eq: ssq_mfg_closed_form} to a much larger interval given by $\Theta_\epsilon$. Although it may be difficult to fully characterize $\Theta_\epsilon$ for a fixed $\epsilon>0$, we can use Taylor's approximation of the exponential function to obtain a close enough subset. Accurately characterizing $\Theta_\epsilon$ is crucial to obtaining the correct tail bounds.  Our finite-time characterization of the queue-length distribution, as shown in Eq. \eqref{eq: ssq_mgf_recursiveeq}, allows us to achieve this expanded range of validity.

The proof of Lemma \ref{lem: ssq_mgf_steady_state} follows by showing that, for any $\theta \in \Theta_\epsilon$, the right-hand side (RHS) in Eq. \eqref{eq: ssq_mgf_recursiveeq} converges to a finite value. The details of the proof of Lemma \ref{lem: ssq_mgf_steady_state} are presented in  Appendix \ref{app: ssq}.

\begin{lemma}
    \label{lem: mgf_to_prob_bound}
    Suppose $X$ is a non-negative random variable, and the MGF of $X$ satisfies the inequality $\EE[e^{\theta X}] \leq \frac{1}{1- \theta/\lambda}, \forall \theta \in (0,\lambda)$,
    where $\lambda>0$. Then, for any $x > \frac{1}{\lambda}$, we have 
    \begin{align*}
        \PP(X>x) \leq e\lambda x e^{- \lambda x}.
    \end{align*}
\end{lemma}

The proof of Lemma \ref{lem: mgf_to_prob_bound} is provided in Appendix \ref{app: essential_lemmas}.

% \begin{lemma}
%     \label{lem: mgf_to_prob_bound}
%     Suppose $X$ is a non-negative random variable, and the MGF of $X$ satisfies the inequality $\EE[e^{\theta X}] \leq \frac{1}{1- \theta/\lambda}, \forall \theta \in (0,\lambda)$,
%     where $\lambda>0$. Then, for any $x > \frac{1}{\lambda}$, we have 
%     \begin{align*}
%         \PP(X>x) \leq e\lambda x e^{- \lambda x}.
%     \end{align*}
% \end{lemma}

% Lemma \ref{lem: mgf_to_prob_bound} simply uses the Markov's inequality to bound the tail probability. the proof of Lemma \ref{lem: mgf_to_prob_bound} is provided in  Appendix \ref{app: essential_lemmas}.

\begin{proof}[Proof of Theorem \ref{thm: ssq_tail_bound}]
By equating the drift of the queue length $q(t)$ to zero in steady state, we prove the following claim.

\begin{claim} 
\label{claim: ssq_uterm_bound}
For any $\theta \in \RR$, 
$1- \EE_{\pi_\epsilon}[e^{-\theta \epsilon u}] \leq \epsilon^2 \theta \mu_\epsilon.$
\end{claim}
Next, using Taylor's expansion of the exponential function we obtain the following claim.
\begin{claim}
\label{claim: ssq_aminussterm_bound}
    For any $\theta \in \Big(0, \frac{\theta_\epsilon}{(1+ \kappa_1 \epsilon )}\Big)$, $1- \EE \big[e^{\epsilon\theta(a-s)}\big] \geq  
        \theta\epsilon^2 \mu_\epsilon \Big(1 -  \frac{\theta}{\theta_\epsilon} \big(1+\kappa_1 \epsilon\big) \Big).$
\end{claim}
The proofs of Claims \ref{claim: ssq_uterm_bound} and \ref{claim: ssq_aminussterm_bound} are provided in  Appendix \ref{app: ssq}.
Next, we use the Lemma \ref{lem: ssq_mgf_steady_state}. Note that, by the assumption $\EE[e^{\theta \epsilon q(0)}] < \infty$ for $\theta<\theta_\epsilon$, we have that $\Big(0, \frac{\theta_\epsilon}{(1+ \kappa_1 \epsilon )}\Big) \subseteq \Theta_\epsilon$. Then, by Lemma \ref{lem: ssq_mgf_steady_state} and the above mentioned claims, we have 
\begin{align*}
         \EE \big[e^{\epsilon\theta q}\big]\leq \Big( 1 - \frac{\theta}{\theta_\epsilon} \Big(1+ \kappa_1 \epsilon \Big) \Big)^{-1}, \quad \forall \theta \in \Big(0, \frac{\theta_\epsilon}{(1+ \kappa_1 \epsilon )}\Big).
\end{align*}
Afterwards, the result in Theorem \ref{thm: ssq_tail_bound} follows simply by using Markov's inequality as shown in Lemma \ref{lem: mgf_to_prob_bound}.% using Lemma \ref{lem: mgf_to_prob_bound}.
\end{proof}

% \begin{remark}
%  There exists a constant $c_1$, independent of $\epsilon$, such that for every $\epsilon>0$, 
%  \begin{align*}
%      \theta_\epsilon > \frac{\mu_\epsilon}{2\sigma_\epsilon^2}(1-c_1\epsilon)
%  \end{align*}
% \end{remark}

% For any random variable $X$ such that $|X| \leq A$ almost surely, we have that
% \begin{align*}
%     \EE[e^{\theta \epsilon X}] &\leq 1 + \theta \epsilon \EE[X] + \frac{1}{2} \epsilon^2 \theta^2 \EE[X^2] + \epsilon^3 \theta^3 \EE[X^3 e^{\theta \epsilon X}]\\
%     & \leq 1 + \theta \epsilon \EE[X] + \frac{1}{2} \epsilon^2 \theta^2 \EE[X^2] + \epsilon^3 \theta^3 A^3 e^{\theta \epsilon A}\\
%     & = 1 - \epsilon^2  \theta \mu_\epsilon +\frac{1}{2} \epsilon^2 \theta^2 \sigma^2_\epsilon + \epsilon^3 \theta^3 A^3 e^{\theta \epsilon A}\\
% \end{align*}
% Then, $ \EE[e^{\theta \epsilon X}] < 1$ holds for any $\theta$ such that 
% \begin{align*}
%     - \mu_\epsilon +\frac{1}{2} \theta \sigma^2_\epsilon + \theta \epsilon^2 A^3 e^{\theta \epsilon A}   < 0. 
% \end{align*}
% \color{black}

\section{$M/M/n$ system}
\label{sec: mmn}

\subsection{Model}
\label{sec: mmn_model}

An $M/M/n$ system (also known as Erlang-C) is a multi-server queue with $n$ servers, a single queue, and follows a continuous-time first-come-first-serve service discipline. Customer arrivals follow a Poisson process with rate $\lambda_n$, and each server has an exponentially distributed service time with constant service rate $\mu$. The rate $\mu$ is independent of the system size; thus, the subscript $n$ is omitted. 
% Additionally, we assume that a waiting customer's patience is exponentially distributed, and its rate parameter, denoted as $\alpha_n$, scales with the system size. Notably, when $\alpha_n = 0$, a waiting customer has infinite patience.

% Whenever a server finishes a job, it pick a new job at the head of the queue, and if the queue is empty is free at the time of arrival, the customer is immediately assigned to that server; otherwise, the customer joins the end of the queue.
To represent the system's dynamics, we use $q_n(t)$ to denote the number of customers in the system at time $t$. We use $w_n(t) = [q_n(t) -n]^+$ to denote the number of customers waiting in the queue, and $r_n(t) = [n-q_n(t)]^+$ for the number of idle servers at time $t$. 
% The model underlying the system is a Continuous-Time Markov Chain (CTMC), with the transition diagram presented in Figure \ref{fig: mmn_queue}.

% \begin{figure}[ht!]
% \centering
% \begin{tikzpicture}[->, auto, thick, node distance=2cm]
% \tikzstyle{every state}=[fill=white,draw=black,thick,text=black,scale=0.7]
% \node[state]    (0)               {$0$};
% \node[state]    (1)[right of=0]   {$1$};
% \node[state]    (2)[right of=1]   {$2$};
% \node (3)[right of=2,node distance=1cm] {$\dots$};
% \node[state]    (4)[right of=3,node distance=1.4cm]   {$n-2$};
% \node[state]    (5)[right of=4]   {$n-1$};
% \node[state,scale =1.2]    (6)[right of=5]   {$n$};
% \node[state]    (7)[right of=6]   {$n+1$};
% \node    (8)[right of=7,node distance=1cm]   {$\dots$};
% \path
% (0) edge[bend left] node{$\lambda_n$} (1)
% (1) edge[bend left] node{$\lambda_n$} (2)
%     edge[bend left] node{$\mu$}   (0)
% (2) edge[bend left] node{$2\mu$}  (1)
% (4) edge[bend left] node{$\lambda_n$} (5)
% (5) edge[bend left] node{$(n-1)\mu$} (4)
%     edge[bend left] node{$\lambda_n$} (6)
% (6) edge[bend left] node{$n\mu$}   (5)
%     edge[bend left] node{$\lambda_n$} (7)
% (7) edge[bend left] node{$n\mu$}   (6);
% \end{tikzpicture}
% \caption{Transition diagram for an M/M/n queue \label{fig: mmn_queue}}
% \end{figure}

% When the abandonment rate is zero, i.e., $\alpha_n=0$, 
The queue-length process $\{q_n(t)\}_{t\geq 0}$ is a CTMC, and it is well known that it is stable (positive recurrent) if $\lambda_n<n\mu$. To be consequent with the previous sections, we assume that $\lambda_n =  n\mu \left(1-\epsilon_n\right)$. Further, a stationary distribution $\pi_n$ exists under this stability condition. Similarly to previous sections, we drop the index $t$ to denote steady-state variables. 
%The queueing system being studied is referred to as the Erlang-C model. In this scenario, the system is considered stable, and the underlying CTMC is positively recurrent if the arrival rate satisfies $\lambda_n < n \mu$. For simplicity of the notation, we assume that $\lambda_n =  n\mu \left(1-\epsilon_n\right)$. 
% On the other hand, when the abandonment rate is strictly greater than zero, the system is referred to as the Erlang-A model, and the underlying CTMC is positively recurrent due to the abandonment effect. We assume that the system is stable, i.e., either (i) $\alpha_n > 0$ or (ii) $\alpha_n =0$ and $\lambda_n < n\mu$. 

%Under the stability conditions, a stationary distribution $\pi_n$ exists for the underlying CTMC.
%For a given value of $n$, we drop the time index $t$ to denote the steady state random variable. We define $q_n$ as the steady-state number of customers in the system, $w_n$ as the steady-state number of waiting customers, and $r_n$ as the steady-state number of idle servers.

Similarly to the JSQ system, to study the Many-Server-HT asymptotics of this system, we define the load  $\rho_n:=\frac{\lambda_n}{n\mu} = 1- \epsilon_n$ and consider $\epsilon_n = cn^{-\alpha}$, where $c>0$ is a constant, and $\alpha>0$ a parameter (as in the JSQ system). In this case, we are interested in $\alpha\in(0,1]$. 

As explained in the introduction, the steady-state dynamics of the system vary with the value of $\alpha$. We study three regimes: (i) \textit{Sub-HW}, where $\alpha  \in \Big(0, \frac{1}{2}\Big)$, (ii)  \textit{HW}, where $\alpha  =\frac{1}{2}$, and (iii) \textit{Super-HW} where $\alpha  \in \Big(\frac{1}{2}, \infty\Big)$. %All three regimes belong to the class of many-server heavy traffic regimes, as the heavy traffic parameter $\epsilon_n \rightarrow 0$ as $n \rightarrow \infty$.  

 % In order to satisfy the stability conditions, we need that either (i) $\alpha_n>0$ and $\epsilon_n \in (-\infty,1)$, or (ii) $\alpha_n =0$ and $\epsilon_n \in (0,1)$. 

% Our objective is to characterize 

% In this work, we focus on an asymptotic regime where the arrival rate approaches the total service rate as the system size grows infinitely large. Specifically, we assume that the arrival rate $\lambda_n$ can be expressed as , where $n\epsilon_n$ is a value between 0 and $n$ and $n\epsilon_n/n$ approaches 0 as $n$ tends to infinity.

\subsection{Results on $w_n$ and $r_n$}
\label{sec: mmn_results}

In this section, we provide separate tail bounds on the number of idle servers and the number of waiting customers because the scaling parameter is different. In Theorem \ref{thm: mmn_idleservers} we present our results for the number of idle servers $r_n$.

\begin{theorem}\label{thm: mmn_idleservers}
    Consider the M/M/n queue as given in Section \ref{sec: mmn_model} with $\lambda_n = n\mu(1-\epsilon_n)$. Let $\tilde r_n := r_n - n \epsilon_n$, and $\eta_n := \frac{1}{\sqrt{n(1-\epsilon_n)}}$.
    \begin{enumerate}[label=(\alph*), ref=\ref{thm: mmn_idleservers}.\alph*]
        \item \label{thm: mmn_idleservers_supHW} Super-HW regime: Suppose $\epsilon_n = c n^{-\alpha}$ with $\alpha > \frac{1}{2}$. Then, %for $n$ large enough such that 
        if $n^{2\alpha -1} > 4c^2$, we have 
        \begin{align}
        \label{eq: mmn_idleser_probbound_supHW}
            \PP(r_n>0) \leq 4e\pi c n^{-\alpha + \frac{1}{2}}, && \lim_{n\rightarrow \infty}\PP(r_n>0) =0,
        \end{align}
        and there exist a constant $\kappa_1$, independent of $n$, such that,
        \begin{align*}
            \PP\big( \eta_n \tilde r_n > x  \big) \leq \frac{\kappa_1 c}{n^{\alpha - \frac{1}{2}}} e^{-\frac{1}{2} x^2}.
        \end{align*}
        \item \label{thm: mmn_idleservers_HW} HW regime: Suppose $\epsilon_n = c n^{-\frac{1}{2}}$. Then, 
        \begin{align}\label{eq: mmn_HW_probr>0}
        \lim_{n\rightarrow \infty} \PP(r_n>0)  = \frac{\sqrt{2\pi}c \exp\left( \frac{c^2}{2}\right) \Phi \left(c\right)}{1+\sqrt{2\pi}c \exp\left( \frac{c^2}{2}\right) \Phi \left(c\right)},
        \end{align}
        and there exist constant $\kappa_2$, such that 
        \begin{align*}
            \PP\big( \eta_n \tilde r_n > x  \big) \leq \kappa_2 e^{-\frac{1}{2} x^2}.
        \end{align*}

        \item \label{thm: mmn_idleservers_subHW} Sub-HW regime: Suppose $\epsilon_n = cn^{-\alpha}$ with $\alpha \in \big(0,\frac{1}{2}\big)$. Then, there exists constant $\kappa_3$ such that
        % In this case, we have that for all $\theta\in \mathbb R$, 
        % \begin{align*}
        %     \EE\Big[ \exp\Big( \frac{1}{\sqrt{n}} \theta \tilde r_n \Big)\Big| r_n > 0\Big] \rightarrow e^{\frac{t^2}{2}}, 
        % \end{align*}
        % and there exists $\tilde \kappa_1$
        \begin{align*}
            \PP(r_n >0) \geq 1-\frac{\kappa_3}{c}  n^{\alpha - \frac{1}{2}}e^{-cn^{\frac{1}{2}-\alpha}}, && \lim_{n\rightarrow \infty}\PP(r_n >0) =1,
        \end{align*}
        and we have
        \begin{align*}
            \PP(\eta_n \tilde r_n  >  x) \leq e^{-\frac{1}{2} x^2}, && \PP(\eta_n \tilde r_n  <  -x) \leq e^{ -\frac{1}{2}x^2 + 2e\epsilon_n x^2}.
        \end{align*}
    \end{enumerate}
\end{theorem}

Theorem \ref{thm: mmn_idleservers} provides the tail bounds on the number of idle servers in the $M/M/n$ queue in each of the three regimes considered in the paper. 
\begin{remark}
    Note that we do not provide a LD result for $r_n$, as we did in the case of SSQ and JSQ systems, as the number of idle servers $r_n$ is always bounded by $n$. As such, for the pre-limit system (i.e., when the system size $n$ is finite), we have that 
    \begin{align*}
        \PP(\eta_n \tilde r_n > x ) =0, \ \forall x > \sqrt{n \rho_n}.
    \end{align*}
\end{remark}

Next, we provide the results for the number of waiting customers $w_n$.

\begin{corollary}\label{cor: mmn_waitingcustomers}
Consider the M/M/n queue as given in Section \ref{sec: mmn_model} with $\lambda_n = n\mu(1-\epsilon_n)$. Then,
for any $\{\epsilon_n\}_{n\geq0}$ such that $\epsilon_n \rightarrow 0$ as $n\rightarrow \infty$, and for any $\theta<1$, we have 
    \begin{align*}
        \lim_{n\rightarrow \infty} \EE \big[e^{\theta \epsilon_n w_n}\big| w_n>0 \big] = \frac{1}{1-\theta}, && [\epsilon_n w_n|w_n>0] \stackrel{d}{ \rightarrow} \Upsilon,
    \end{align*}
    where $\Upsilon$ is an exponential random variable with mean $1$. Further, for any $x>\epsilon_n$, we have 
    \begin{align*}
       \frac{1}{(1-\epsilon_n)^2} e^{-\theta_n x} \leq  \PP(\epsilon_n w_n \geq x | w_n >0 ) \leq \frac{1}{1-\epsilon_n} e^{-\theta_n x},
    \end{align*}
    where $\theta_n = \frac{1}{\epsilon_n} \log \frac{1}{1-\epsilon_n}$. As a consequence, we get the LD result
    \begin{align}\label{eq: mmn_wn_largedeviation}
        \lim_{x \rightarrow \infty} \frac{1}{x} \PP(\epsilon_n w_n \geq x ) = \lim_{x \rightarrow \infty} \frac{1}{x} \PP(\epsilon_n w_n \geq x | w_n >0 ) = \theta_n.
    \end{align}
\end{corollary}

Corollary \ref{cor: mmn_waitingcustomers} provides the limiting distribution of the steady state number of waiting customers in the system. For any $n$ and $\epsilon_n$, the distribution of $[w_n|w_n>0]$ can be exactly characterized. As will be depicted in Lemma \ref{lem: mmn_steady_state_mgf}, it turns out that $[w_n|w_n>0]$ follows a geometric distribution with parameter $\epsilon_n$. Such a result can also be derived by simply solving for the steady-state distribution of an $M/M/n$ queue.  
Furthermore, as we can exactly characterize the distribution of $[w_n|w_n>0]$, there is no need for using Markov's inequality to get a tail bound on the number of waiting customers. Hence, we do not obtain the pre-limit nor the pre-exponent error in the case of the number of waiting customers.

\begin{remark}\label{rmk: MMn-def-p}
It is worth noting that although the conditional distribution $[w_n|w_n>0]$ follows a geometric distribution, it is possible for the probability $\PP(w_n>0)$ to approach zero. This can be observed by examining the bounds on $\PP(r_n>0)$ provided in Theorem \ref{thm: mmn_idleservers}. Specifically, in the Sub-HW regime, we have
\begin{align*}
\PP(w_n>0) \leq 1 - \PP(r_n>0) \leq \frac{\kappa_3}{c} n^{\alpha - \frac{1}{2}}e^{-cn^{\frac{1}{2}-\alpha}}.
\end{align*}
 It is well-known that $\PP(w_n>0)$ decreases to zero in this regime. However, we additionally characterize the rate at which this convergence occurs. 
\end{remark}

\begin{remark}
The results presented in Theorem \ref{thm: mmn_idleservers} and Corollary \ref{cor: mmn_waitingcustomers} only consider Many-Server-HT regimes. There, the HT parameter $\epsilon_n$ approaches zero as the system size $n$ grows. One can also consider a conventional HT regime, where the system size $n$ is constant and the HT parameter $\epsilon \rightarrow 0$ (using the notation $\lambda_n = n\mu(1-\epsilon)$) independent of $n$. In this case, the tail bound for the relevant quantities closely aligns with the tail-bound results obtained in the Super-HW regime. By employing similar calculations as, it can be shown that the number of waiting customers exhibits the same tail bounds as presented in Corollary \ref{cor: mmn_waitingcustomers} after replacing $\epsilon_n$ by $\epsilon$. Furthermore, the probability $\PP(r_n>0)$ is on the order of $O(\epsilon)$. Moreover, in conventional HT, there is no need for a tail bound on $r_n$ as $r_n \leq n$ and $n$ is a constant.
\end{remark}

 \subsubsection{Steady-state distribution of $w_n$ and $r_n$}
In this section, we provide the complete characterization of the steady state distribution of $\{q_n(t)\}_{t\geq 0}$. Specifically, we provide the MGF of the steady-state number of waiting customers ($w_n$) and idle servers ($r_n$).

\begin{lemma}
\label{lem: mmn_steady_state_mgf}
    Consider the M/M/n queue as given in Section \ref{sec: mmn_model} and suppose $\lambda_n = n\mu(1-\epsilon_n)$. Define
    \begin{align*}
        G_n(t) &:= \exp\Big( - n\epsilon_n  t -n(1-\epsilon_n) (e^{- t} + t -1)\Big).
        % G_n(t)  &:= \exp\Big( \frac{n\mu \epsilon_n}{\alpha_n}  t -\frac{n\mu}{\alpha_n} (e^{-t} + t -1)\Big).
    \end{align*}
    % Also, for any function $G(\cdot)$, let $T \big( G (\theta) \big) := G^{-1}(\theta)\left(\int_{-\infty}^0 G(t) dt \right)^{-1} \int_{-\infty}^\theta G(t) dt$. 
     Then, we have the following results. 
    \begin{enumerate}[label=(\alph*), ref=\ref{lem: mmn_steady_state_mgf}.\alph*]
    \item \label{lem: mmn_steady_state_mgf_a} For any $\theta \in \mathbb R$, we have 
    \begin{align*}
        \EE \left[e^{\theta  r_n} \Big| r_n>0 \right] = G_n^{-1}(\theta)\left(\int_{-\infty}^0 G_n dt \right)^{-1} \int_{-\infty}^\theta G_n dt.
    \end{align*}
    % \begin{align*}
    %     \EE \left[e^{\theta \eta_n r_n} | r_n>0 \right] =  \tilde  G_n^{-1}(\theta) \left(\int_{-\infty}^0 G_n(t) dt \right)^{-1} \int_{-\infty}^\theta G_n(t) dt.
    % \end{align*}
    \item \label{lem: mmn_steady_state_mgf_b} For any $\theta<  \log  \frac{1}{1-\epsilon_n} $, we have 
    \begin{align*}
        \EE \left[e^{\theta w_n}\Big| w_n>0 \right]  = \frac{1}{1 - \frac{1}{\epsilon_n} \left(1-e^{-\theta  }\right)}. 
    \end{align*}
    % \item[3.] \label{lem: mmn_steady_state_mgf_c} For $\alpha_n > 0$, we have that for any $\theta \in \RR$,
    % \begin{align*}
    %     \EE \left[e^{\theta  w_n} \Big| w_n>0 \right] = T \circ G_n (\theta).
    % \end{align*}
    % Further, we have $\PP(q_n=n)=\left(1+  \frac{\lambda_n }{ \alpha_n}\int_{-\infty}^0 G_n(t) dt + n \int_{-\infty}^0 G_n(t) dt   \right)^{-1}$, and
    % \begin{align*}
    %     \PP(w_n > 0) &=  \frac{\lambda_n }{ \alpha_n} \PP(q_n=n)\int_{-\infty}^0 G_n(t) dt,\\
    %      \PP(r_n>0) &= n \PP(q_n=n)\int_{-\infty}^0 G_n(t) dt.
    % \end{align*}
    \end{enumerate}
    Further, we have $\PP(q_n=n)=\left(  \frac{1}{\epsilon_n} + n\int_{-\infty}^0 G_n(t) dt   \right)^{-1}$, and
    \begin{align*}
        \PP(w_n>0) = \frac{1-\epsilon_n}{\epsilon_n} \PP(q_n=n), && \PP(r_n>0) = n\PP(q_n=n)\int_{-\infty}^0 G_n(t) dt.
    \end{align*}
\end{lemma}

Lemma \ref{lem: mmn_steady_state_mgf} provides the steady state distribution of $r_n$ and $w_n$ for any value of $n$ and $\epsilon_n$. The results in Theorem \ref{thm: mmn_idleservers} are derived by using the result in Lemma \ref{lem: mmn_steady_state_mgf_a}, Markov's inequality. Lemma \ref{lem: mmn_steady_state_mgf_a} indicates that the distribution of $[r_n|r_n >0]$ closely resembles that of a truncated normal random variable. By replacing $G_n(t)$ with $\tilde G_n(t) := \exp \big( - n\eta_n\epsilon_n t - \frac{1}{2} t^2 \big)$, we have that the function
\begin{align*}
\tilde G_n^{-1}(\theta)\left(\int_{-\infty}^0 \tilde G_n dt \right)^{-1} \int_{-\infty}^\theta \tilde G_n dt
\end{align*}
defines the MGF of a truncated normal random variable, specifically $[Y_n|Y_n>0]$, where $Y_n$ follows a normal distribution with mean $n\eta_n \epsilon_n$ and variance $1$.

Intuitively, for large values of $n$, we can approximate $G_n(\eta_n t)$ by $\tilde G_n( t)$ because $\frac{1}{\eta_n^2}\big(e^{-\eta_n t} + \eta_n t -1\big) \approx \frac{t^2}{2}$, where $\eta_n = \frac{1}{\sqrt{n \rho_n}}$. Consequently, we observe that, for large $n$, the distribution of $[\eta_n r_n |r_n>0]$ closely matches the distribution of $[ Y_n|Y_n>0]$, where $Y_n$ is defined above. Consequently, for large values of $n$, $[\eta_n \tilde r_n |r_n>0]$ closely matches the distribution of a truncated standard normal random variable. In the proof of Theorem \ref{thm: mmn_idleservers}, use this idea to establish tail bounds on $r_n$ in each of the three regimes. 

% Intuitively, for large values of $n$, we have $G_n(\eta_n t) \approx \tilde G_n(\eta_n t)$, where $\eta_n = \frac{1}{\sqrt{n \rho_n}}$. As such, we get that, for large $n$, the distribution of $[\eta_n r_n|r_n>0]$ closed matches with the distribution of $[\eta_n Y_n|Y_n>0]$, where $\eta_n Y_n$ is normally distributed with mean $n\eta_n \epsilon_n$ and variance $1$. 

The result in Corollary \ref{cor: mmn_waitingcustomers} immediately follows from Lemma \ref{lem: mmn_steady_state_mgf_b} just be observing that the MGF of $[w_n|w_n>0]$ matches with that of a geometric random variable with parameter $\epsilon_n$. 
The mathematical details for results provided in this section (Section \ref{sec: mmn_results}) is provided in Appendix \ref{app: mmn}.

% \begin{remark}
% The quantity $w_n| w_n>0$ always converges to Exp$(1)$ as the arrival rate converges to the service rate. However, if might happen that $\PP(w_n>0)$ goes to zero, in which case the distribution of $w_n| w_n>0$ does not really matter.
% \end{remark}

% \begin{remark}
%     For $w_n| w_n>0$, the scaling is always $\epsilon_n = n\epsilon_n/n$.
% \end{remark}

\bibliographystyle{elsarticle-num}
\bibliography{references}

\begin{thebibliography}{10}
\expandafter\ifx\csname url\endcsname\relax
  \def\url#1{\texttt{#1}}\fi
\expandafter\ifx\csname urlprefix\endcsname\relax\def\urlprefix{URL }\fi
\expandafter\ifx\csname href\endcsname\relax
  \def\href#1#2{#2} \def\path#1{#1}\fi

\bibitem{GupWal_NDS_JSQ}
V.~Gupta, N.~Walton, Load balancing in the {N}ondegenerate {S}lowdown {R}egime,
  Operations Research 67~(1) (2019) 281--294.

\bibitem{HurMag_2022_alpha_continuous}
D.~Hurtado-Lange, S.~T. Maguluri, A load balancing system in the many-server
  heavy-traffic asymptotics, Queueing Systems 101~(3-4) (2022) 353--391.

\bibitem{hajek_drift}
B.~Hajek, Hitting-time and occupation-time bounds implied by drift analysis
  with applications, Advances in Applied Probability (1982) 502--525.

\bibitem{atilla}
A.~Eryilmaz, R.~Srikant, Asymptotically tight steady-state queue length bounds
  implied by drift conditions, Queueing Systems 72~(3-4) (2012) 311--359.

\bibitem{MagSri_SSY16_Switch}
S.~T. Maguluri, R.~Srikant, \href{http://dx.doi.org/10.1214/15-SSY193}{Heavy
  traffic queue length behavior in a switch under the {M}ax{W}eight algorithm},
  Stochastic Systems 6~(1) (2016) 211--250.
\newline\urlprefix\url{http://dx.doi.org/10.1214/15-SSY193}

\bibitem{hurtado2020transform}
D.~Hurtado-Lange, S.~T. Maguluri, Transform methods for heavy-traffic analysis,
  Stochastic Systems 10~(4) (2020) 275--309.

\bibitem{braverman2017stein}
A.~Braverman, J.~Dai, J.~Feng, Stein's method for steady-state diffusion
  approximations: {A}n introduction through the {E}rlang-{A} and {E}rlang-{C}
  models, Stochastic Systems 6~(2) (2017) 301--366.

\bibitem{kingman1962_brownian}
J.~Kingman, On queues in heavy traffic, Journal of the Royal Statistical
  Society. Series B (Methodological) (1962) 383--392.

\bibitem{harrison1998heavy}
J.~M. Harrison, Heavy traffic analysis of a system with parallel servers:
  asymptotic optimality of discrete-review policies, Annals of applied
  probability (1998) 822--848.

\bibitem{Williams_CRP}
R.~Williams, On dynamic scheduling of a parallel server system with complete
  resource pooling, Fields Institute Communications 28~(49-71) (2000) 5--1.

\bibitem{harrison2004dynamic}
J.~M. Harrison, A.~Zeevi, Dynamic scheduling of a multiclass queue in the
  halfin-whitt heavy traffic regime, Operations Research 52~(2) (2004)
  243--257.

\bibitem{gamarnik2006validity}
D.~Gamarnik, A.~Zeevi, Validity of heavy traffic steady-state approximations in
  {G}eneralized {J}ackson {N}etworks, The Annals of Applied Probability (2006)
  56--90.

\bibitem{braverman_steins_2017}
A.~Braverman, J.~Dai, \href{https://doi.org/10.1214/16-AAP1211}{Stein’s
  method for steady-state diffusion approximations of {M/Ph/n+M} systems}, The
  Annals of Applied Probability 27~(1) (Feb. 2017).
\newblock \href {https://doi.org/10.1214/16-AAP1211}
  {\path{doi:10.1214/16-AAP1211}}.
\newline\urlprefix\url{https://doi.org/10.1214/16-AAP1211}

\bibitem{braverman2017stein2}
A.~Braverman, J.~Dai, Stein's method for steady-state diffusion approximations
  of {M/Ph/n+ M} systems, The Annals of Applied Probability 27~(1) (2017)
  550--581.

\bibitem{Lei_steinHT_SIGM17}
L.~Ying, \href{http://doi.acm.org/10.1145/3084449}{Stein's method for mean
  field approximations in light and heavy traffic regimes}, Proc. ACM Meas.
  Anal. Comput. Syst. 1~(1) (2017) 12:1--12:27.
\newblock \href {https://doi.org/10.1145/3084449} {\path{doi:10.1145/3084449}}.
\newline\urlprefix\url{http://doi.acm.org/10.1145/3084449}

\bibitem{Walton_SteinHT}
R.~Gaunt, N.~Walton, Stein's method for the single server queue in heavy
  traffic, Statistics \& Probability Letters 156 (2020) 108566.

\bibitem{braverman_BAR}
A.~Braverman, J.~Dai, M.~Miyazawa, Heavy traffic approximation for the
  stationary distribution of a {G}eneralized {J}ackson {N}etwork: The {BAR}
  approach, Stochastic Systems 7~(1) (2017) 143--196.

\bibitem{braverman_BAR_23}
A.~Braverman, J.~Dai, M.~Miyazawa, The bar-approach for multiclass queueing
  networks with sbp service policies, arXiv preprint arXiv:2302.05791 (2023).

\bibitem{DaiGlynnXu_BAR_23}
J.~Dai, P.~Glynn, Y.~Xu, Asymptotic product-form steady-state for generalized
  jackson networks in multi-scale heavy traffic, arXiv preprint
  arXiv:2304.01499 (2023).

\bibitem{JhunZubMag_2022_JSQ-A}
P.~R. Jhunjhunwala, M.~Zubeldia, S.~T. Maguluri, Join-the-shortest queue with
  abandonment: Critically loaded and heavily overloaded regimes, arXiv preprint
  arXiv:2211.15050 (2022).

\bibitem{vallender1974calculation}
S.~Vallender, Calculation of the wasserstein distance between probability
  distributions on the line, Theory of Probability \& Its Applications 18~(4)
  (1974) 784--786.

\bibitem{mukherjee2018universality}
D.~Mukherjee, S.~C. Borst, J.~S. Van~Leeuwaarden, P.~A. Whiting, Universality
  of power-of-d load balancing in many-server systems, Stochastic Systems 8~(4)
  (2018) 265--292.

\bibitem{mitzenmacher_po2}
M.~Mitzenmacher, Load balancing and density dependent jump {M}arkov processes,
  in: focs, IEEE, 1996, p. 213.

\bibitem{mitzenmacher_po2_2}
M.~Mitzenmacher, The power of two choices in randomized load balancing, IEEE
  Transactions on Parallel and Distributed Systems 12~(10) (2001) 1094--1104.

\bibitem{stolyar_JIQ}
A.~Stolyar, Pull-based load distribution among heterogeneous parallel servers:
  The case of multiple routers, Queueing Systems 85~(1-2) (2017) 31--65.

\bibitem{liu2018simple}
X.~Liu, L.~Ying, A simple steady-state analysis of load balancing algorithms in
  the sub-{H}alfin-{W}hitt regime, ACM SIGMETRICS Performance Evaluation Review
  46~(2) (2019) 15--17.

\bibitem{varma2022powerofd}
S.~M. Varma, F.~Castro, S.~T. Maguluri, Power-of-$d$ choices load balancing in
  the sub-halfin whitt regime (2022).
\newblock \href {http://arxiv.org/abs/2208.07539} {\path{arXiv:2208.07539}}.

\bibitem{bhamidi2022near}
S.~Bhamidi, A.~Budhiraja, M.~Dewaskar, Near equilibrium fluctuations for
  supermarket models with growing choices, The Annals of Applied Probability
  32~(3) (2022) 2083--2138.

\bibitem{BanMukh_JSQ_tail_asymptotics}
S.~Banerjee, D.~Mukherjee, Join-the-shortest queue diffusion limit in
  {H}alfin--{W}hitt regime: {T}ail asymptotics and scaling of extrema, The
  Annals of Applied Probability 29~(2) (2019) 1262--1309.

\bibitem{BanMukh_JSQ_sensitivity}
S.~Banerjee, D.~Mukherjee, Join-the-shortest queue diffusion limit in
  {H}alfin--{W}hitt regime: {S}ensitivity on the heavy-traffic parameter, The
  Annals of Applied Probability 30~(1) (2020) 80--144.

\bibitem{Braverman2020_jsq}
A.~Braverman, Steady-state analysis of the join-the-shortest-queue model in the
  {H}alfin--{W}hitt regime, Mathematics of Operations Research (2020).

\bibitem{Gamarnik_JSQ}
P.~Eschenfeldt, D.~Gamarnik, Join the shortest queue with many servers. the
  heavy-traffic asymptotics, Mathematics of Operations Research 43~(3) (2018)
  867--886.

\bibitem{HalfinWhitt_Regime}
S.~Halfin, W.~Whitt, Heavy-traffic limits for queues with many exponential
  servers, Operations research 29~(3) (1981) 567--588.

\bibitem{liu2021universal}
X.~Liu, L.~Ying, Universal scaling of distributed queues under load balancing
  in the super-halfin-whitt regime, IEEE/ACM Transactions on Networking 30~(1)
  (2021) 190--201.

\bibitem{zhao2023manyserver}
Z.~Zhao, S.~Banerjee, D.~Mukherjee, Many-server asymptotics for
  join-the-shortest queue in the super-halfin-whitt scaling window (2023).
\newblock \href {http://arxiv.org/abs/2106.00121} {\path{arXiv:2106.00121}}.

\bibitem{foschini1978basic}
G.~Foschini, J.~Salz, A basic dynamic routing problem and diffusion, IEEE
  Transactions on Communications 26~(3) (1978) 320--327.

\bibitem{Williams_state_space}
R.~Williams, Diffusion approximations for open multiclass queueing networks:
  Sufficient conditions involving state space collapse, Queueing Systems Theory
  and Applications (1998) 27 -- 88.

\bibitem{harrison1999heavy}
J.~M. Harrison, M.~J. L{\'o}pez, Heavy traffic resource pooling in
  parallel-server systems, Queueing systems 33~(4) (1999) 339--368.

\bibitem{gupta2019load}
V.~Gupta, N.~Walton, Load balancing in the nondegenerate slowdown regime,
  Operations Research 67~(1) (2019) 281--294.

\bibitem{ganesh2004big}
A.~J. Ganesh, Big queues, Springer Science \& Business Media, 2004.

\bibitem{glynn1994logarithmic}
P.~W. Glynn, W.~Whitt, Logarithmic asymptotics for steady-state tail
  probabilities in a single-server queue, Journal of Applied Probability 31~(A)
  (1994) 131--156.

\bibitem{puhalskii1995large}
A.~Puhalskii, Large deviation analysis of the single server queue, Queueing
  Systems 21 (1995) 5--66.

\bibitem{duffield1995large}
N.~G. Duffield, N.~O'connell, Large deviations and overflow probabilities for
  the general single-server queue, with applications, in: Mathematical
  Proceedings of the Cambridge Philosophical Society, Vol. 118, Cambridge
  University Press, 1995, pp. 363--374.

\bibitem{dupuis1995large}
P.~Dupuis, R.~S. Ellis, The large deviation principle for a general class of
  queueing systems. i, Transactions of the American Mathematical Society
  347~(8) (1995) 2689--2751.

\bibitem{ridder2005large}
A.~Ridder, A.~Shwartz, Large deviations without principle: Join the shortest
  queue, Mathematical Methods of Operations Research 62 (2005) 467--483.

\bibitem{foley2001join}
R.~D. Foley, D.~R. McDonald, Join the shortest queue: stability and exact
  asymptotics, Annals of Applied Probability (2001) 569--607.

\bibitem{puhalskii2007large}
A.~A. Puhalskii, A.~A. Vladimirov, A large deviation principle for join the
  shortest queue, Mathematics of Operations Research 32~(3) (2007) 700--710.

\bibitem{LinVen_09_LD}
X.~Lin, V.~Venkataramanan, On the large-deviations optimality of scheduling
  policies minimizing the drift of a lyapunov function, in: 2009 47th Annual
  Allerton Conference on Communication, Control, and Computing (Allerton),
  IEEE, 2009, pp. 919--926.

\bibitem{venkataramanan2013queue}
V.~Venkataramanan, X.~Lin, On the queue-overflow probability of wireless
  systems: A new approach combining large deviations with lyapunov functions,
  IEEE transactions on information theory 59~(10) (2013) 6367--6392.

\bibitem{BanKanQian_18_LD-MW}
S.~Banerjee, Y.~Kanoria, P.~Qian, Large deviations optimal scheduling of closed
  queueing networks, arXiv preprint arXiv:1803.04959 (2018).

\bibitem{ZubMan_21_LD_GaussianInput}
M.~Zubeldia, M.~Mandjes, Large deviations for acyclic networks of queues with
  correlated gaussian inputs, Queueing Systems 98~(3-4) (2021) 333--371.

\bibitem{Man_07_LD-book}
M.~Mandjes, Large deviations for Gaussian queues: modeling communication
  networks, John Wiley \& Sons, 2007.

\bibitem{vershynin2018high}
R.~Vershynin, High-dimensional probability: An introduction with applications
  in data science, Vol.~47, Cambridge university press, 2018.

\bibitem{peres2015graphical}
Y.~Peres, K.~Talwar, U.~Wieder, Graphical balanced allocations and the (1+
  $\beta$)-choice process, Random Structures \& Algorithms 47~(4) (2015)
  760--775.

\bibitem{bansal2022power}
N.~Bansal, O.~N. Feldheim, The power of two choices in graphical allocation,
  in: Proceedings of the 54th Annual ACM SIGACT Symposium on Theory of
  Computing, 2022, pp. 52--63.

\bibitem{Bill86}
P.~Billingsley, Probability and Measure, 2nd Edition, John Wiley and Sons,
  1986.

\end{thebibliography}

\begin{appendix}

\section{Essential Lemmas and their proofs}
\label{app: essential_lemmas}

\begin{proof}[Proof of Lemma \ref{lem: mgf_to_prob_bound}]
The proof simply follows by using the Markov's inequality and then optimizing over the value of $\theta$. For any $\theta\in (0, 1/\lambda)$, we have
\begin{align}
\label{eq: markov_ineq}
    \PP(X>x) \leq e^{-\theta x} \EE[e^{\theta X}] \leq \frac{1}{1-\theta/\lambda } e^{-\theta x}.
\end{align}
As the above inequality is valid for any $\theta\in (0, \lambda)$, we can minimize the RHS in the above inequality over the range $(0, \lambda)$. Now, the derivative of the RHS equates to zero when $\theta = \theta^* := \lambda \Big(1- \frac{1}{\lambda x} \Big)$, where $\theta^* \in (0, \lambda)$ whenever $x > 1/\lambda$. Thus, for any $x> 1/\lambda$, by substituting $\theta = \theta^*$ in Eq. \eqref{eq: markov_ineq}, we get the result.
\end{proof}

\begin{lemma} 
\label{lem: steady_state_lyapunov}
Consider a Discrete Time Markov Chain $\{\mathbf x(t)\}_{t\geq 0}$, where $\mathbf x(t)\in \RR^d$ for all $t\geq 0$. Let the Markov chain be positive recurrent, with stationary distribution $\pi$, and $\mathbf x_\infty$ be a random variable that follows the distribution $\pi$.
Suppose $f: \RR^d \times \RR \rightarrow \RR_{+}$ is a Lyapunov function such that following conditions hold.
\begin{itemize}
    \item[C1.] There exists $\theta_0 >0$ such that $\forall \theta< \theta_0$, we have, 
    \begin{align*}
        \EE[f(\mathbf x(t),\theta)] < \infty, \ \ \forall t \geq 0.
    \end{align*}
    \item[C2.] There exists $\theta_1>0$ such that for all $\theta< \theta_1$, the function $f$ satisfies the negative drift condition, 
    \begin{align}
    \label{eq: neg_drift}
        \EE[\Delta f(\mathbf x(t+1),\theta) | \mathbf x(t)]  = -\gamma(\theta) f(\mathbf x(t),\theta) + \beta(\mathbf x(t),\theta),
    \end{align}
    where for all $\theta< \theta_1$, one of the following conditions is satisfied
    \begin{enumerate}
    \item There exists $\bar \beta(\theta)$ such that
    \begin{align*}
      \gamma(\theta) > 0, && \EE[\beta(\mathbf x(t),\theta)] < \bar \beta(\theta), && \lim_{t \rightarrow \infty}\EE[\beta(\mathbf x(t),\theta)] \ \text{ exists}. 
    \end{align*}
    \item There exists a function $b(\theta)<\infty$ such that 
        \begin{align*}
            \EE[f(\mathbf x(t), \theta)] \leq b(\theta), \ \forall t\geq 0.
        \end{align*}
    \end{enumerate}
\end{itemize}
Then, we have that for all $\theta < \min\{\theta_0, \theta_1\}$, 
    \begin{align*}
        \EE_\pi[f(\mathbf x_\infty,\theta)] = \frac{1}{\gamma(\theta)} \EE_\pi[\beta(\mathbf x_\infty,\theta)] < \infty. 
    \end{align*}
\end{lemma}

\begin{proof}[Proof of Lemma \ref{lem: steady_state_lyapunov}]
For the complete part of the proof, we assume that $\theta< \min\{\theta_0,\theta_1\}$. From, we Condition C1, we have that $ \EE[f(\mathbf x(t),\theta)]< \infty, \ \forall t\ge 0$. Thus, in Eq. \eqref{eq: neg_drift} in Condition C2, we can take expectation on both sides to get that,
\begin{align*}
    \EE[f(\mathbf x(t+1),\theta)] = (1-\gamma(\theta)) \EE[f(\mathbf x(t),\theta)] + \EE[\beta(\mathbf x(t),\theta)].
\end{align*}
Depending on the value of $\gamma(\theta)$, we consider two cases. 
\begin{itemize}
    \item \textit{Case 1: $\gamma(\theta)\geq 1$. } In this case, we have that $1-\gamma(\theta)< 0$. Then, as $f(\mathbf x(t),\theta) \geq 0$, from the previous equation, we have 
    \begin{align*}
       \EE[f(\mathbf x(t+1),\theta)] \leq \EE[\beta(\mathbf x(t),\theta)] \leq \bar \beta(\theta).
    \end{align*}
    As $\EE[f(\mathbf x(t),\theta)]$ is uniformly bounded by $\bar \beta$ for all $t\geq 0$,
    \begin{align*}
       \EE[f(\mathbf x_\infty,\theta)] = \limt \EE[f(\mathbf x(t+1),\theta)] \leq \bar \beta(\theta).
    \end{align*}
    Thus, under steady state, we can rewrite Eq. \eqref{eq: neg_drift} in Condition C2 as,
    \begin{align*}
         \EE[f(\mathbf x_\infty,\theta)] = (1-\gamma(\theta)) \EE[f(\mathbf x_\infty,\theta)] + \EE[\beta(\mathbf x_\infty,\theta)].
    \end{align*}
    This gives us the result. Note that the above result also holds when we directly have that $\EE[f(\mathbf x(t),\theta)]$ is uniformly bounded, as mentioned in Condition C2 part (2).
    
    \item \textit{Case 2: $\gamma(\theta)\in (0,1)$. } In this case, we can use Eq. \eqref{eq: neg_drift} in Condition C2 recursively, to get that 
    \begin{align}
    \label{eq: case2_recursion}
        \EE[f(\mathbf x(t+1),\theta)] &= (1-\gamma(\theta))^{t+1} \EE[f(\mathbf x(0),\theta)] \nonumber \\
        & \quad + \sum_{i=0}^t (1-\gamma(\theta))^{i} \EE[\beta(\mathbf x(t-i),\theta)].
    \end{align}
    Now, as $\lim_{t \rightarrow \infty}\EE[\beta(\mathbf x(t),\theta)]$ exists, and is equal to $\EE[\beta(\mathbf x_\infty,\theta)]$, we have that
    \begin{align*}
        \limt \sum_{i=0}^t (1-\gamma(\theta))^{i} \EE[\beta(\mathbf x(t-i),\theta)] & = \frac{1}{\gamma(\theta)}\EE[\beta(\mathbf x_\infty,\theta)],\\
        \limt (1-\gamma(\theta))^{t+1} \EE[f(\mathbf x(0),\theta)] &= 0.
    \end{align*}
    Substituting the above in Eq. \eqref{eq: case2_recursion} gives us the result.
\end{itemize}
\end{proof}
\section{Proof of results in Section \ref{sec: jsq_cont}}
\label{app: jsq_cont}

\begin{proof}[Proof of Proposition \ref{prop: jsq_cont_ssc}]
For any $i\in \{1,2,\dots,n\}$, we consider the following exponential Lyapunov functions
\begin{align*}
    V_{\perp i}(\mathbf x ;\theta) =e^{\theta \mathbf x_{\perp i}}= e^{\theta (x_i - \hat x)}, && U_{\perp i}(\mathbf x ;\theta) =e^{\theta \mathbf x_{\perp i}}= e^{-\theta (x_i - \hat x)},
\end{align*}
where $\hat x = \frac{1}{n} \sum_{i=1}^n x_i$. 

First, note that the total queue length at any time $t$ is less than the total number of arrivals till time $t$. And as the arrivals follow a Poisson process, the total number of arrivals till time $t$ is a Poisson random variable with parameter $\lambda_n t$. This gives us that
\begin{align*}
    \EE[V_{\perp i}(\q(t);\theta)] \leq \exp\Big( \lambda_n t \big(e^{\theta } -1\big)\Big) < \infty.
\end{align*}
Thus, $V_{\perp i}(\cdot)$ satisfies the Condition C1 of Lemma \ref{lem: steady_state_lyapunov}. Similarly, $U_{\perp i}(\cdot)$ also satisfies the Condition C1 of Lemma \ref{lem: steady_state_lyapunov}.

Suppose $G$ denotes the generator matrix of the underlying CTMC of the JSQ system. Then, the drift of $V_{\perp i}(\q(t);\theta)$, denoted by $\Delta V_{\perp i}(\q(t);\theta)$, is given by
\begin{align*}
    \EE\big[\Delta V_{\perp i}(\q(t);\theta) \big| \q(t) \big] := \sum_{\q'} G(\q',\q(t)) \big[V_{\perp i}(\q';\theta) - V_{\perp i}(\q(t);\theta)\big].
\end{align*}
Using this, we have 
\begin{align*}
   \EE\big[& \Delta V_{\perp i}(\q(t);\theta) \big| \q(t)  \big] \nonumber \\
   &= \lambda_n \left( e^{\theta q_{\perp i}(t) + \theta\big(1-\frac{1}{n} \big) } \mathds 1_{\{i=i^*\}} + e^{\theta q_{\perp i}(t) - \frac{\theta}{n} } \mathds 1_{\{i\neq i^*\}} - e^{\theta q_{\perp i}(t)} \right)\allowdisplaybreaks\nonumber\\
    & \ \ \ \ + \sum_{j=1}^n \mu \left( e^{\theta q_{\perp i}(t) - \theta\big(1-\frac{1}{n} \big) } \mathds 1_{\{j=i\}} + e^{\theta q_{\perp i}(t) + \frac{\theta}{n} } \mathds 1_{\{j\neq i\}} - e^{\theta q_{\perp i}(t)} \right)\mathds 1_{\{q_j>0\}}\allowdisplaybreaks\nonumber\\
    & = \lambda_n e^{\theta q_{\perp i}(t)} \left( e^{ \theta\big(1-\frac{1}{n} \big) } \mathds 1_{\{i=i^*\}} + e^{ - \frac{\theta}{n} } \mathds 1_{\{i\neq i^*\}} - 1 \right)\allowdisplaybreaks\nonumber\\
    & \ \ \ \ + \mu e^{\theta q_{\perp i}(t)} \sum_{j=1}^n  \left( e^{ - \theta\big(1-\frac{1}{n} \big) } \mathds 1_{\{j=i\}} + e^{ \frac{\theta}{n} } \mathds 1_{\{j\neq i\}} - 1 \right)\mathds 1_{\{q_j>0\}}\allowdisplaybreaks\nonumber\\
    & = \lambda_n e^{\theta q_{\perp i}(t)} \left( \big(e^{\theta} - 1\big)e^{-\frac{\theta}{n} } \mathds 1_{\{i=i^*\}} + e^{ - \frac{\theta}{n} } - 1 \right)\\
    & \quad + \mu e^{\theta q_{\perp i}(t)} \Big( \big(e^{-\theta} - 1\big)e^{\frac{\theta}{n} }\mathds 1_{\{q_i(t) >0\}} + \big(e^{ \frac{\theta}{n} } -1\big) \sum_{j}\mathds 1_{\{q_j>0\}}  \Big),
\end{align*}
where $i^*$ is the index chosen by the JSQ dispatcher.

Next, we consider the function $V_\perp(\mathbf x;\theta) = \sum_{i=1}^n V_{\perp i}(\mathbf x;\theta) $. Then, the drift of $V_\perp(\q(t);\theta)$ is just the sum of the drifts of $V_{\perp i}(\q(t);\theta)$, that is,
\begin{align}
    \EE\big[ & \Delta V_\perp(\q(t);\theta)\big|\q(t)\big]\allowdisplaybreaks\nonumber\\
    & = \sum_{i=1}^n \EE\big[\Delta V_{\perp i}(\q(t);\theta)\big|\q(t)\big]\allowdisplaybreaks\nonumber\\
    & =  \lambda_n  \left( e^{ \theta}- 1 \right)e^{-\frac{\theta}{n}} e^{\theta q_{\perp,\min}(t)} + \lambda_n \left(e^{-\frac{\theta}{n}}-1\right) \sum_{i=1}^n e^{\theta q_{\perp i}(t)} \allowdisplaybreaks\nonumber  \\
    & \ \ \ \ + \mu \sum_{i=1}^n e^{\theta q_{\perp i}(t)} \Big( \big(e^{-\theta} - 1\big)e^{\frac{\theta}{n} }\mathds 1_{\{q_i(t) >0\}} + \big(e^{ \frac{\theta}{n} } -1\big) \sum_{j}\mathds 1_{\{q_j>0\}}  \Big)\allowdisplaybreaks\nonumber\\
    & = \lambda_n  \left( e^{ \theta}- 1 \right)e^{-\frac{\theta}{n}} e^{\theta q_{\perp,\min}(t)} + \lambda_n \left(e^{-\frac{\theta}{n}}-1\right) \sum_{i=1}^n e^{\theta q_{\perp i}(t)} \allowdisplaybreaks\nonumber  \\
    & \ \ \ \ + \mu  \left( e^{ - \theta}- 1 \right)e^{\frac{\theta}{n}} \sum_{i=1}^n e^{\theta q_{\perp i}(t) }   +n\mu  \left(e^{\frac{\theta}{n}}-1\right)  \sum_{i=1}^n e^{\theta q_{\perp i}(t) } \allowdisplaybreaks\nonumber \\
    &\ \ \ \ -  \mu \left( e^{ - \theta}- 1 \right)e^{\frac{\theta}{n}} \sum_{i=1}^n      e^{\theta q_{\perp i}(t) }\mathds   1_{\{q_i(t) =0\}} -  \mu \left(e^{\frac{\theta}{n}}-1\right) \sum_{j=1}^n \mathds 1_{\{q_j(t) =0\}} \sum_{i=1}^n      e^{\theta q_{\perp i}(t) }. \nonumber 
\end{align}
Next, by using $\theta>0$, $q_{\perp,\min}(t) = \min_i \big\{q_i(t) - \frac{1}{n}\sum_{j=1}^n q_j(t) \big\} < 0$, and $e^{\theta q_{\perp i}(t) }\mathds   1_{\{q_i(t) =0\}} = e^{- \frac{\theta}{n}\sum_{j=1}^n q_j(t) }\mathds   1_{\{q_i(t) =0\}} \leq 1$, we have 
\begin{align*}
    \EE\big[\Delta V_\perp(\q(t);\theta)\big|\q(t)\big] &\leq \lambda_n  \left( e^{ \theta}- 1 \right)e^{-\frac{\theta}{n}}  + n\mu \left(1 - e^{ - \theta} \right)e^{\frac{\theta}{n}} \\
    & \quad + \Big( \lambda_n \left(e^{-\frac{\theta}{n}}-1\right) + \mu  \left( e^{ - \theta}- 1 \right) \\
    &\quad + n\mu  \left(e^{\frac{\theta}{n}}-1\right) \Big)\sum_{i=1}^n e^{\theta q_{\perp i}(t)}
\end{align*}
For any $\theta > 0$, we have that 
\begin{align*}
   -\frac{1}{n}e^{\theta} \leq \frac{e^{-\frac{\theta}{n}} - 1}{1-e^{-\theta}}   \leq -\frac{1}{n}, && \frac{1}{n}\leq \frac{e^{\frac{\theta}{n}} - 1}{1-e^{-\theta}}   \leq \frac{1}{n} e^{\theta},
\end{align*}
This gives us that, for any $\theta>0$, 
\begin{align}
    \frac{1}{1-e^{-\theta}}   \EE\big[\Delta V_\perp(\q(t);\theta)\big|\q(t)\big] &\leq \lambda_n e^{\theta} + n\mu e^{\frac{\theta}{n}} + \left( -\frac{\lambda_n}{n} + \mu (e^{\theta} -1)  \right)  \sum_{i=1}^n e^{\theta q_{\perp i}(t)} \allowdisplaybreaks\nonumber  \\
    % & \ \ \ \ ++ \mu \left( \frac{1}{n} + \frac{2e\theta}{n}  \right)\sum_{j=1}^n \mathds 1_{\{q_j(t) =0\}} \sum_{i=1}^n      e^{\theta q_{\perp i} }
    & \leq 2 n\mu e^{\theta}  -\frac{\lambda_n}{n}  \sum_{i=1}^n e^{\theta q_{\perp i}(t)}  + \mu \theta  e^{\theta}   \sum_{i=1}^n e^{\theta q_{\perp i}(t)}\label{eq: jsq_cont_sscproof_vperp}
\end{align}

By using similar calculation as in Eq. \eqref{eq: jsq_cont_sscproof_vperp}, and defining $U_\perp (\mathbf x; \theta) := \sum_{i=1}^n U_{\perp i } (\mathbf x; \theta)$, we have that
\begin{align*}
    \EE\big[ & \Delta U_\perp(\q(t);\theta)\big|\q(t)\big]\allowdisplaybreaks \nonumber\\
    & = \lambda_n  \left( e^{ -\theta}- 1 \right)e^{\frac{\theta}{n}} e^{-\theta q_{\perp,\min}(t)} + \lambda_n \left(e^{\frac{\theta}{n}}-1\right) \sum_{i=1}^n e^{-\theta q_{\perp i}(t)}  \allowdisplaybreaks \nonumber \\
    & \ \ \ \ + \mu  \left( e^{  \theta}- 1 \right)e^{-\frac{\theta}{n}} \sum_{i=1}^n e^{-\theta q_{\perp i}(t) }   +n\mu  \left(e^{-\frac{\theta}{n}}-1\right)  \sum_{i=1}^n e^{-\theta q_{\perp i}(t) }\allowdisplaybreaks \nonumber\\
    &\ \ \ \ -  \mu \left( e^{ \theta}- 1 \right)e^{-\frac{\theta}{n}} \sum_{j=1}^n      e^{-\theta q_{\perp i}(t) }\mathds   1_{\{q_i(t) =0\}} \allowdisplaybreaks\nonumber\\
    &\quad -  \mu \left(e^{-\frac{\theta}{n}}-1\right) \sum_{j=1}^n\mathds 1_{\{q_j(t) =0\}} \sum_{i=1}^n      e^{-\theta q_{\perp i}(t) }.
\end{align*}
Thus, for any $\theta>0$, by using $ q_{\perp i}(t) = q_{\perp,\min}(t) = - \frac{1}{n} \sum_{j=1}^n q_i(t) $ whenever $q_i(t)=0$, we have 
\begin{align}
     \frac{1}{1-e^{-\theta}}   \EE\big[ & \Delta U_\perp(\q(t);\theta)\big|\q(t)\big]\allowdisplaybreaks \nonumber\\
     &\leq -\lambda_n e^{-\theta q_{\perp,\min}(t)} + \left(\frac{\lambda_n}{n } e^{\theta}  + \mu(e^{\theta}-1)   \right)\sum_{i=1}^n e^{-\theta q_{\perp i}(t) } \allowdisplaybreaks\nonumber \\
     & \ \ \ \ - \mu e^{-\theta q_{\perp,\min}(t) } \sum_{j=1}^n \mathds 1_{\{q_j(t) =0\}} + \frac{\mu}{n} e^{\theta}\sum_{j=1}^n \mathds 1_{\{q_j(t) =0\}} \sum_{i=1}^n      e^{-\theta q_{\perp i}(t) } \allowdisplaybreaks\nonumber\\
     &\stackrel{(a)}{\leq} -\lambda_n e^{-\theta q_{\perp,\min}(t)} + \left(\frac{\lambda_n}{n } e^{\theta}  + 2\mu(e^{\theta}-1)   \right)\sum_{i=1}^n e^{-\theta q_{\perp i}(t) }\allowdisplaybreaks\nonumber\\
     &\stackrel{(b)}{\leq} -\lambda_n e^{-\theta q_{\perp,\min}(t)} + \frac{\lambda_n}{n } \sum_{i=1}^n e^{-\theta q_{\perp i}(t) } +  3\mu\theta e^{\theta}   \sum_{i=1}^n e^{-\theta q_{\perp i}(t) },\label{eq: jsq_cont_sscproof_uperp}
\end{align}
where $(a)$ follows by using $e^{-\theta q_{\perp,\min}(t) } \geq \sum_{i=1}^n      e^{-\theta q_{\perp i}(t) }$ and $\sum_{j=1}^n \mathds 1_{\{q_j(t) =0\}} \leq n$; and (b) follows by using $\frac{\lambda_n}{n } e^{\theta} = \frac{\lambda_n}{n } + \frac{\lambda_n}{n } (e^{\theta} -1) \leq \frac{\lambda_n}{n } + \mu\theta e^{\theta}$.

Before proceeding with the next step, we provide a bound on the following function. 

\begin{align*}
  D(\mathbf x; \theta):=  - e^{-\theta x_{\min}} + \frac{1}{n } \sum_{i=1}^n e^{-\theta x_i } - \frac{1}{n } \sum_{i=1}^n e^{\theta x_i },
\end{align*}
where $\mathbf x \in \RR^n$ such that $\sum_{i=1}^n x_i =0$ and $x_{\min} = \min_i x_i$. Then, $x_{\min} \leq 0$.
Suppose $\mathcal I^+ = \{i: x_i \geq 0\}$ and $\mathcal I^- = [n]\backslash \mathcal I^+ $. Suppose $n_+ = |\mathcal I^+|$ and $n_- = |\mathcal I^-|$

\begin{itemize}
    \item[1.] \textit{Case 1:}
    Consider the case 
    $$n_+ \geq n/4, \text{ and } n_- \leq 3n/4.$$ 
    In this case, we have
    \begin{align*}
    D(\mathbf x; \theta)& = - e^{-\theta x_{\min}} + \frac{1}{n} \sum_{i\in \mathcal I^+ } e^{-\theta x_i }  +\frac{1}{n} \sum_{i\in \mathcal I^- } e^{-\theta x_i } - \frac{1}{n } \sum_{i=1}^n e^{\theta x_i } \\
       & \stackrel{(a)}{\leq} - e^{-\theta x_{\min}} + \frac{n_+}{n}  +\frac{n_-}{n}e^{-\theta x_{\min} } - \frac{1}{n } \sum_{i=1}^n e^{\theta x_i } \\
       & = \frac{n_+}{n} \big( 1 - e^{-\theta x_{\min}}\big)- \frac{1}{n } \sum_{i=1}^n e^{\theta x_i }\\
       &\stackrel{(b)}{\leq}  \frac{1}{4}\big( 1 - e^{-\theta x_{\min}}\big) - \frac{1}{n } \sum_{i=1}^n e^{\theta x_i }\\
       &\leq  \frac{1}{4} - \frac{1}{4n} \sum_{i=1}^n e^{-\theta x_i }- \frac{1}{n } \sum_{i=1}^n e^{\theta x_i },
    \end{align*}
    where (a) follows by using $x_i \geq 0$ for all $ i \in \mathcal{I}^+$; and (b) follows by using $x_{\min} \leq 0$.
     % Thus, in this case, we have that 
     % \begin{align*}
     %     - e^{-\theta x_{\min}} + \frac{1}{n } \sum_{i=1}^n e^{-\theta x_i } - \frac{1}{n } \sum_{i=1}^n e^{\theta x_i } \leq \frac{1}{4} - \frac{1}{4n} \sum_{i=1}^n e^{-\theta x_i } -\frac{1}{4n} \sum_{i=1}^n e^{\theta x_i } 
     % \end{align*}
     \item[2.] \textit{Case 2:}
     Consider the case 
     $$\frac{1}{n } \sum_{i=1}^n e^{-\theta x_i }\leq \frac{1}{2}e^{-\theta x_{\min}} + \frac{1}{2n } \sum_{i=1}^n e^{\theta x_i }.$$
     In this case, we easily have that,
     \begin{align*}
          D(\mathbf x; \theta) & \leq  - \frac{1}{2} e^{-\theta x_{\min}}  -\frac{1}{2n} \sum_{i=1}^n e^{\theta x_i } \leq - \frac{1}{2n} \sum_{i=1}^n e^{-\theta x_i } -\frac{1}{2n} \sum_{i=1}^n e^{\theta x_i } 
     \end{align*}
     \item[3.] \textit{Case 3:} Consider the case 
     $$\frac{1}{n } \sum_{i=1}^n e^{-\theta x_i } > \frac{1}{2}e^{-\theta x_{\min}} + \frac{1}{2n } \sum_{i=1}^n e^{\theta x_i } \text{ and } n_+ \leq n/4.$$ 
     In this case, suppose $\sum_{i\in \mathcal I^+} x_i = B$, then $\sum_{i\in \mathcal I^-} -x_i = B$ as $\sum_{i=1}^n x_i =0$. Further, as $n_- \geq 3n/4$, suppose $\mathcal{I}^-_1$ is the collection of indices of $\frac{n}{2}$ smallest elements in $\mathcal{I^-}$ and $\mathcal I^-_2 = \mathcal{I^-} \backslash \mathcal I^-_1$. Then, for any $i\in \mathcal I^-_1$ and $j \in \mathcal I^-_2$, we have $x_i\leq x_j$. 
     
     We can argue that for any $j \in \mathcal I^-_2$, we have $-x_j \le \frac{2B}{n} $. 
     This can be shown by contraction. Suppose $\exists j \in I^-_2$ such that $-x_j > \frac{2B}{n}$. Then, $-x_i > \frac{2B}{n}$ for all $i \in I^-_1$. As $|I^-_1| = \frac{n}{2}$, we have $\sum_{i \in I^-} -x_i \geq \sum_{i \in I^-_1} -x_i > B$. This creates a contradiction as $\sum_{i\in \mathcal I^-} -x_i = B$. Thus, $-x_j \le \frac{2B}{n} $ for all $j \in \mathcal I^-_2$.
     
     This implies that,
     \begin{align}\label{eq: jsq_cont_sscproof_case3term1}
         \frac{1}{n } \sum_{i=1}^n e^{-\theta x_i } &\leq \frac{1}{n } \sum_{i\in \mathcal{I}^+}e^{-\theta x_i } +\frac{1}{n } \sum_{i\in \mathcal{I}^-_2}e^{-\theta x_i }  +\frac{1}{n } \sum_{i\in \mathcal{I}^-_1}e^{-\theta x_i } \allowdisplaybreaks\nonumber\\
         &\leq \frac{n_+}{n} + \frac{1}{n} \sum_{i\in \mathcal I^-_2} e^{\frac{2\theta B}{n}} + \frac{1}{2} e^{-\theta x_{\min}}\allowdisplaybreaks\nonumber\\
         &\leq \frac{1}{4} + \frac{1}{2}e^{\frac{2\theta B}{n}} + \frac{1}{2} e^{-\theta x_{\min}}.
     \end{align}
     Next, we have further two cases
     \begin{itemize}
         \item \textit{Case 3.1: $\theta B \leq \frac{n}{2}$.}  In this case, we have that $e^{\frac{2\theta B}{n}} \leq e$.
     \item \textit{Case 3.2: $\theta B\geq \frac{n}{2}$}. In this case, by using the Jensen's inequality, 
     \begin{align*}
         \frac{1}{n } \sum_{i=1}^n e^{\theta x_i }  \geq \frac{1}{n } \sum_{i\in \mathcal I^+} e^{\theta x_i } \geq \frac{n_+}{n} e^{\frac{\theta B}{n_{+}}} \geq \frac{1}{4} e^{\frac{4\theta B}{n}},
     \end{align*}
     where the last inequality holds as $\theta B \geq \frac{n}{2}$, and so, $\frac{n_+}{n} e^{\frac{\theta B}{n_{+}}}$ is minimized by taking $n_+ = \frac{n}{4}$. Thus, by using the assumption $\frac{1}{n } \sum_{i=1}^n e^{-\theta x_i } > \frac{1}{2}e^{-\theta x_{\min}} + \frac{1}{2n } \sum_{i=1}^n e^{\theta x_i }$ and Eq. \eqref{eq: jsq_cont_sscproof_case3term1}, we have that
     \begin{align*}
         \frac{1}{8} e^{\frac{4\theta B}{n}} \leq \frac{1}{4} +\frac{1}{2} e^{\frac{2\theta B}{n}} \leq \frac{3}{4} e^{\frac{2\theta B}{n}}.
     \end{align*}
     This gives us that $e^{\frac{2\theta B}{n}} \leq 6$. 
     \end{itemize}
     Thus, in both the above cases, $e^{\frac{2\theta B}{n}} \leq 6$. 
     This gives us
     \begin{align*}
        \frac{1}{n } \sum_{i=1}^n e^{-\theta x_i } \leq \frac{1}{4} + 3+ \frac{1}{2} e^{-\theta x_{\min}} \leq 4 +\frac{1}{2} e^{-\theta x_{\min}}.
         \end{align*}
     And thus,
    \begin{align*}
        D(\mathbf x;\theta)  & \leq 4  - \frac{1}{2} e^{-\theta x_{\min}}  -\frac{1}{2n} \sum_{i=1}^n e^{\theta x_i } \leq 4 - \frac{1}{2n} \sum_{i=1}^n e^{-\theta x_i } -\frac{1}{2n} \sum_{i=1}^n e^{\theta x_i } 
     \end{align*}
\end{itemize}

By combining the results from Case 1, Case 2 and Case 3, we get
\begin{align}\label{eq: jsq_cont_sscproof_driftterm}
    D(\mathbf x;\theta) \leq 4 - \frac{1}{4n} \sum_{i=1}^n e^{-\theta x_i } -\frac{1}{2n} \sum_{i=1}^n e^{\theta x_i }
\end{align}
Now, from Eq. \eqref{eq: jsq_cont_sscproof_vperp} and \eqref{eq: jsq_cont_sscproof_uperp}, for any $\theta>0$, we have
\begin{align*}
    \frac{1}{1-e^{-\theta}}  & \big(\EE[ \Delta V_\perp(\q(t);\theta)| \q(t)]  + \EE[ \Delta U(\q(t);\theta)| \q(t)] \big) \\
    &\leq 2n\mu e^{\theta } + \lambda_n D(\q(t);\theta) + \mu \theta  e^{\theta}   \sum_{i=1}^n e^{\theta q_{\perp i}(t)} + 3\mu\theta e^{\theta}   \sum_{i=1}^n e^{-\theta q_{\perp i}(t) }\\
    &\stackrel{(a)}{\leq} 2n\mu e^{\theta } + 4\lambda_n + \Big( \mu \theta  e^{\theta} - \frac{\lambda_n}{2n}\Big)\sum_{i=1}^n e^{\theta q_{\perp i}(t)} \\
    & \quad + \Big( 3\mu \theta  e^{\theta} - \frac{\lambda_n}{4n}\Big)\sum_{i=1}^n e^{-\theta q_{\perp i}(t)} \\
    &\stackrel{(b)}{\leq}  8n\mu + \Big( 2\mu \theta  - \frac{\lambda_n}{2n}\Big)\sum_{i=1}^n e^{\theta q_{\perp i}(t)} + \Big( 6\mu \theta   - \frac{\lambda_n}{4n}\Big)\sum_{i=1}^n e^{-\theta q_{\perp i}(t)} \\
    &\stackrel{(c)}{\leq}  8n\mu  - \frac{\lambda_n}{8n}\sum_{i=1}^n e^{\theta q_{\perp i}(t)}   - \frac{\lambda_n}{8n}\sum_{i=1}^n e^{-\theta q_{\perp i}(t)},
\end{align*}
where (a) follows using Eq. \eqref{eq: jsq_cont_sscproof_driftterm}; (b) follows by taking $\theta<1/2$, and so $e^{\theta}<2$; and (c) follows by using $\theta\leq \frac{\lambda_n}{48n\mu}$. 
Thus, the Lyapunov function $V_\perp (\cdot)+ U_\perp (\cdot)$ satisfies the Condition C2 of Lemma \ref{lem: steady_state_lyapunov}. Thus, by using Lemma \ref{lem: steady_state_lyapunov}, the drift is zero in steady state, i.e.,
\begin{align*}
    \EE_{\pi_n}[ \Delta V_\perp(\q;\theta)]  + \EE_{\pi_n}[ \Delta U(\q;\theta)] =0.
\end{align*}
So, for any $n\geq 2$, by picking $\theta_\perp := \frac{1}{96} \leq \frac{\lambda_n}{48n\mu}$, we have that for any $0<\theta\leq \theta_\perp$,
\begin{align*}
    \EE_{\pi_n}\Big[ \sum_{i=1}^n e^{-\theta q_{\perp i}}+   \sum_{i=1}^n e^{\theta q_{\perp i}} \Big] \leq \frac{64n^2\mu}{\lambda_n} = \frac{64n}{1-\epsilon_n} \leq 128n,
\end{align*}
where the last inequality holds whenever $\epsilon_n \leq 1/2$. Now, as the servers are homogeneous, by using symmetry of the JSQ system, we have 
\begin{align*}
     \EE_{\pi_n}\big[  e^{\theta |q_{\perp i}|}\big]  \leq  \EE_{\pi_n}\big[  e^{\theta q_{\perp i}}\big] + \EE_{\pi_n}\big[  e^{-\theta q_{\perp i}}\big] \leq 128.
\end{align*}
This completes the proof.
\end{proof}

\begin{proof}[Proof of Lemma \ref{lem: jsq_cont_steady_state}]
Recall that we use the notation $\lambda_n = n\mu(1-\epsilon_n)$ and $\overline{q} = \sum_{i=1}^n q_i$. Consider the exponential Lyapunov function 
\begin{align}\label{eq: jsq_cont_lyapunov_bounded}
    V(\mathbf x;\theta) := \exp\Big(\theta \epsilon_n \sum_{i=1}^n x_i \Big).
\end{align}

Note that the total queue length at any time $t$ is less than the total number of arrivals till time $t$. And as the arrivals follow a Poisson process, the total number of arrivals till time $t$ is a Poisson random variable with parameter $\lambda_n t$. This gives us that
\begin{align*}
    \EE[V(\q(t);\theta)] \leq \exp\Big( \lambda_n t \big(e^{\theta \epsilon_n} -1\big)\Big) < \infty.
\end{align*}
Thus, $V(\cdot)$ satisfies the Condition C1 of Lemma \ref{lem: steady_state_lyapunov}. 

Next, the drift of the function $V(\cdot)$ is given by 
\begin{align*}
    \EE[\Delta V(\q(t);\theta)| \q(t)] &= \Big[ \lambda_n \Big( e^{\epsilon_n\theta} -1 \Big) + \mu \Big( e^{-\epsilon_n\theta} -1 \Big) \sum_{i=1}^n \mathds 1_{\{q_i(t) > 0\}}  \Big] V(\q(t);\theta)\\
    &= \Big(1- e^{-\epsilon_n\theta}  \Big)\Big[ \lambda_n  e^{\epsilon_n\theta} - n \mu \Big] V(\q(t);\theta) \allowdisplaybreaks \nonumber \\
    & \quad + \mu\Big(1- e^{-\epsilon_n\theta}  \Big) \sum_{i=1}^n \mathds 1_{\{q_i(t) = 0\}} V(\q(t);\theta)\\
    & = \Big(1- e^{-\epsilon_n\theta}  \Big) \Big( -\gamma_n(\theta) V(\q(t);\theta) + \beta_n (\q(t);\theta)\Big),
\end{align*}
where 
\begin{align*}
    \gamma_n(\theta) :=  n \mu -\lambda_n  e^{\epsilon_n\theta} , && \beta_n (\q(t);\theta):= \mu\sum_{i=1}^n \mathds 1_{\{q_i(t) = 0\}} V(\q(t);\theta).
\end{align*}
Then, we have that 
\begin{align}
\label{eq: jsq_cont_gammafunc}
  \gamma_n(\theta)>0   \text{ for any } \theta < \theta_n := \frac{1}{\epsilon_n} \log\Big(\frac{n\mu}{\lambda_n}\Big) = -\frac{1}{\epsilon_n} \log\big( 1- \epsilon_n\big).
\end{align}
Note that as $\theta_n \geq 1$, $\gamma_n(\theta)>0$ holds for any $\theta \leq 1$.

Next, by using Eq. \eqref{eq: jsq_cont_lyapunov_bounded}, we have 
\begin{align}
\label{eq: jsq_cont_betafuncfinite}
    \EE[\beta_n (\q(t);\theta)] \leq n\mu \EE[V(\q(t);\theta)] < \infty.
\end{align}
And, in steady state, for $\theta \in (0,\theta_n)$
\begin{align}
\label{eq: jsq_cont_betafuncsteady}
    \EE_{\pi_n}[\beta_n (\q;\theta)] &= \mu\sum_{i=1}^n \EE_{\pi_n} \Big[ \mathds 1_{\{q_i = 0\}} e^{\theta \epsilon_n \overline{q}}\Big]\allowdisplaybreaks\nonumber\\
    &\stackrel{(a)}{=}\mu\sum_{i=1}^n \EE_{\pi_n} \Big[ \mathds 1_{\{q_i = 0\}} e^{\theta n \epsilon_n (q_i - q_{\perp i})}\Big]\allowdisplaybreaks\nonumber\\
    & = \mu\sum_{i=1}^n \EE_{\pi_n} \Big[ \mathds 1_{\{q_i = 0\}} e^{-\theta n \epsilon_n q_{\perp i}}\Big]\allowdisplaybreaks\nonumber\\
    &\leq \mu\sum_{i=1}^n \EE_{\pi_n} \Big[ e^{\theta n \epsilon_n |q_{\perp i}|}\Big]\allowdisplaybreaks\nonumber\\
    &\stackrel{(b)}{=} \kappa_\perp n\mu,
\end{align}
where (a) follows by using $q_{\perp i} = q_i - \frac{1}{n} \overline{q}$ by definition; and (b) follows by using Proposition \ref{prop: jsq_cont_ssc}, whenever $\theta n \epsilon_n< \theta_\perp$ for all $\theta\in(0,\theta_n)$, where $\theta_\perp$ is as given in Proposition \ref{prop: jsq_cont_ssc}. Then, for any $\theta\in(0,\theta_n)$ and $\epsilon_n \leq \frac{\theta_\perp}{\theta n}$, by combining Eq. \eqref{eq: jsq_cont_gammafunc}, \eqref{eq: jsq_cont_betafuncfinite} and \eqref{eq: jsq_cont_betafuncsteady}, we have that $V(\cdot)$ satisfies the Condition C2 of Lemma \ref{lem: steady_state_lyapunov}. Thus, by using Lemma \ref{lem: steady_state_lyapunov}, we have that for any $\theta\in(0,\theta_n)$
\begin{align}
\label{eq: jsq_cont_distproof}
    \EE_{\pi_n} [V(\q;\theta)] = \frac{1}{\gamma_n(\theta)} \EE_{\pi_n} [\beta_n(\q;\theta)].
\end{align}
Further, when $\theta<0$, we easily have
\begin{align*}
    \EE_{\pi_n} [V(\q;\theta)] = \EE_{\pi_n} \Big[  e^{\theta \epsilon_n \overline{q}}\Big] \leq 1, && \EE_{\pi_n}[\beta_n (\q;\theta)] &\leq \mu\sum_{i=1}^n \EE_{\pi_n} \Big[ \mathds 1_{\{q_i = 0\}} \Big] \leq n\mu.
\end{align*}
Thus, as $\EE_{\pi_n} [V(\q;\theta)] $ and $\EE_{\pi_n}[\beta_n (\q;\theta)]$ are bounded, by Condition C2 of Lemma \ref{lem: steady_state_lyapunov}, we have that for any $\theta<0$, the Eq. \eqref{eq: jsq_cont_distproof} is still satisfied.
% Finally, by combining Eq. \eqref{eq: jsq_cont_gammafunc}, \eqref{eq: jsq_cont_betafuncfinite} and \eqref{eq: jsq_cont_betafuncsteady}, we have that $V(\cdot)$ satisfies the Condition C2 of Lemma \ref{lem: steady_state_lyapunov} for all $\theta\in(0,\theta_n)$ and $n\geq n_\alpha$. Thus, by using Lemma \ref{lem: steady_state_lyapunov}, we have that for any $\theta < \theta_n$
% \begin{align*}
%     \EE_{\pi_n} [V(\q;\theta)] = \frac{1}{\gamma_n(\theta)} \EE_{\pi_n} [\beta_n(\q;\theta)].
% \end{align*}
This completes the proof.
\end{proof}

\begin{proof}[Proof of Claim \ref{claim: jsq_cont_beta}]
    By using Lemma \ref{lem: jsq_cont_steady_state} and taking $\theta<0$ and $\theta \rightarrow 0$, we have 
\begin{align*}
  \mu \EE_{\pi_n}\big[\sum_{i=1}^n  \mathds 1_{\{q_i = 0\}}\big] &=\lim_{\theta\rightarrow 0^- }\EE_{\pi_n} [\beta_n(\q;\theta)]\allowdisplaybreaks \nonumber\\ 
  &= \lim_{\theta \rightarrow 0^-} \gamma_n(\theta)\EE_{\pi_n} [V(\q;\theta)] \allowdisplaybreaks \nonumber\\
  &= n\mu - \lambda_n \allowdisplaybreaks \nonumber\\
  &= n\epsilon_n\mu.
\end{align*}
  Also, by symmetry, we have $\EE_{\pi_n}\big[  \mathds 1_{\{q_i = 0\}}\big] =\EE_{\pi_n}\big[  \mathds 1_{\{q_j = 0\}}\big]$ for all $i,j$. Thus, we get, 
  \begin{align*}
      \EE_{\pi_n}\big[  \mathds 1_{\{q_i = 0\}}\big] = \epsilon_n, \quad \forall i \in \{1,2,\dots,n\}.
  \end{align*}
  Further, we have 
\begin{align}
    \sum_{i=1}^n \EE_{\pi_n} \Big[ \mathds 1_{\{q_i = 0\}} \big|  e^{\theta \epsilon_n \overline{q}} -1\big|\Big]&\stackrel{(a)}{=}\sum_{i=1}^n \EE_{\pi_n} \Big[ \mathds 1_{\{q_i = 0\}} \big|e^{\theta n \epsilon_n (q_i - q_{\perp i})}-1\big|\Big]\allowdisplaybreaks\nonumber\\
    & = \sum_{i=1}^n \EE_{\pi_n} \Big[ \mathds 1_{\{q_i = 0\}} \big|e^{-\theta n \epsilon_n q_{\perp i}}-1\big|\Big]\allowdisplaybreaks\nonumber\\
    &\stackrel{(b)}{=} n\EE_{\pi_n} \Big[ \mathds 1_{\{q_i = 0\}} \big|e^{-\theta n \epsilon_n q_{\perp i}}-1\big|\Big]\allowdisplaybreaks\nonumber\\
    &\stackrel{(c)}{\leq} n\EE_{\pi_n} \big[ \mathds 1_{\{q_i = 0\}}^p \big]^{\frac{1}{p}} \EE_{\pi_n} \Big[ \big| e^{-\theta n \epsilon_n q_{\perp i}}-1\big|^{r}\Big]^{\frac{1}{r}}\allowdisplaybreaks\nonumber\\
    &\stackrel{(d)}{\leq} \theta n^2 \epsilon_n^{1 +\frac{1}{p}} \EE_{\pi_n} \Big[ |q_{\perp i}|^r e^{r|\theta| n \epsilon_n |q_{\perp i}|}\Big]^{\frac{1}{r}}\allowdisplaybreaks\nonumber\\
    &\stackrel{(e)}{\leq} \theta n^2 \epsilon_n^{1 +\frac{1}{p}} \EE_{\pi_n} \big[ |q_{\perp i}|^{2r}\big]^{\frac{1}{2r}} \EE_{\pi_n} \Big[ e^{2r|\theta| n \epsilon_n |q_{\perp i}|}\Big]^{\frac{1}{2r}}\allowdisplaybreaks\nonumber\\
    &\stackrel{(f)}{\leq} \theta n^2 \epsilon_n^{1 +\frac{1}{p}} \times \frac{2\kappa_\perp r}{\theta_\perp} \times \kappa_\perp,
\end{align}
  where (a) follows by using $q_{\perp i} = q_i - \overline q$; (b) follows by using the symmetry of the system as the servers are homogeneous; (c) and (e) follows by using H\"older's inequality by choosing $r,p >0$ and $\frac{1}{r}+\frac{1}{p} =1$; (d) follows as $|e^{x} -1| \leq |x| e^{|x|}$ for all $x \in \RR$ and Eq. \eqref{eq: jsq_cont_unused}; and (f) follows by using Proposition \ref{prop: jsq_cont_ssc}, by choosing $n$ large enough such that $2r|\theta| n \epsilon_n < \theta_\perp$.

Now, by choosing $r = -\log \epsilon_n $ and $p = \frac{\log \epsilon_n }{ \log \epsilon_n  +1}$, we have 
  \begin{align*}
      \sum_{i=1}^n \EE_{\pi_n} \Big[ \mathds 1_{\{q_i = 0\}} \big| e^{\theta \epsilon_n \overline{q}} -1\big|\Big]& \leq \frac{2e\kappa_\perp^2}{\theta_\perp} \alpha \theta n^{2} \epsilon_n^2 \log \Big(\frac{1}{\epsilon_n}\Big),
  \end{align*}
  where the inequality holds whenever $2|\theta |  n \epsilon_n\log \big(\frac{1}{\epsilon_n}\big)  < \theta_\perp$.
\end{proof}

\section{Proof of results in Section \ref{sec: ssq}}
\label{app: ssq}

\begin{proof}[Proof of Lemma \ref{lem: ssq_mgf_steady_state}]
First, we show that the MGF of $q(t)$ exists for any $t\geq 0$ and for any $\theta < \theta_0$. Note that as $a(t) \leq A$ almost surely, we have, $q(t) \leq q(0) + At, \forall t\geq 0$. Thus, by assumption $V(q(0);\theta) < \infty$, we have
\begin{align*}
    \EE \big[e^{\theta \epsilon q(t)}\big] \leq \EE[V( q(0);\theta )] e^{\theta \epsilon A t} < \infty, \ \forall t\geq 0. 
\end{align*}
This shows that $V(\cdot)$ satisfies the Condition C1 of Lemma \ref{lem: steady_state_lyapunov}. 
Further, the MGF of $a(t),s(t)$ and $u(t)$ also exists as they are bounded random variables, as given in Section \ref{sec: ssq_model}. Next,
the recursion in Eq. \eqref{eq: ssq_lindley} gives us that, $q(t+1) -u(t) = q(t) +a(t) -s(t)$. We write this in an exponential form and take expectation to get 
\begin{align}
\label{eq: ssq_qtplus1_term1}
    \EE \big[e^{\theta \epsilon (q(t+1)-u(t))} \big| q(t) \big]  &= \EE \big[e^{\theta \epsilon (q(t)+a(t)-s(t))}\big| q(t) \big] \nonumber\\
    &=\EE \big[e^{\theta \epsilon (a-s)}\big] V(q(t);\theta) \nonumber\\
    &= (1-\gamma_\epsilon(\theta)) V(q(t);\theta),
\end{align}
where the second equality follows as the arrival and the potential service are identically distributed across time, and are independent of the queue length process.
Using the relation $q(t+1)u(t) =0, \forall t\geq0$, we have
\begin{align*}
    \EE \big[e^{\theta \epsilon (q(t+1)-u(t))} \big| q(t) \big] &= \EE[V(q(t+1);\theta)| q(t)]+ \EE \big[e^{-\theta \epsilon u(t)} \big| q(t) \big] -1\\
    &= \EE[V(q(t+1);\theta)| q(t)] -\beta_\epsilon (t;\theta).
\end{align*}
Substituting this in the Eq. \eqref{eq: ssq_qtplus1_term1}, we arrive at
\begin{align}
\label{eq: ssq_mgfeq_finitet}
    \EE[V(q(t+1);\theta)|q(t)] = (1-\gamma_\epsilon(\theta)) V(q(t);\theta) + \beta_\epsilon (t;\theta).
\end{align}
By taking expectation on both sides, we get
\begin{align}
\label{eq: ssq_mgfeq_finitet_expectation}
    \EE[V(q(t+1);\theta)] = (1-\gamma_\epsilon(\theta)) \EE[V(q(t);\theta)] + \EE[\beta_\epsilon (t;\theta)].
\end{align}
Now the result in Eq. \eqref{eq: ssq_mgf_recursiveeq} follows by recursively using the Eq. \eqref{eq: ssq_mgfeq_finitet_expectation}.

Further, we can rewrite Eq. \eqref{eq: ssq_mgfeq_finitet} as
\begin{align*}
    \EE[\Delta V(q(t);\theta)|q(t)] = -\gamma_\epsilon(\theta) V(q(t);\theta) + \beta_\epsilon (t;\theta).
\end{align*}
As $u(t)\leq A$, for any $\theta\in \RR$, $\beta_\epsilon (t;\theta) \leq e^{|\theta|\epsilon A}$. Further, $\gamma_\epsilon(\theta) >0$ for any $\theta< \theta_\epsilon$. Thus, $V(\cdot)$ satisfies the Condition C2 of Lemma \ref{lem: steady_state_lyapunov}. Thus, from Lemma \ref{lem: steady_state_lyapunov}, we get that for all $\theta< \min\{\theta_0, \theta_\epsilon\}$,
 \begin{align*}
     \EE[V(q;\theta)] = \lim_{t\rightarrow \infty} \EE[V(q(t);\theta)] = \frac{1}{\gamma_\epsilon(\theta)} \limt \EE[\beta_\epsilon (t;\theta)] = \frac{1- \EE_{\pi_\epsilon}[e^{-\theta \epsilon u}]}{1-\EE \big[e^{\epsilon\theta(a-s)}\big]}.
 \end{align*}
 This completes the proof.
 % From the definition of $\gamma_\epsilon(\theta)$ and $\theta_\epsilon$, we have that $1-\gamma_\epsilon (\theta)\in (0,1)$ for all $\theta<\theta_\epsilon$. Then, as $V(q(0);\theta)<\infty$ for all $\theta < \theta_0 \leq \theta_\epsilon$, we have
 % \begin{align}
 % \label{eq: ssq_mgf_initial}
 %     \lim_{t\rightarrow \infty} (1-\gamma_\epsilon(\theta))^{t+1}V(q(0);\theta) = 0, \ \ \ \ \forall \theta<\theta_\epsilon.
 % \end{align}
 % As the steady state distribution $\pi_\epsilon$ exists, and $u(t)$ is almost surely bounded for all $t\geq 0$, we have \[\lim_{t\rightarrow \infty} \mathbb E\big[ e^{-\theta \epsilon u(t)} \big] = \EE_{\pi_\epsilon} \big[e^{-\theta \epsilon u} \big], \ \ \ \ \forall \theta \in \mathbb R.\]
 % This implies that
 % \begin{align*}
 %     \lim_{t\rightarrow \infty} \beta_\epsilon (\theta; t ) = \beta_\epsilon (\theta): = 1- \EE_{\pi_\epsilon} \big[e^{-\theta \epsilon u} \big], \ \ \ \ \forall \theta \in \mathbb R.
 % \end{align*}
 % This in turn implies that
 % \begin{align}
 %  \label{eq: ssq_mgf_unusedservice}
 %     \lim_{t\rightarrow \infty}  \sum_{i=0}^t (1-\gamma_\epsilon(\theta))^i \beta_\epsilon (\theta;t-i) = \frac{1}{\gamma_\epsilon(\theta)} \beta_\epsilon (\theta) < \infty, \ \ \ \ \forall \theta<\theta_\epsilon.
 % \end{align}
 % Plugging Eq. \eqref{eq: ssq_mgf_initial} and Eq. \eqref{eq: ssq_mgf_unusedservice} in the Eq. \eqref{eq: ssq_mgf_recursiveeq}, we have that $\lim_{t\rightarrow \infty} M(\theta;\epsilon q(t)) <\infty$ for all $\theta<\theta_\epsilon$, and so,
\end{proof}

\begin{proof}[Proof of Claim \ref{claim: ssq_uterm_bound}]
    By simply equating the drift of the queue length to zero in steady state, that is
    \begin{align*}
        \EE_{\pi_\epsilon}\big[ \EE[q(t+1) - q(t) | q(t)] \big] =0,
    \end{align*}
     we have that 
     \begin{align*}
         \EE_{\pi_\epsilon}[u] = \EE[s -a] = \epsilon \mu_\epsilon.
     \end{align*}
    Since $e^{y} \geq 1+ y$ for all $y \in \mathbb R$, we have
    \begin{align*}
        1- \EE_{\pi_\epsilon}[e^{-\theta \epsilon u}] \leq \theta \epsilon \EE_{\pi_\epsilon}[u] = \epsilon^2 \theta \mu_\epsilon.
    \end{align*}
\end{proof}

\begin{proof}[Proof of Claim \ref{claim: ssq_aminussterm_bound}]
We have
    \begin{align}
    \label{eq: ssq_thmproof_mfg_aminuss}
        \EE \big[e^{\epsilon\theta(a-s)}\big] & \leq 1  + \theta \epsilon \EE[a-s] + \frac{1}{2}\epsilon^2 \theta^2 \EE[(a-s)^2] + \frac{1}{6}\epsilon^3 \theta^3 \EE[(a-s)^3] \\
        & \quad + \epsilon^4 \theta^4 \EE[(a-s)^4 e^{\epsilon \theta (a-s)}] \nonumber\\
        & \leq 1  - \theta \epsilon^2 \mu_\epsilon + \frac{1}{2}\epsilon^2 \theta^2 (\sigma_\epsilon^2 + \epsilon^2\mu_\epsilon^2) + \frac{1}{6}\epsilon^3 \theta^3 E_3 + (\epsilon \theta A)^4  \EE[e^{\epsilon \theta (a-s)}]
        % & \geq \epsilon^2 \theta \mu_\epsilon -\frac{1}{2}\epsilon^2 \theta^2 (\sigma_\epsilon^2 + \epsilon^2\mu_\epsilon^2) -  \epsilon^3 \theta^3 A^3 e^{\theta A}\\
        % & = \epsilon^2 \theta \mu_\epsilon \Big( 1 - \frac{\sigma_\epsilon^2}{2\mu_\epsilon}\theta \Big(1+ \frac{\epsilon^2 \mu_\epsilon^2}{\sigma_\epsilon^2} + \frac{2\epsilon \theta A^3}{\sigma_\epsilon^2} e^{\theta A} \Big) \Big)\\
        % &\geq \epsilon^2 \theta \mu_\epsilon \Big( 1 - \frac{\sigma_\epsilon^2}{2\mu_\epsilon}\theta \Big(1+ \kappa_1 \epsilon \Big) \Big),
    \end{align}
    where the second inequality holds by the properties of the arrival and the service process and by using $E_3:= \max\{0, \EE[(a-s)^3]\}$. This implies that 
    \begin{align*}
         \EE \big[e^{\epsilon\theta(a-s)}\big] \leq \frac{1}{1- (\epsilon \theta A)^4} \Big(1  - \theta \epsilon^2 \mu_\epsilon + \frac{1}{2}\epsilon^2 \theta^2 (\sigma_\epsilon^2 + \epsilon^2\mu_\epsilon^2) + \frac{1}{6}\epsilon^3 \theta^3 E_3 \Big). 
    \end{align*}
    Thus, to ensure $\EE \big[e^{\epsilon\theta(a-s)}\big] < 1$, we only need
    \begin{align*}
        1  - \theta \epsilon^2 \mu_\epsilon + \frac{1}{2}\epsilon^2 \theta^2 (\sigma_\epsilon^2 + \epsilon^2\mu_\epsilon^2) + \frac{1}{6}\epsilon^3 \theta^3 E_3 \leq 1- (\epsilon \theta A)^4. 
    \end{align*}
    This translates to the condition, 
    \begin{align*}
        -  \mu_\epsilon + \frac{1}{2} \theta (\sigma_\epsilon^2 + \epsilon^2\mu_\epsilon^2) + \frac{1}{6}\epsilon \theta^2 E_3 \leq - \epsilon^2 \theta^3 A^4.
    \end{align*}
    By simple calculations, we have that the above inequality holds true whenever 
    \begin{align*}
        \theta\in \Bigg(0, \mu_\epsilon \Big( \frac{1}{2}\sigma_\epsilon^2 +  \epsilon \frac{\mu_\epsilon E_3}{ 3 \sigma_\epsilon^2} + \epsilon^2 \mu_\epsilon^2+ \epsilon^2 \frac{4\mu_\epsilon^2 A^4}{\sigma_\epsilon^4} \Big)^{-1} \Bigg).
    \end{align*}
    Thus, by choosing 
    \begin{align*}
        \kappa_1 := \frac{ 2\mu_\epsilon E_3}{3\sigma_\epsilon^4} + \epsilon \frac{9 \mu_\epsilon^2 A^4}{\sigma_\epsilon^6}, && \tilde \theta_\epsilon := \frac{2\mu_\epsilon}{\sigma_\epsilon^2(1+ \kappa_1 \epsilon )},
    \end{align*}
    we get that
    \begin{align}
    \label{eq: ssq_thmproof_mgf_lesthan1}
        \EE \big[e^{\epsilon\theta(a-s)}\big] < 1, \quad \forall \theta\in \big(0, \tilde \theta_\epsilon  \big).
    \end{align}
    Further, from Eq. \eqref{eq: ssq_thmproof_mfg_aminuss}, for any $\theta\in \big(0,\tilde \theta_\epsilon  \big)$, we also have
    \begin{align*}
        1- \EE \big[e^{\epsilon\theta(a-s)}\big] &\geq  \theta \epsilon^2 \mu_\epsilon - \frac{1}{2}\epsilon^2 \theta^2 (\sigma_\epsilon^2 + \epsilon^2\mu_\epsilon^2) - \frac{1}{6}\epsilon^3 \theta^3 E_3 - (\epsilon \theta A)^4  \EE[e^{\epsilon \theta (a-s)}]\allowdisplaybreaks\nonumber\\
        &\stackrel{(a)}{\geq} \theta \epsilon^2 \mu_\epsilon - \frac{1}{2}\epsilon^2 \theta^2 (\sigma_\epsilon^2 + \epsilon^2\mu_\epsilon^2) - \frac{1}{6}\epsilon^3 \theta^3 E_3 - (\epsilon \theta A)^4\allowdisplaybreaks\\
        &\stackrel{(b)}{\geq} \theta\epsilon^2 \mu_\epsilon \Big(1 - \theta \frac{\sigma_\epsilon^2}{2\mu_\epsilon} \big(1+\kappa_1 \epsilon\big) \Big),
    \end{align*}
    where (a) follows by using Eq. \eqref{eq: ssq_thmproof_mgf_lesthan1}; and (b) follows by using $\theta < \tilde \theta_\epsilon \leq  \frac{2\mu_\epsilon}{\sigma_\epsilon^2}$.    
\end{proof}

% \begin{proof}[Proof of Theorem \ref{thm: ssq_tail_bound}]
 
%     This gives us that
%     \begin{align*}
%          \EE \big[e^{\epsilon\theta q}\big] =  \frac{1- \EE_{\pi_\epsilon}[e^{-\theta \epsilon u}]}{1-\EE \big[e^{\epsilon\theta(a-s)}\big]} \leq \Big( 1 - \frac{\sigma_\epsilon^2}{2\mu_\epsilon}\theta \Big(1+ \kappa_1 \epsilon \Big) \Big)^{-1}
%     \end{align*}
%     Now, by using Lemma \ref{lem: mgf_to_prob_bound}, we get that for any $x > \tilde \theta_\epsilon$, we have 
%     \begin{align*}
%         \PP (\epsilon q> x) \leq \frac{2e\mu_\epsilon x}{\sigma_\epsilon^2(1+ \kappa_1 \epsilon )} \exp\Big(-\frac{2\mu_\epsilon x }{\sigma_\epsilon^2 (1 + \kappa_1\epsilon 
%          )}\Big) \leq \frac{2e\mu_\epsilon x}{\sigma_\epsilon^2} \exp\Big(-\frac{2\mu_\epsilon (1 - \kappa_1\epsilon 
%          ) x }{\sigma_\epsilon^2 }\Big).
%     \end{align*}
%     This completes the proof.
%     % This implies that 
%     % \begin{align*}
%     %     \PP (\epsilon q> x)  \leq e^{-\theta x} \Big( 1 -  \theta \frac{\sigma_\epsilon^2 }{2\mu_\epsilon} \Big( 1 +\epsilon \frac{\theta_\epsilon  C_3}{3\sigma_\epsilon^2}   + o(\epsilon)\Big)\Big)^{-1}.
%     % \end{align*}
%     % \begin{align*}
%     %    \PP (\epsilon q> x) \leq \frac{2e\mu_\epsilon x}{\sigma_\epsilon^2 } \exp\Big(-\frac{2\mu_\epsilon }{\sigma_\epsilon^2 }\Big(1 - \epsilon \frac{\theta_\epsilon  C_3}{3\sigma_\epsilon^2}   + o(\epsilon) \Big)x\Big). 
%     % \end{align*}
%     % This completes the proof.
% \end{proof}

\section{Proof of results in Section \ref{sec: mmn}}
\label{app: mmn}

\subsection{Proof of Lemma \ref{lem: mmn_steady_state_mgf}}
 \begin{proof} 
Consider the function $V_1(x; \theta) = e^{\theta  [x-n]^+} $. Then, from the definition of $q_n(t)$ and  $w_n(t)$, $V_1(q_n(t);\theta) = e^{\theta  w_n(t)}$, and similarly,  $V_1(q_n;\theta) = e^{\theta  w_n}$. Suppose $a_n(t)$ is the total number of arrivals till time $t$. Then, $a_n(t)$ follows Poission distribution with parameter $\lambda_n t$, and $w_n(t) \leq q_n(t) \leq a_n(t)$. This gives us that,
\begin{align*}
    \EE[V_1(q_n(t);\theta)] = \EE[e^{\theta  w_n(t)}] \leq
    % \EE[e^{\theta  q_n(0)}]
    \EE[e^{\theta  a_n(t)}] =
    % \EE[e^{\theta  q_n(0)}]
    e^{\lambda_n t (e^{\theta  } -1)}<\infty.
\end{align*}
% Under the assumption that $\EE[e^{\theta  q_n(0)}] < \infty$ for all $\theta< \theta_1$ and $\forall n\geq 1$, we have that $\EE[V_1(q_n(t);\theta)] < \infty$ for all $t\geq 0$ and $\theta < \theta_1$. 
Thus, $V_1(\cdot)$ satisfies Condition C1 of Lemma \ref{lem: steady_state_lyapunov}.

Suppose $G$ denotes the generator matrix of the underlying CTMC of $M/M/n+M$ queue. When the number of jobs in the system is $q$, the drift of $V_1(q; \theta)$, denoted by $\Delta V_1(q; \theta)$ is defined as 
    \begin{align*}
        \EE[\Delta V_1(q(t); \theta)|q(t) =q] := \sum_{q' \in \mathbb W} G(q',q) \left(V_1(q';\theta) - V_1(q;\theta)\right).
    \end{align*}
    Then, 
    \begin{align*}
        \EE[& \Delta V_1(q(t); \theta)|q(t) =q]\\
        &= \lambda_n V_1(q+1; \theta) + \mu\min\{n,q\} V_1(q-1; \theta) - (\lambda_n +\mu\min\{n,q\} ) V_1(q; \theta)\\
        &= e^{\theta [q-n]^+}\left(\lambda_n \left( e^{\theta } -1 \right)  \mathds 1_{\{q\geq n\}}+ n\mu \left( e^{-\theta } -1 \right) \mathds 1_{\{q>n\}}\right)\\
        & = \left( e^{\theta } -1 \right) \left[e^{\theta [q-n]^+}\left(\lambda_n  - n\mu e^{-\theta }   \right)\mathds 1_{\{q>n\}} + \lambda_n \mathds 1_{\{q= n\}} \right]\\
        & = \left( e^{\theta } -1 \right) \left(\lambda_n  - n\mu e^{-\theta } \right) e^{\theta [q-n]^+}\\
        & \ \ \ \ - \left( e^{\theta } -1 \right) \left( \lambda_n \mathds 1_{\{q< n\}} - n\mu e^{-\theta }\mathds 1_{\{q\leq  n\}} \right).
    \end{align*}

    We define 
    \begin{align*}
        \gamma_n(\theta) &:= -\left(\lambda_n  - n\mu e^{-\theta } \right),\\
        \beta_n(q;\theta) &:=  \left(n\mu e^{-\theta }\mathds 1_{\{q\leq  n\}} - \lambda_n \mathds 1_{\{q< n\}}  \right).
    \end{align*}
    Then, we have
    \begin{align*}
        \EE[\Delta V_1(q_n(t); \theta) | q_n(t)] = \left( e^{\theta } -1 \right) \Big( - \gamma_n(\theta) V_1(q_n(t);\theta) + \beta_n(q_n(t);\theta) \Big).
    \end{align*}
    Note that for any
\begin{align*}
    \theta < \theta_n:= - \log \Big( \frac{\lambda_n}{n\mu} \Big) =  \log  \frac{1}{1-\epsilon_n},
\end{align*}
we have $\gamma_n(\theta) > 0$. Further, $\beta_n(\cdot)$ is a bounded function as,  $\beta_n(q_n(t);\theta) \leq n\mu\left( e^{\theta } -1 \right)$, and then, as $q_n(t)$ converges to $q_n$ in distribution, we have \begin{align*}
     \EE[\beta_n(q_n;\theta)] = \left( e^{\theta } -1 \right) \left(n\mu e^{-\theta }\PP(q_n\leq  n) - \lambda_n \PP(q_n<n) \right)
\end{align*} 
Thus, $V_1(\cdot)$ also satisfies Condition C2 of Lemma \ref{lem: steady_state_lyapunov}. Then, we get that for all $\theta< \theta_n$, we have 
\begin{align}
\label{eq: wn_alpha0_wnmgf}
    \EE[e^{\theta  w_n}] = \frac{n\mu e^{-\theta }\PP(q_n\leq  n) - \lambda_n \PP(q_n<n)}{n\mu e^{-\theta } - \lambda_n }.
\end{align}
By putting $\theta =0$ in the above equation, we get 
\begin{align*}
    n\mu \PP(q_n\leq  n) - \lambda_n \PP(q_n<n) = n\mu- \lambda_n =n\mu\epsilon_n.
\end{align*}
This implies that,
\begin{align}
\label{eq: wn_alpha0_prob_wng0}
    -n\mu\epsilon_n\PP(w_n>0)+ \lambda_n\PP (q_n =n) =0.
\end{align}
Further, $\EE[e^{\theta  w_n}] = \EE[e^{\theta  w_n} \mathds 1_{\{w_n>0\}}] + \PP(q_n \leq n)$. Combining this with Eq. \eqref{eq: wn_alpha0_wnmgf} and \eqref{eq: wn_alpha0_prob_wng0}, we arrive at
\begin{align*}
    \EE \left[e^{\theta  w_n}\mathds 1_{\{w_n>0\}} \right] = \frac{n\mu\epsilon_n \PP(w_n>0)}{ n\mu e^{-\theta } - \lambda_n}.
\end{align*} 
This implies that,
\begin{align*}
    \EE \left[e^{\theta  w_n}\Big| w_n>0 \right] = \frac{n\mu\epsilon_n}{ n\mu e^{-\theta } - \lambda_n} = \frac{1}{1 - \frac{1}{\epsilon_n} \left(1-e^{-\theta }\right)}. 
\end{align*}

Next, we evaluate the MGF of $r_n$. 
Consider the function $V_2(x; \theta) = e^{\theta [n-x]^+} $. Then, from the definition of $q_n(t)$ and  $r_n(t)$, $V_2(q_n(t);\theta) = e^{\theta r_n(t)}$, and similarly,  $V_2(q_n;\theta) = e^{\theta r_n}$.
The drift of $V_2(q; \theta)$, denoted by $\Delta V_2(q; \theta)$ is defined as 
\begin{align*}
    \EE[\Delta V_2(q_n(t); \theta) | q_n(t) =q] := \sum_{q' \in \mathbb W} G(q',q) \left(V_2(q';\theta) - V_2(q;\theta)\right).
\end{align*} 
Then,
\begin{align*}
   \EE[& \Delta V_2(q_n(t); \theta) | q_n(t) =q] \\
   &= \lambda_n V_2(q+1; \theta) + \mu\min\{n,q\} V_2(q-1; \theta) - (\lambda_n +\mu\min\{n,q\} ) V_2(q; \theta)\\
    &= e^{\theta [n-q]^+}\left(\lambda_n \left( e^{-\theta } -1 \right)  \mathds 1_{\{q< n\}}+ \mu\min\{n,q\} \left( e^{\theta } -1 \right) \mathds 1_{\{q\leq n\}}\right)\\
    & = e^{\theta [n-q]^+}\left(\lambda_n \left( e^{-\theta } -1 \right) + \mu(n-[n-q]^+)\left( e^{\theta } -1 \right) \right)\\
    & \ \ \ \ + \lambda_n \left(1- e^{-\theta }  \right)\mathds 1_{\{q\geq n\}} - n\mu \left( e^{\theta } -1 \right)\mathds 1_{\{q> n\}}\\
    & = (1-e^{-\theta }) \Big( - \tilde \gamma_n(q;\theta) V_2(q;\theta) + \tilde \beta_n(q;\theta) \Big),
\end{align*}
where
\begin{align*}
    \tilde \gamma_n(q;\theta) &:=  \left(\lambda_n - \mu(n-[n-q]^+)e^{\theta }  \right),\\
    \tilde \beta_n(q;\theta) &:= \big(\lambda_n \mathds 1_{\{q\geq n\}} - n\mu e^{\theta } \mathds 1_{\{q> n\}}\big).
\end{align*}
Note that $V_2(q ;\theta) \leq e^{\theta  n} $. Thus, as $V_2(\cdot)$ is uniformly bounded, in steady-state, for any $\theta\in \RR$, we have 
\begin{align*}
     \EE[ \Delta V_2(q_n; \theta) ] =0.
\end{align*}
and so,
\begin{align*}
    \EE[\tilde \gamma_n(q_n;\theta)V_2(q_n;\theta) ] = \EE[\tilde \beta_n(q_n;\theta)].
\end{align*}
This gives us that for any $\theta\in \RR$,
\begin{align}
\label{eq: rn_mgf_eq}
 \left(\lambda_n e^{-\theta } -n \mu \right) \EE\left[e^{\theta  r_n}\right] + \mu \EE\left[r_n e^{\theta  r_n}\right] = \lambda_n e^{-\theta } \PP(q_n\geq n) - n\mu \PP(q_n> n).
\end{align}
Further, $\EE[e^{\theta  r_n}] = \EE[e^{\theta  r_n} \mathds 1_{\{r_n>0\}}] + \PP(q_n \geq n)$, and \[\EE[r_n e^{\theta  r_n}] = \EE[r_ne^{\theta  r_n} \mathds 1_{\{r_n>0\}}] = \frac{d}{d\theta} \EE[e^{\theta  r_n} \mathds 1_{\{r_n>0\}}].\]  
Substituting this in Eq. \eqref{eq: rn_mgf_eq}, we have
\begin{align*}
    \left(n(1-\epsilon_n)e^{-\theta } -n  \right) \EE[e^{\theta  r_n} \mathds 1_{\{r_n>0\}}] + \frac{d}{d\theta} \EE[e^{\theta  r_n} \mathds 1_{\{r_n>0\}}] = n\PP(q_n = n).
\end{align*}
Solving the above differential equation gives us that 
\begin{align*}
    \EE \left[e^{\theta  r_n} \mathds 1_{\{r_n>0\}} \right] = n\PP(q_n=n) \tilde  G_n^{-1}(\theta) \int_{-\infty}^\theta G_n(t) dt,
\end{align*}
where \[G_n(t) = \exp\left( -n  t -(1-\epsilon_n)(e^{- t} -1)\right) = \exp\Big( -n\epsilon_n  t -(1-\epsilon_n) (e^{- t} +  t -1)\Big).\]
By substituting $\theta=0$, we get 
\begin{align}
\label{eq: rn_prob_rng0}
  \PP(r_n > 0) =  n\PP(q_n=n)\int_{-\infty}^0 G_n(t) dt.
\end{align}
Thus, we finally get
\begin{align*}
    \EE \left[e^{\theta  r_n} \mathds 1_{\{r_n>0\}} \right] =  \PP(r_n > 0)\tilde  G_n^{-1}(\theta) \left(\int_{-\infty}^0 G_n(t) dt \right)^{-1} \int_{-\infty}^\theta G_n(t) dt.
\end{align*}
This implies that, for any $\theta\in \RR$,
\begin{align*}
     \EE \left[e^{\theta  r_n} | r_n>0 \right] =  \tilde  G_n^{-1}(\theta) \left(\int_{-\infty}^0 G_n(t) dt \right)^{-1} \int_{-\infty}^\theta G_n(t) dt.
\end{align*}
Next, we calculate the steady state values of $\PP(r_n>0)$, $\PP(w_n>0)$ and $\PP(q_n=n)$. Note that
\begin{align*}
    \PP(r_n>0) + \PP(q_n=n) + \PP(w_n>0) = \PP(q_n<n) +  \PP(q_n=n) + \PP(q_n > n) =1.
\end{align*}
We use this to solve for $\PP(q_n=n)$.
By using Eq. \eqref{eq: wn_alpha0_prob_wng0} and \eqref{eq: rn_prob_rng0} to get that
    \begin{align*}
        \left( \frac{\lambda_n}{n\mu\epsilon_n} +1 + n\int_{-\infty}^0 G_n(t) dt   \right)\PP(q_n =n) =1.
    \end{align*}
    This gives us that
    \begin{align*}
       \PP(q_n =n) = \left(\frac{1}{\epsilon_n} + n\int_{-\infty}^0 G_n(t) dt   \right)^{-1}.
    \end{align*}
    Now, the result follows using Eq. \eqref{eq: wn_alpha0_prob_wng0} and \eqref{eq: rn_prob_rng0}.
\end{proof}

\begin{proof}[Proof of Theorem \ref{thm: mmn_idleservers}]
First note that for any $a,b$ and $s_1,s_2$, we have 
\begin{align}
\label{eq: normal_cdf}
    \int_{s_1}^{s_2} & \exp\Big(at - \frac{1}{2} bt^2\Big) \nonumber\\
    &= \sqrt{\frac{2\pi}{b}} \exp\Big( \frac{a^2}{2b}\Big) \Big[ \Phi \Big(\sqrt{b} s_2 - \frac{a}{\sqrt{b}} \Big)-\Phi \Big(\sqrt{b} s_1 - \frac{a}{\sqrt{b}} \Big)  \Big],
\end{align}
where $\Phi(\cdot)$ denotes the CDF of a standard normal random variable. First, we prove the result for the number of idle servers when the system is in Super-HW regime. In this regime, $\epsilon_n = c n^{-\alpha}$ where $\alpha> \frac{1}{2}$. For convenience, we use $\rho_n = 1-\epsilon_n$ and $\zeta_n := \sqrt{ \frac{n\epsilon_n^2}{\rho_n}}$. For system in Super-HW regime, 
\begin{align*}
    \zeta_n = \Big( \frac{c^2 n^{1-2\alpha}}{1-cn^{-\alpha}} \Big)^{\frac{1}{2}} \leq 2cn^{\frac{1}{2} - \alpha},
\end{align*}
where the last inequality follows by assuming $n\geq 4c^2$. Further, if $n^{\alpha-\frac{1}{2}}> 2c$, we have $\zeta_n \leq 1$
Now, 
    \begin{align*}
        \int_{-\infty}^\theta & G_n(t) dt \\&= \int_{-\infty}^\theta \exp\Big( - n\epsilon_n  t -n\rho_n (e^{- t} + t -1)\Big) dt \allowdisplaybreaks \nonumber\\
        & = \int_{-\theta}^\infty \exp\Big( n\epsilon_n  t -n\rho_n (e^{ t} - t -1)\Big) dt\allowdisplaybreaks \nonumber\\
        & = \int_{-\theta}^\infty \exp\Big( n\epsilon_n  t -\frac{1}{2} n\rho_nt^2 \Big)  \exp\Big( -n\rho_n \big(e^{ t} - \frac{1}{2}t^2 - t -1\big)\Big) dt\allowdisplaybreaks \nonumber\\
        & \stackrel{(a)}{\leq }  \int_{0}^\infty \exp\Big( n\epsilon_n  t -\frac{1}{2} n\rho_nt^2 \Big) dt\allowdisplaybreaks \nonumber\\
        & \quad + \exp\Big( -\frac{1}{6}n\rho_n \theta^3\Big) \int_{-\theta}^0 \exp\Big( n\epsilon_n  t -\frac{1}{2} n\rho_n t^2 \Big) dt\allowdisplaybreaks \nonumber\\
        & \stackrel{(b)}{\leq }  \sqrt{\frac{2\pi}{n\rho_n}} \exp\Big( \frac{1}{2} \zeta_n^2\Big) \Phi \Big(\zeta_n\Big) \allowdisplaybreaks \nonumber\\
        & \quad + \sqrt{\frac{2\pi}{n\rho_n}}\exp\Big( \frac{1}{2} \zeta_n^2 -\frac{1}{6}n\rho_n \theta^3\Big) \Big[ \Phi \Big( \sqrt{n\rho_n} \theta +\zeta_n\Big) - \Phi \Big( \zeta_n\Big) \Big],
    \end{align*}
    where (a) follows by using $e^{ t} - \frac{1}{2}t^2 - t -1 \geq \frac{1}{6} t^{3}$ for $t<0$; and (b) follows by using Eq. \eqref{eq: normal_cdf}. This also gives us that 
    \begin{align}\label{eq: mmn_nenneginfto0gn}
        n\epsilon_n \int_{-\infty}^0 & G_n(t) dt \leq \sqrt{2\pi}\zeta_n \exp\Big( \frac{1}{2} \zeta_n^2\Big) \Phi \Big(\zeta_n\Big) \leq \sqrt{2\pi e}\zeta_n,
    \end{align}
    where the last inequality follows by using $\zeta_n \leq 1$ whenever $n^{\alpha - \frac{1}{2}}>2c$.
    This implies that, in Super-HW regime, 
    \begin{align} \label{eq: mmn_sup_HW_ppr_n>0}
        \PP(r_n>0) \stackrel{(a)}{=}  \frac{n\epsilon_n \int_{-\infty}^0  G_n(t) dt }{1+ n\epsilon_n \int_{-\infty}^0  G_n(t) dt } \leq n\epsilon_n \int_{-\infty}^0  G_n(t) dt \leq  4e\pi c n^{-\alpha + \frac{1}{2}},
    \end{align}
    where (a) follows by using Lemma \ref{lem: mmn_steady_state_mgf_b}. This proves the probability bound presented in Eq. \eqref{eq: mmn_idleser_probbound_supHW}. Next, we provide a bound on the tail probability. We also have,
    \begin{align*}
         \int_{-\infty}^0 & G_n(t) dt \allowdisplaybreaks\nonumber\\
         &= \int_{-\infty}^0\exp\Big( - n\epsilon_n  t -n\rho_n (e^{- t} + t -1)\Big) dt\allowdisplaybreaks\nonumber\\
        & = \int_{0}^\infty \exp\Big( n\epsilon_n  t -n\rho_n (e^{ t} - t -1)\Big) dt\allowdisplaybreaks\nonumber\\
        & \geq \exp \big( \zeta_n \big) \int_{1/\sqrt{n\rho_n}}^\infty \exp\Big( -n\rho_n (e^{ t} - t -1)\Big) dt \allowdisplaybreaks\nonumber\\
        & \stackrel{(a)}{\geq} \exp \big( \zeta_n \big) \int_{1/\sqrt{n\rho_n}}^\infty \exp\Big(- \frac{1}{2}n\rho_n t^2 -n\rho_n t^3 e^{ t} \Big) dt \allowdisplaybreaks\nonumber\\
        & \geq \exp \big( \zeta_n \big) \int_{1/\sqrt{n\rho_n}}^1 \exp\Big(- \frac{1}{2}n\rho_n t^2 -e n\rho_n t^2 \Big) dt \allowdisplaybreaks\nonumber\\
        & \geq \exp \big( \zeta_n \big) \int_{1/\sqrt{n\rho_n}}^1 \exp\big(- 4 n\rho_n t^2  \big) dt \allowdisplaybreaks\nonumber\\
        & = \frac{1}{\sqrt{n\rho_n}} \exp \big( \zeta_n \big) \int_{1}^{\sqrt{n\rho_n}} \exp\big(- 4 t^2  \big) dt,
    \end{align*}
    where (a) follows by using $e^t - t -1 \leq \frac{1}{2} t^2 + t^3 e^t$ for all $t \geq 0$. This gives us that 
    \begin{align}\label{eq: mmn_tildegn_neginfto0_lowerbound}
        \int_{-\infty}^0 & G_n(t) dt \geq \frac{1}{\tilde \kappa_1} \times \frac{1}{\sqrt{n\rho_n}} \exp \big( \zeta_n \big),
    \end{align}
    where $\tilde \kappa_1 := \big(\int_{1}^{2} \exp\big(- 4 t^2  \big)dt\big)^{-1}$. Further, we have that for any $\theta\in (0,\sqrt{n\rho_n})$,
    \begin{align}
    \label{eq: mmn_tildegn_0totheta}
        \int_{0}^{\theta/\sqrt{n\rho_n}} & G_n(t) dt \allowdisplaybreaks\nonumber\\
        &=\int_{0}^{\theta/\sqrt{n\rho_n}} \exp\Big( - n\epsilon_n  t -n\rho_n (e^{- t} + t -1)\Big) dt \allowdisplaybreaks\nonumber \\
        &\leq \int_{0}^{\theta/\sqrt{n\rho_n}} \exp\Big(-\frac{1}{2}n\rho_n t^2   -n\rho_n \big(e^{- t} -\frac{1}{2} t^2+ t -1 \big)\Big) dt\allowdisplaybreaks\nonumber\\
        &\stackrel{(a)}{\leq} \int_{0}^{\theta/\sqrt{n\rho_n}} \exp\Big(-\frac{1}{2}n\rho_n t^2   +\frac{1}{6}n\rho_n t^3\Big) dt\allowdisplaybreaks\nonumber\\
        &\stackrel{(b)}{\leq} \frac{1}{\sqrt{n\rho_n}} \int_{0}^{\theta} \exp\Big(-\frac{1}{2} t^2   +\frac{1}{6\sqrt{n\rho_n}} t^3\Big) dt\allowdisplaybreaks\nonumber\\
        & \stackrel{(c)}{\leq} \frac{1}{\sqrt{n\rho_n}} \int_{0}^{\theta} \exp\Big(-\frac{1}{3} t^2 \Big) dt\allowdisplaybreaks\nonumber\\
        & \leq \frac{1}{\sqrt{n\rho_n}} \int_{0}^{\infty} \exp\Big(-\frac{1}{3} t^2 \Big) dt\allowdisplaybreaks\nonumber\\
        & \stackrel{(d)}{\leq} \frac{\sqrt{3\pi}}{2\sqrt{n\rho_n}},
    \end{align}
    where (a) follows by using $e^{- t} -\frac{1}{2} t^2+ t -1 \geq -\frac{1}{6} t^3$; (b) follows by substituting $t \rightarrow \frac{1}{\sqrt{n\rho_n}} t$; (c) follows for $\theta < \sqrt{n\rho_n}$; and finally, (d) follows by using Eq. \eqref{eq: normal_cdf}.
    Now, by using Lemma \ref{lem: mmn_steady_state_mgf_a} and substituting $\tilde r_n = r_n -n\epsilon_n$, for $\theta\in (0,\sqrt{n\rho_n})$,  we have 
    \begin{align}\label{eq: mmn_tildern_mgf}
        \EE& \Big[ \exp \Big( \frac{1}{\sqrt{n\rho_n}} \theta \tilde r_n \Big)\Big| r_n>0\Big]\allowdisplaybreaks\nonumber\\
        &= \exp \Big(n\rho_n (e^{-\frac{\theta}{\sqrt{n\rho_n}}} + \frac{\theta}{\sqrt{n\rho_n}} -1)\Big) \Big[ 1+ \left(\int_{-\infty}^0 G_n(t) dt \right)^{-1}\int_{0}^{\frac{\theta}{\sqrt{n\rho_n}}} G_n(t) dt  \Big]\allowdisplaybreaks\nonumber\\
        & \stackrel{(a)}{\leq} \exp \Big(\frac{\theta^2}{2}\Big) \Big[ 1+ \tilde \kappa_1 \frac{\sqrt{3\pi}}{2} \exp \big( -\zeta_n \big)\Big],
    \end{align}
    where (a) follows by using $e^{-x} +x -1\leq \frac{x^2}{2}$ for any $x\geq 0$ and Eq. \eqref{eq: mmn_tildegn_neginfto0_lowerbound} and \eqref{eq: mmn_tildegn_0totheta}. Note that the result in Eq. \eqref{eq: mmn_tildern_mgf} is valid for all three regimes. 
    
    Next, by using Markov's inequality, for any $\theta>0$ and $x\geq 0$,
    \begin{align*}
        \PP\Big(\frac{1}{\sqrt{n\rho_n}}  r_n > x + \zeta_n \Big|  r_n>0 \Big)\leq \Big[ 1+\tilde \kappa_1 \frac{\sqrt{3\pi}}{2} \Big] \exp \Big(-\theta x+\frac{\theta^2}{2}\Big).
    \end{align*}
    As $r_n\leq n$, we only consider the range of $x$ for which $x+ \zeta_n \leq \frac{1}{\sqrt{n\rho_n}} n $, which gives $x< \sqrt{n\rho_n}$. Then, by substituting $\theta=x< \sqrt{n\rho_n}$ in the previous equation, we have 
    \begin{align}\label{eq: mmn_idleser_markobineq}
        \PP\Big(\frac{1}{\sqrt{n\rho_n}} r_n > x + \zeta_n \Big|  r_n>0  \Big)\leq \Big[ 1+\tilde \kappa_1 \frac{\sqrt{3\pi}}{2} \Big]  e^{-\frac{x^2}{2}}.
    \end{align}
    Now, result in Theorem \ref{thm: mmn_idleservers_supHW} follows by using Eq. \eqref{eq: mmn_sup_HW_ppr_n>0} and picking $$\kappa_1 : = 4e\pi \Big[ 1+\tilde \kappa_1 \frac{\sqrt{3\pi}}{2} \Big].$$
    This completes the proof of Theorem \ref{thm: mmn_idleservers_supHW}. 

    Next, we consider the HW regime. In this case, we have $\epsilon_n = cn^{-\frac{1}{2}}$, and 
    \begin{align*}
    \zeta_n =  \frac{c }{\sqrt{ 1-cn^{-\frac{1}{2}}}}. 
    \end{align*}
    Further,
\begin{align*}
        & \int_{-\infty}^0G_n (t) dt - \sqrt{\frac{2\pi}{n\rho_n}} \exp\Big( \frac{1}{2} \zeta^2 \Big) \Phi \big( \zeta_n\big) \allowdisplaybreaks \nonumber\\
        & = \int_{0}^\infty  \exp\Big( n\epsilon_n  t -n\rho_n (e^{ t} - t -1)\Big) dt - \int_{0}^\infty  \exp\Big(  n\epsilon_n  t -\frac{1}{2} n\rho_n t^2\Big) dt \allowdisplaybreaks \nonumber\\
        &= \int_{0}^\infty \exp\Big( n\epsilon_n  t -\frac{1}{2} n\rho_n t^2\Big)  \Big( \exp\Big( -n\rho_n \big(e^{ t} -\frac{1}{2}t^2 - t -1\big)\Big) -1\Big) dt\\
        &\stackrel{(a)}{\geq } -n\rho_n \int_{0}^\infty \big(e^{ t} -\frac{1}{2}t^2 - t -1\big) \exp\Big( n\epsilon_n  t -\frac{1}{2} n\rho_n t^2\Big) dt\\
        & \stackrel{(b)}{\geq }-n\rho_n \int_{0}^\infty t^3e^{ t} \exp\Big(  n\epsilon_n  t -\frac{1}{2} n\rho_n t^2\Big) dt\\
        &\stackrel{(c)}{\geq } -\frac{1}{n\rho_n} \int_{0}^\infty t^3 \exp\Big( (1+2c) t -\frac{1}{2}t^2\Big) dt\\
        & \stackrel{(d)}{\geq } -\frac{\tilde \kappa_2 }{n\rho_n},
    \end{align*}
    where (a) follows by using $e^{x}-1\geq x$ for all $x\in \RR$; (b) follows by using $e^{ t} -\frac{1}{2}t^2 - t -1 \leq t^3 e^{t}$ for all $t\geq 0$; (c) follows by substituting $t \rightarrow \frac{1}{\sqrt{n\rho_n}} t$ and using $ \frac{1}{\sqrt{n\rho_n}}  + \zeta_n \leq  (1+2c)$ when the system is in HW regime and $n \geq 4c^2$; and (d) follows by taking $\tilde \kappa_2 := \int_{0}^\infty t^3 \exp\Big( (1+2c) t -\frac{1}{2}t^2\Big) dt $. This gives us that
    \begin{align*}
        n\epsilon_n\int_{-\infty}^0G_n (t) dt \geq  \sqrt{2\pi} \zeta_n \exp\Big( \frac{1}{2} \zeta^2 \Big) \Phi \left( \zeta_n\right) - \tilde \kappa_2 \frac{\epsilon_n}{\rho_n}.
    \end{align*}
    Combining the above inequality with Eq. \eqref{eq: mmn_nenneginfto0gn}, we have 
    \begin{align*}
        - \tilde \kappa_2 \frac{\epsilon_n}{\rho_n} \leq n\epsilon_n \int_{-\infty}^0G_n (t) dt - \sqrt{2\pi} \zeta_n  \exp\Big( \frac{1}{2} \zeta^2 \Big) \Phi \big( \zeta_n\big)  \leq 0.
    \end{align*}
    Thus, we have 
    \begin{align*}
         \lim_{n\rightarrow \infty} n\epsilon_n \int_{-\infty}^0 & G_n(t) dt =\lim_{n\rightarrow \infty}  \sqrt{2\pi}\zeta_n \exp\Big( \frac{1}{2} \zeta_n^2\Big) \Phi \left(\zeta_n\right) = \sqrt{2\pi}c \exp\left( \frac{c^2}{2}\right) \Phi \left(c\right),
    \end{align*}
    and so,
    \begin{align*}
        \lim_{n\rightarrow \infty} \PP(r_n>0) = \lim_{n\rightarrow \infty} \frac{n\epsilon_n \int_{-\infty}^0  G_n(t) dt }{1+ n\epsilon_n \int_{-\infty}^0  G_n(t) dt } = \frac{\sqrt{2\pi}c \exp\left( \frac{c^2}{2}\right) \Phi \left(c\right)}{1+\sqrt{2\pi}c \exp\left( \frac{c^2}{2}\right) \Phi \left(c\right)}.
    \end{align*}
    Next, by using Eq. \eqref{eq: mmn_tildern_mgf} and Markov's inequality, we have 
    \begin{align*}
        \PP\Big(\frac{1}{\sqrt{n\rho_n}} r_n > x + \zeta_n  \Big) & \leq \PP\Big(\frac{1}{\sqrt{n\rho_n}} r_n > x + \zeta_n \Big|  r_n>0  \Big)\\
        &\leq \Big[ 1+\tilde \kappa_1 \frac{\sqrt{3\pi}}{2} \Big]  e^{-\frac{x^2}{2}},
    \end{align*}
    where the last inequality follows by same calculations as in Eq. \eqref{eq: mmn_idleser_markobineq}.
    Now, result in Theorem \ref{thm: mmn_idleservers_HW} follows by picking $$\kappa_2 : = \Big[ 1+\tilde \kappa_1 \frac{\sqrt{3\pi}}{2} \Big].$$
    This completes the proof of Theorem \ref{thm: mmn_idleservers_HW}. 
    
    Finally, we consider the Sub-HW regime, i.e., $\epsilon_n = cn^{-\alpha}$ where $\alpha\in \Big(0,\frac{1}{2} \Big)$. In this case,
    \begin{align*}
        \zeta_n =  \Big( \frac{c^2 n^{1-2\alpha}}{1-cn^{-\alpha}} \Big)^{\frac{1}{2}} \geq cn^{\frac{1}{2}-\alpha}.
    \end{align*}
    As such $\zeta \rightarrow \infty$ as $n\rightarrow \infty$.
    From Eq. \eqref{eq: mmn_tildegn_neginfto0_lowerbound}, we have that
    \begin{align*}
        n\epsilon_n\int_{-\infty}^0 & G_n(t) dt \geq \frac{1}{\tilde \kappa_1}  \zeta_n \exp \big( \zeta_n \big).
    \end{align*}
     Then,
    \begin{align*}
        1- \PP(r_n >0) = \frac{1}{1+ n \epsilon_n\int_{-\infty}^0  G_n(t) dt } \leq \tilde \kappa_1  \zeta_n^{-1} \exp \big( -\zeta_n \big) \leq  \frac{\tilde \kappa_1}{c}  n^{\alpha - \frac{1}{2}}e^{-cn^{\frac{1}{2}-\alpha}}.
    \end{align*}
    Thus, for $\alpha \in \big( 0, \frac{1}{2}\big)$, by choosing $\kappa_3 = \tilde \kappa_1$, we have 
    \begin{align*}
        \PP(r_n >0) \geq 1-\frac{\kappa_3}{c}  n^{\alpha - \frac{1}{2}}e^{-cn^{\frac{1}{2}-\alpha}}, && \lim_{n\rightarrow \infty} \PP(r_n >0) =1. 
    \end{align*}
    Next, by using Eq. \eqref{eq: mmn_tildern_mgf} and Markov's inequality, we have 
    \begin{align*}
        \PP\Big(\frac{1}{\sqrt{n\rho_n}} \tilde r_n > x   \Big) & \leq \PP\Big(\frac{1}{\sqrt{n\rho_n}} r_n > x \Big|  r_n>0  \Big)\\
        &\leq \Big[ 1+\tilde \kappa_1 \frac{\sqrt{3\pi}}{2} e^{-cn^{\frac{1}{2}-\alpha}} \Big]  e^{-\frac{x^2}{2}}.
    \end{align*}
    Further, for $\theta> 0$, we have $\left(\int_{-\infty}^0 G_n(t) dt \right)^{-1}\int_{\infty}^{-\theta} G_n(t) dt \leq 1$. Then,
    \begin{align*}
        \EE \Big[ \exp \Big( -\frac{1}{\sqrt{n\rho_n}} \theta \tilde r_n \Big)\Big| r_n>0\Big] &\leq \exp \Big(n\rho_n \big(e^{\frac{\theta}{\sqrt{n\rho_n}}} - \frac{\theta}{\sqrt{n\rho_n}} -1\big)\Big)\\
        &\leq \exp \Big( \frac{1}{2}\theta^2 + \frac{\theta^3}{\sqrt{n\rho_n}}e^{\frac{\theta}{\sqrt{n\rho_n}}}  \Big),
    \end{align*}
    This gives us that 
    \begin{align*}
        \PP\Big(-\frac{1}{\sqrt{n\rho_n}} \tilde r_n > x\Big| r_n >0 \Big) \leq  e^{-\theta x}\exp \Big( \frac{1}{2}\theta^2 + \frac{\theta^3}{\sqrt{n\rho_n}}e^{\frac{\theta}{\sqrt{n\rho_n}}} \Big).
    \end{align*}
    By substituting $\theta =x$ in the above equation, we get that for any $x\geq 0$,
    \begin{align*}
         \PP\Big(-\frac{1}{\sqrt{n\rho_n}} \tilde r_n > x\Big| r_n >0 \Big) \leq \exp \Big( -\frac{1}{2}x^2 + \frac{x^3}{\sqrt{n\rho_n}}e^{\frac{x}{\sqrt{n\rho_n}}} \Big).
    \end{align*}
    Note that $-\tilde r_n \leq n \epsilon_n$ as $r_n \geq 0$. Hence, we consider the range of $x$ such that $\sqrt{n\rho_n} x < n\epsilon_n$, or $x< \zeta_n$. For any $x\geq \zeta_n$, the above inequality holds trivially. Then, for $x< \zeta_n$, we have $\frac{x}{\sqrt{n\rho_n}} \leq \frac{\epsilon_n}{\rho_n} \leq 2 \epsilon_n \leq 1 $, assuming $n^{\alpha} > 2c$ (which holds when $n\geq 4c^2$). Then, 
    \begin{align*}
        \PP\Big(-\frac{1}{\sqrt{n\rho_n}} \tilde r_n > x\Big| r_n >0 \Big) \leq \exp \Big( -\frac{1}{2}x^2 + 2e\epsilon_n x^2\Big).
    \end{align*}
    This completes the proof. 

\end{proof}

\end{appendix}

\end{document}